\documentclass[11pt,reqno]{amsart}
\usepackage{amssymb,mathrsfs,color}
\usepackage{enumerate}
\usepackage{graphicx}
\usepackage{xcolor} % A package to add color.
\usepackage{geometry}
\usepackage{amsfonts,amssymb,amsmath}
\usepackage{slashed}
\usepackage{ mathrsfs }
\usepackage{txfonts}
\usepackage[titletoc,title]{appendix}
\usepackage{wrapfig}
\usepackage{amsmath}

\usepackage{cancel}
\usepackage{cite}
\usepackage{bm}
\usepackage{nth}
\usepackage[T1]{fontenc}
\usepackage{hyperref}

\usepackage{marginnote}
\usepackage{cancel}
\usepackage{pgfplots}
\usepackage{upgreek}
\usepackage{slashed}

\definecolor{green}{rgb}{0,0.8,0} % Redefines the color green.

\definecolor{deepgreen}{cmyk}{1,0,1,0.5}

\newcommand{\Del}[1]{}

\numberwithin{equation}{section}

\newtheorem{theorem}{Theorem}[section]
\newtheorem{corollary}[theorem]{Corollary}%[section]
\newtheorem{lemma}[theorem]{Lemma}%[section]
\newtheorem{proposition}[theorem]{Proposition}%[section]
\newtheorem{remark}[theorem]{Remark}%[section]
\renewcommand{\div}{\mathrm{div}\,}

\renewcommand{\hbar}{{\underline h}}

\newcommand{\ybar}{{\underline y}}

\newcommand{\bbR}{\mathbb R}

\newcommand{\calB}{\mathcal B}

\newcommand{\calE}{\mathcal E}

\newcommand{\calN}{\mathcal N}

\newcommand{\tilR}{{\tilde{R}}}

\newcommand{\scE}{{\mathscr{E}}}

\makeatletter
\newsavebox{\@brx}
\newcommand{\llangle}[1][]{\savebox{\@brx}{\(\m@th{#1\langle}\)}%
  \mathopen{\copy\@brx\kern-0.5\wd\@brx\usebox{\@brx}}}
\newcommand{\rrangle}[1][]{\savebox{\@brx}{\(\m@th{#1\rangle}\)}%
  \mathclose{\copy\@brx\kern-0.5\wd\@brx\usebox{\@brx}}}
\makeatother

\usepackage{nth}
\begin{document}
	
	\title[Stability and instability hard phase]{On stability analysis for steady states of the free boundary hard phase model in general relativity}
	\author{Zeming Hao
		\and Shuang Miao}
	\maketitle
	\begin{abstract}
		The hard phase model describes a relativistic barotropic fluid with sound speed equal to the speed of light. In the framework of general relativity, the motion of the fluid is coupled to the Einstein equations which describe the structure of the underlying spacetime. This model with free boundary admits a $1$-parameter family of steady states with spherical symmetry. In this work, for perturbations within spherical symmetry, we study the stability and instability of this family. We prove that the linearized operator around steady states with large central densities admits a growing mode, while such growing modes do not exist for steady states with small central densities. Based on the linear analysis, we further demonstrate a dynamical nonlinear instability for steady states with large central densities. The proof relies on a spectral analysis of the linearized operator and an a priori estimate on the full nonlinear free boundary problem. 
	\end{abstract}

	\section{Introduction}
	\subsection{General setup}
		General relativity presents us with a unified theory about space, time and gravitation, which is is considered on a $4$-dimensional spacetime manifold \(M\) with the metric \(g\) of signature $(-1,+1,+1,+1)$. An ideal fluid is defined to be matter whose energy-momentum tensor satisfies
	\begin{align}\label{energy tensor}
	T_{\alpha\beta}=(\rho+p)u_{\alpha}u_{\beta}+pg_{\alpha\beta},
	\end{align}
	where \(\rho\) is mass-energy density, \(p\) is pressure and \(u\) is a future directed unit timelike vector field. Einstein equations give the relation between the space-time  geometry and the matter represented by the tensor \(T\)
	\begin{align}\label{einstein eq}
	\bar{R}_{\alpha\beta}-\frac{1}{2}g_{\alpha\beta}\bar{R}=8\pi T_{\alpha\beta},
	\end{align}
	where \(\bar{R}_{\alpha\beta}\) and \(\bar{R}\) are respectively the Ricci curvature and scalar curvature of \(g\). The Bianchi identity applied to the field equations \(\eqref{einstein eq}\) yields conservation laws
	\begin{align}\label{conservation law}
	\nabla^{\alpha}T_{\alpha\beta}=0,
	\end{align}
	where \(\nabla\) denotes the Levi-Civita connection of $g$. When the fluid is isentropic, the component of \(\eqref{conservation law}\) along \(u\) is
	\begin{align}\label{mass conservation}
	u^{\alpha}\nabla_{\alpha}\rho+(\rho+p)\nabla^{\alpha}u_{\alpha}=0.
	\end{align}
	the projection of \(\eqref{conservation law}\) with respect to \(u\) is
	\begin{align}\label{momentum conservation}
	(\rho+p)u^{\alpha}\nabla_{\alpha}u_{\beta}+(g_{\alpha\beta}+u_{\alpha}u_{\beta})\nabla^{\alpha}p=0.
	\end{align}
	Assume the fluid is initially supported on a compact domain \(\mathcal{B}_{0}\subset\mathbb{R}^{3}\). The domain \(\calB(t)\) occupied by fluid does not have a fixed shape and as the system evolves \(\calB(t)\) changes in time. Therefore it leads to a free boundary problem subject to the following boundary conditions:
	\begin{align}\label{boundary condition fluid}
	p=0\quad \textrm{on}\quad \partial\mathcal{B}(t),\quad\quad
	u|_{\partial\calB}\in  \,T(t,\partial\mathcal{B}(t)).
	\end{align}
	In this work we consider an equation of state of the following form
	\begin{align}\label{state eq}
	p=\rho-\rho_{0},
	\end{align}
where $\rho_{0}>0$ is a constant. In this work we choose appropriate units so that $\rho_{0}=1$. The fluid described by \eqref{state eq} is called \textit{hard phase} fluid, which is an idealized model for the physical situation where, during the gravitational collapse of the degenerate core of a massive star, the mass-energy density exceeds the nuclear saturation density. See \cite{Ch-video,Ch-hp1, F-P, Lich-book, Rez-book, Walecka, Zeldovich, MS}. The well-posedness of the free boundary problem \eqref{energy tensor}-\eqref{einstein eq}-\eqref{conservation law}-
	 \eqref{boundary condition fluid}-\eqref{state eq} is proved in \cite{MS}\footnote{See also a version \cite{MSW1} in a fixed Minkowski spacetime and \cite{Oliynyk3, GinLin1} for similar models in a fixed spacetime.}.
	 
	\subsection{Statement of the main result}
	 	In this work we assume the underlying spacetime is spherically symmetric, i.e. the group $SO(3)$ acts as an isometry group on $M$ \footnote{See \cite{Ch-hp1} for more detailed discussions on the geometry of spacetime with spherical symmetry.}. In Schwarzschild coordinates $(t,r,\theta,\varphi)=(x^{0},x^{1},x^{2},x^{3})$, the spacetime metric $g$ can be written in the form
	\begin{align}\label{metric}
	g=-e^{2\mu(t,r)}dt^{2}+e^{2\lambda(t,r)}dr^{2}+r^{2}(d\theta^{2}+\sin^{2}\theta d\varphi^{2}).
	\end{align}
	It should be noted that the fluid occupies a bounded region, but the gravitational field exists in whole space.
	We consider space-time to be asymptotically flat, therefore the boundary condition hold:
	\begin{align}\label{boundary condition field}
	\lim_{r\to\infty}\lambda(t,r)=\lim_{r\to\infty}\mu(t,r)=0.
	\end{align}
	As radius $r\rightarrow0_{+}$, the mass of the portion of fluid occupying the ball centered at the center of symmetry with radius $r$ vanishes faster than $r$. This implies the following condition on $\lambda$ at the center: 
	\begin{align}\label{center condition}
	\lambda(t,0)=0.
	\end{align}
	 The nonvanishing connection coefficients of the metric \(\eqref{metric}\) are \(\Gamma^{1}_{AA},\Gamma^{A}_{BC}\) (\(A,B,C\in\{2,3\}\)) and
	\begin{align*}
	\Gamma^{0}_{00}=\dot{\mu},\quad \Gamma^{0}_{10}=\mu',\quad  \Gamma^{0}_{11}=e^{-2\mu}e^{2\lambda}\dot{\lambda},\quad  \Gamma^{1}_{00}=e^{-2\lambda}e^{2\mu}\mu',\quad \Gamma^{1}_{10}=\dot{\lambda},\quad \Gamma^{1}_{11}=\lambda',\quad \Gamma^{2}_{12}=\Gamma^{3}_{13}=r^{-1},
	\end{align*}
	where \(\dot{}\) and \('\) denote the derivatives with respect to \(t\) and \(r\) respectively. The spherically symmetric fluid quantities \(\rho=\rho(t,r),p=p(t,r),u=u(t,r)\) are scalar functions, and the four velocity is \(\vec{u}=(u^{0},u,0,0)\) where
	\begin{align*}
	u^{0}=e^{-\mu}\sqrt{1+e^{2\lambda}u^{2}}=:e^{-\mu}\left\langle u\right\rangle .
	\end{align*}
	A straightforward calculation gives the following Einstein-Euler system under spherical symmetry
	\begin{align}\label{1field eq00}
	&e^{-2\lambda}(2r\lambda'-1)+1=8\pi r^{2}\left(\rho+e^{2\lambda}(\rho+p)u^{2} \right),\\ \label{1field eq11}
	&e^{-2\lambda}(2r\mu'+1)-1=8\pi r^{2}\left(p+e^{2\lambda}(\rho+p)u^{2} \right),\\ \label{1field eq01}
	&\dot{\lambda}=-4\pi r e^{\mu+2\lambda}\left\langle u\right\rangle u(\rho+p),\\ \label{field eq k}
	&e^{-2\lambda}\left(\mu''+(\mu'-\lambda')(\mu'+\frac{1}{r}) \right) -e^{-2\mu}\left( \ddot{\lambda}+\dot{\lambda}(\dot{\lambda}-\dot{\mu})\right) =8\pi p,\\ \label{1mass eq SS}
	&\dot{\rho}+e^{\mu}\frac{u}{\left\langle u\right\rangle}\rho'+(\rho+p)\left[\dot{\lambda}+e^{\mu}\frac{u}{\left\langle u\right\rangle}\left( \lambda'+\mu'+\frac{2}{r}\right)+e^{\mu}\frac{u'}{\left\langle u\right\rangle} +e^{2\lambda}\frac{u}{\left\langle u\right\rangle} \frac{\dot{\lambda}u+\dot{u}}{\left\langle u\right\rangle}   \right] =0,\\ \label{1momentum eq SS}
	&(\rho+p)\left[ e^{2\lambda}(\dot{u}+2\dot{\lambda}u)+e^{\mu}\left\langle u\right\rangle\mu'+e^{\mu+2\lambda}\frac{u}{\left\langle u\right\rangle}(u'+\lambda'u)\right] +e^{\mu}\left\langle u\right\rangle p'+e^{2\lambda}u \dot{p}=0.
	\end{align}
Outside the fluid domain, we simply set $p=\rho=0$.
    As we shall show, the above system \eqref{1field eq00}-\eqref{1momentum eq SS} admits a family of steady states which satisfy the following static equations:
    \begin{align}\label{1steady field eq00}
    &e^{-2\lambda}(2r\lambda'-1)+1=8\pi r^{2}\rho,\\ \label{1steady field eq11}
    &e^{-2\lambda}(2r\mu'+1)-1=8\pi r^{2}p,\\ \label{steady field eq k}
    &e^{-2\lambda}\left(\mu''+(\mu'-\lambda')(\mu'+\frac{1}{r}) \right) =8\pi p,\\ \label{1steady momentum eq}
    &(\rho+p)\mu'+p'=0.
    \end{align}
	We consider the density at the center of symmetry called \textit{central density}. Then given a value of central density, we shall show in Section \ref{sec2} that there is a solution to the above static system.
	Therefore we obtain a 1-parameter family \((\rho_{\kappa},p_{\kappa},\lambda_{\kappa},\mu_{\kappa})\), where \(\kappa\) is called \textit{central redshift} defined in Section \ref{sec2}, which is used to parameterize the steady state solutions. In this paper we demonstrate a stability analysis for the hard phase model around the above  equilibrium states:
	\begin{theorem}\label{1.1.}
		Let \((\rho_{\kappa},p_{\kappa},\lambda_{\kappa},\mu_{\kappa})\) be a one-parameter family of spherically symmetric steady states to the Einstein-Euler system with equation of state \eqref{state eq} and \(\kappa\) be the central redshift of the corresponding steady state. Then,
		\begin{enumerate}
			\item For \(\kappa\) sufficiently large, the associated steady state is linearly unstable in
			the sense that its linearized system possesses an exponentially growing mode, while for a steady state with \(\kappa\) sufficiently small, there is no such growing mode, which means the steady state is linearly stable. \label{main result1}
			\item For the full nonlinear system \eqref{1field eq00}-\eqref{1momentum eq SS} and \(\kappa\) sufficiently large,  no matter how small the amplitude of the initial perturbation around the steady state is, we can find a solution such that the corresponding energy escapes at a later time: the associated steady state is unstable. More precisely, the nonlinear instability is driven by a linear growing mode.\label{main result2}
		\end{enumerate}
	\end{theorem}
\begin{remark}\label{rmk: general model}
	In this work we are interested in the hard phase fluid whose equation of state is given by \eqref{state eq}. In fact our argument can be generalized (without any essential modifications) to more general  model with equation of state in the form 
	\begin{align}\label{general model}
	p=c^{2}_{s}(\rho-\rho_{0}),\quad 0<c_{s}^{2}\leq 1,\quad \rho_{0}>0.
	\end{align}
\end{remark}
	
	\subsection{History and related works}
	In the past few decades, much progress has been made beyond the local theory for Einstein equations. However, in the presence of isolated bodies, especially with free boundary, our understanding on long time evolution is very limited. Such a theory for isolated bodies is of central importance as it is naturally a preliminary step in any further analysis of the motion and interaction of gravitating bodies. In \cite{FoSc1}, the authors proved the existence of solutions to static system \eqref{1steady field eq00}-\eqref{1steady momentum eq}, and show that \emph{small} stars are stable in the sense that these stars lie in a local minimum of a certain mass energy functional. Moreover in \cite{FoSc1} for small stars the authors proved a uniform boundedness for an energy of linearized system. However, the spectral stability of the linearized operator does not seem to be investigated in \cite{FoSc1}.
	
	For Einstein-Euler system over the entire spacetime, in \cite{HLR, HL} the authors gave a precise description on the spectral stability of the linearized operator around a $1$-parameter family of steady states, under the assumption of spherical symmetry. There are crucial differences between our work and \cite{HLR,HL}. First, the equation of state considered in \cite{HLR,HL} does not seem to cover the hard phase model, i.e. the sound speed of the model in \cite{HLR,HL} is not always equal to the speed of light. Second, when the steady states for the model considered in \cite{HLR,HL} are compactly supported, the density $\rho$ is continuous across the boundary of the support, which is more like a ``gaseous" model. While the hard phase model is more like a ``liquid" model in the sense that the density is discontinuous across the boundary of compact support. Third, the authors in \cite{HLR,HL} use a Hamiltonian formulation to investigate the spectral stability, while for the hard phase model we use the boundary condition satisfied by $\zeta$ to reduce the operator $L$ to a Sch\"odinger type operator with a localized potential and use a more direct approach to analyze the spectrum of $L$. %\Blue{The important contribution of this method is to partially solve the open problems raised by \cite{HLR}, that is, to give the regularity of eigenfunctions corresponding to the fastest growing mode of the linearized system, which is a key step in proving nonlinear instability.} 
	See \cite{Lam} for a  similar approach on linear stability analysis for liquid Lane-Emden stars in Newtonian framework.
	
	Based on the linear analysis in this work and the a priori estimates for the full general hard phase free boundary problem established in \cite{MS}, we further prove the nonlinear instability for steady states with sufficiently large central density. See a counterpart \cite{Jang,HM} for Newtonian Lane-Emden stars.  
	
	\subsection{Main ideas for the proof}
	\subsubsection{linear stability and instability}
	The hard phase model posses a corresponding one-parameter family of compactly supported steady states with finite mass, parameterized by the value of the central redshift \(\kappa>0\). For a member of this family, given a point $p_{0}$ in the corresponding static spacetime. Then let $y=y(p_{0})>0$ be the distance between $p_{0}$ and the center of symmetry. Let $\eta(y,t)$ be the radial position of the fluid particle at time $t$ so that 
	\begin{align}\label{def Lag coord}
	\partial_{t}\eta=\frac{u}{u^{0}}\circ \eta.
	\end{align} 
	Then for a steady state we have $\eta(y,t)\equiv y$ for all $t$. Now we consider a perturbation around a steady state such that 
	\begin{align}\label{def perturbation}
	\eta(y,t)=y(1+\zeta(y,t)).
	\end{align}
	In Proposition \ref{prop: linearized operator}, we derive a linearized equation for $\zeta$ of the form (see \eqref{linear eq})
	\begin{align}\label{eq:linear pre}
	\ddot{\zeta}+L\zeta=0,
	\end{align}
	with a Robin boundary condition, where $L$ is a Sch\"odinger type operator. Therefore the stability analysis on the linearized system is reduced to the analysis on the eigenvalues of the operator \(L\). According to the functional tools introduced in Lemma \ref{4.1..}, the smallest eigenvalue of the operator \(L\) corresponds to the minimum of the functional \(\left\langle L\chi,\chi\right\rangle\), where \(\chi\) is a function of the radial variable \(y\). For \(\kappa>0\) sufficiently small, we can prove \(\left\langle L\chi,\chi\right\rangle\ge 0\) for any \(\chi\), so there is no growing mode for the linearized system \eqref{eq:linear pre} and it is spectrally stable. To prove the instability for large values of \(\kappa>0\) it is therefore natural to construct an explicit test function \(\chi_{\kappa}\) such that \(\left\langle L\chi_{\kappa},\chi_{\kappa}\right\rangle<0\). The key to the construction is a precise understanding of the limiting behavior of the sequence of steady states \((\rho_{\kappa},p_{\kappa},\lambda_{\kappa},\mu_{\kappa})\) in the singular limit \(\kappa \to \infty\). We
	show in Lemma \ref{2.9..} (see also \cite{HLR}) that in a suitably rescaled annulus around the center \(r=0\) the behavior of the steady states is asymptotic to the equilibrium point of a autonomous planar dynamical system . Therefore we can make a judicious choice of a test function \(\chi_{\kappa}\) such that \(\left\langle L\chi_{\kappa},\chi_{\kappa}\right\rangle<0\), which implies the linear instability.
	
	\subsubsection{Nonlinear instability} The passage from linear instability to nonlinear instability requires a sharp nonlinear energy estimate that allows us to control the growth of high energy norm in terms of the fastest linear growth rate \(\sqrt{-\nu_{\ast}}\) defined in Section \ref{sec4}. To obtain such an estimate, we follow the approach introduced in \cite{MSW1}. We start by deriving a quasilinear system for fluid variables from Einstein-Euler system \eqref{1field eq00}-\eqref{1momentum eq SS}. By introducing the renormalized fluid velocity field \(\vec{V}\) and the enthalpy \(\sigma\) (see \eqref{vecV}), we obtain a coupled quasilinear system for fluid velocity \(V\) and enthalpy perturbation variable \(\varepsilon\) (see \eqref{fk}, \eqref{Fk} and \eqref{Hk}):
	\begin{equation}\label{1.25}
	\square D_{\vec{V}}^{k}V=F_{k}+\frac{1}{r}F_{k-1}\quad \textrm{in}\ \calB(t),\quad\quad
	\left(D_{\vec{V}}^{2}+\frac{1}{2}aD_{n}\right)D_{\vec{V}}^{k}V=f_{k}\quad \textrm{on}\ \partial\calB(t),
	\end{equation}
	and
	\begin{equation}\label{1.26}
	\square D_{\vec{V}}^{k+1}\varepsilon=-12e^{-2\lambda}\mu'\partial_{r}D_{\vec{V}}^{k+1}\varepsilon+H_{k}+\frac{1}{r}H_{k-1}\quad \textrm{in}\ \calB(t),\quad\quad
    D_{\vec{V}}^{k+1}\varepsilon=0 \quad \textrm{on}\ \partial\calB(t),
	\end{equation}
	for any \(k\ge 0\), where \(f_{k},F_{k}\) and \(H_{k}\) are described in Lemma \ref{2.3.}-\ref{2.5.}. In \cite{MS} a similar quasilinear system was shown to be hyperbolic type, and a local-wellposedness result of free boundary hard phase fluids was obtained. For the fluid velocity equations \eqref{1.25}, We multiply the boundary equation and interior equation by \(\frac{1}{a}(D_{\vec{V}}^{k+1}V),D_{\vec{V}}^{k+1}V\) respectively and integrate. By observing the signs of
	the boundary terms, we obtain the following energy in Lemma \ref{3.8.}
	\begin{align*}
	\int_{\calB(t)}|\partial_{t,r}D_{\vec{V}}^{k}V|^{2}dx+\int_{\partial\calB(t)}|D_{\vec{V}}^{k+1}V|^{2}dS.
	\end{align*}
	For the wave equation for \(D_{\vec{V}}^{k+1}\varepsilon\) with Dirichlet boundary conditions, we choose a suitable multiplier consisting of an appropriate linear combination of \(\vec{V}\) and the normal \(n\) to \(\partial\calB(t)\), and apply integration by parts in Lemma \ref{3.9.} and \ref{3.10.}. The energy functional for \(\eqref{1.26}\) controls
	\begin{align*}
	\int_{\calB(t)}|\partial_{t,r}D_{\vec{V}}^{k+1}\varepsilon|^{2}dx+\int_{0}^{t}\int_{\partial\calB(\tau)}|\partial_{t,r}D_{\vec{V}}^{k+1}\varepsilon|^{2}dSd\tau.
	\end{align*}
	Then we apply elliptic estimates to \eqref{1.25} and \eqref{1.26} to control high energy norm \(\scE_{l}(t)\) defined in Section \ref{sec6} (see details in Proposition \ref{3.1.}). To close the energy estimate and prove nonlinear instability, we consider the energy estimate in Lagrangian coordinate system that is equivalent to the coordinate system $(t,r,\theta,\varphi)$ under the assumptions \eqref{bootstrap}. Applying Sobolev interpolation inequality, we get the following sharp energy estimate
	\begin{align*}
	\bar{\scE_{l}}(t)\le C_{0}\bar{\scE_{l}}(0)+\int_{0}^{t}\varrho\bar{\scE_{l}}(s)+C_{2}\bar{\scE_{l}}^{\frac{3}{2}}(s)+C_{\varrho}\Arrowvert\bar{V},\bar{\varepsilon}\Arrowvert_{L^{2}(\Omega)}ds,
	\end{align*}
	for small enough \(\varrho>0\) and constants \(C_{0}, C_{2}, C_{\varrho}>0\). By a standard bootstrap argument (see \cite{GHS}), we have \(\Arrowvert\bar{V},\bar{\varepsilon}\Arrowvert_{L^{2}(\Omega)}\lesssim \frac{1}{2}\delta e^{\sqrt{-\nu_{\ast}}t}\), where \(\delta\) is the magnitude of the initial perturbation. Choosing the appropriate initial perturbation \(\delta\bar{V}_{0},\delta\bar{\varepsilon}_{0}\)
	, we have \(\Arrowvert e^{tL}(\delta\bar{V}_{0},\delta\bar{\varepsilon}_{0})\Arrowvert_{L^{2}(\Omega)}\sim \delta e^{\sqrt{-\nu_{\ast}}t}\) due to linear instability. This means linear solution dominate the nonlinear dominate correction. Therefore at the escape time \(t=T^{\delta}\) defined by \eqref{T delta}, we shall see
	\begin{align*}
	\Arrowvert\bar{V}(T^{\delta}),\bar{\varepsilon}(T^{\delta})\Arrowvert_{L^{2}(\Omega)}\gtrsim \frac{1}{2}\delta e^{\sqrt{-\nu_{\ast}}T^{\delta}}=\frac{1}{2}\theta_{0},
	\end{align*}
	where \(\theta_{0}\) is independent of \(\delta\), and instability happens (see details in Section \ref{sec7}).
	\subsection{Outline of the paper}
	In Section \ref{sec2} we prove the existence of a 1-parameter family of steady states. In Section \ref{sec3} we derive the linearized equation for the perturbation \(\zeta\) and investigate the structure of the linearized operator. In Section \ref{sec4} we complete the proof of linear stability and instability, which proves the part \eqref{main result1} of Theorem \ref{1.1.}. In Section \ref{sec5} we derive the quasilinear equations for renormalized fluid velocity \(V\) and enthalpy perturbation variable \(\varepsilon\). In Section \ref{sec6} we prove a sharp nonlinear energy estimate (see Proposition \ref{3.1.}). Finally in Section \ref{sec7} we close all bootstrap assumption and prove the nonlinear instability, which completes the proof of the part \eqref{main result2} of Theorem \ref{1.1.}.
	\subsection*{Acknowledgment}
This work was supported by National Key R \& D Program of China 2021YFA1001700,  NSFC grants 12071360, 12221001, \& 12326344.

		\section{Steady states}\label{sec2}
For a steady state, all the physical quantities are time-independent and the spatial components of the velocity field necessarily vanishes. Notice that equations \(\eqref{1field eq01}\) and \(\eqref{1mass eq SS}\) are satisfied identically. The Einstein-Euler system becomes \eqref{1steady field eq00}-\eqref{1steady momentum eq}. We reduce the above equations as an ODE. Define
	\begin{align*}
	Q(\rho):=\int_{1}^{\rho}\frac{p'(s)}{s+p(s)}ds, \ \rho\ge 1.
	\end{align*}
	Then \(\eqref{1steady momentum eq}\) can be written as
	\begin{align*}
	\frac{d}{dr}(Q(\rho)+\mu)=0, \ \Longrightarrow \ Q(\rho(r))+\mu(r)=const.
	\end{align*}
	We introduce \(\ybar(r)=const-\mu(r)\) and find that \(\rho\) is given in terms of \(\ybar\)
	\begin{align}\label{rho y}
	\rho=g(\ybar):=\left\{
	\begin{aligned}
	Q^{-1}(\ybar)&, \ \ybar\ge0,\\
	0\quad &,\ \ybar<0.
	\end{aligned}
	\right.
	\end{align}
	Taking into account the equation of state \(\eqref{state eq}\) it follows that
	\begin{align}\label{p y}
	p=h(\ybar):=g(\ybar)-1.
	\end{align}
	Then we eliminate the metric coefficient \(\lambda\) in the system. By integrating the field equation \(\eqref{1steady field eq00}\) and using the center condition \(\eqref{center condition}\), we get
	\begin{align}\label{lamda eliminate}
	e^{-2\lambda(r)}=1-\frac{2m(r)}{r},
	\end{align}
	where the mass function \(m\) is defined by
	\begin{align}\label{m def}
	m(r)=4\pi\int_{0}^{r}s^{2}\rho(s)ds=4\pi\int_{0}^{r}s^{2}g(y(s))ds.
	\end{align}
	Finally, using \(\eqref{lamda eliminate}\) to substitute for the term \(e^{-2\lambda}\) together with \(\eqref{1steady field eq11}\), we have
	\begin{align}\label{ode}
	\ybar^{\prime}(r)=-\frac{1}{1-2m(r)/r}\left( \frac{m(r)}{r^{2}}+4\pi rp(r)\right) ,
	\end{align}
	where \(m\) and \(p\) are given in terms of \(\ybar\) by \(\eqref{p y}\) and \(\eqref{m def}\). For any central value
	\begin{align}\label{central value}
	\ybar(0)=\kappa>0,
	\end{align}
	we can solve the ODE \(\eqref{ode}\) to obtain the properties of the steady-state solution which are recorded in the following several lemmas.
	\begin{lemma}\label{2.1..}
		Let \(m(r)\) and \(p(r)\) be defined by \(\eqref{p y}\) and \(\eqref{m def}\) respectively. Then there exists a unique solution \(\ybar:[0,\tilde{\delta}]\to \mathbb{R} \) for sufficiently small \(\tilde{\delta}>0\) satisfying the equation \(\eqref{ode}\) and initial condition \(\eqref{central value}\).
	\end{lemma}
	\begin{proof}
		Due to the relation \(\ybar(r)=const-\mu(r)\), we can consider the equation equivalently
		\begin{align}\label{mu ode}
		\mu'(r)=\frac{1}{1-2m(r)/r}\left( \frac{m(r)}{r^{2}}+4\pi rp(r)\right)\quad \textrm{with} \quad \mu(0)=\mu_{0}.
		\end{align}
		Defining
		\begin{align*}
		(T\mu)(r):=\mu_{0}+\int_{0}^{r}\frac{1}{1-2m(s)/s}\left( \frac{m(s)}{s^{2}}+4\pi sp(s)\right)ds.
		\end{align*} 
		we obtain the following fixed point problem for \(\mu\):
		\begin{align*}
		\mu(r)=(T\mu)(r).
		\end{align*}
		It is straightforward to verify that the set
		\begin{align*}
		M:=\left\lbrace \mu\in C([0,\delta])\ \ \rvert\ \ \mu(0)=\mu_{0},\ \mu_{0}\le\mu(r)\le\mu_{0}+1, \ \frac{2m}{r}\le\frac{1}{2},\ r\in[0,\tilde{\delta}]   \right\rbrace 
		\end{align*}
		is complete with respect to the norm \(\Arrowvert\cdot\Arrowvert_{\infty}\). We show that \(T\) is a contraction map on the set \(M\). For any \(\mu_{1},\mu_{2}\in M\), we have
		\begin{align*}
		\begin{split}
		&(T\mu_{1})(r)-(T\mu_{2})(r)\\
		=&\int_{0}^{r}\frac{1}{1-2m_{1}(s)/s}\left( \frac{m_{1}(s)}{s^{2}}+4\pi sp_{1}(s)\right)-\frac{1}{1-2m_{2}(s)/s}\left( \frac{m_{2}(s)}{s^{2}}+4\pi sp_{2}(s)\right)  ds\\
		=&\int_{0}^{r}\frac{2(m_{1}(s)-m_{2}(s))}{(1-2m_{1}(s)/s)(1-2m_{2}(s)/s)}\left( \frac{m_{1}(s)}{s^{3}}+4\pi p_{1}(s)\right)ds\\
		&+\int_{0}^{r}\frac{1}{1-2m_{2}(s)/s}\left( \frac{m_{1}(s)-m_{2}(s)}{s^{2}}+4\pi s(p_{1}(s)-p_{2}(s))\right)ds.
		\end{split}
		\end{align*}
		Then we compute
		\begin{align*}
		m_{1}(s)-m_{2}(s)
		\le 4\pi\int_{0}^{s}\sigma^{2}\max_{\ybar\ge 0}|g'(\ybar)||\ybar_{1}(\sigma)-\ybar_{2}(\sigma)|d\sigma\le C\Arrowvert\mu_{1}-\mu_{2}\Arrowvert_{\infty}s^{3},
		\end{align*}
		and
		\begin{align*}
		p_{1}(s)-p_{2}(s)\le C\Arrowvert\mu_{1}-\mu_{2}\Arrowvert_{\infty}.
		\end{align*}
		Choosing small enough \(\tilde{\delta}>0\), it is proved that \(T\) maps the set \(M\) acts as a contraction by
		\begin{align*}
		\Arrowvert T\mu_{1}-T\mu_{2}\Arrowvert_{\infty}\le C\Arrowvert\mu_{1}-\mu_{2}\Arrowvert_{\infty}\int_{0}^{r}(s^{3}+s)ds\le C(\tilde{\delta}^{4}+\tilde{\delta}^{2})\Arrowvert\mu_{1}-\mu_{2}\Arrowvert_{\infty}\le \frac{1}{2}\Arrowvert\mu_{1}-\mu_{2}\Arrowvert_{\infty}.
		\end{align*}
		Fixed point theorem gives the existence and uniqueness of the solution about \(\eqref{mu ode}\) and the proof is complete.
	\end{proof}
	In order to show that the above solutions actually extend to \(r=\infty\) we give an important relation which is known as the Tolman-Oppenheimer-Volkov equation.
	\begin{lemma}
		Let \(\lambda\), \(\mu\), \(\rho\) and \(p\) be steady-state solutions of Einstein-Euler system. Then the following identity holds:
		\begin{align}\label{tov}
		\lambda'+\mu'=4\pi r e^{2\lambda}(\rho+p).
		\end{align}
	\end{lemma}
	\begin{proof}
		It is straightforward to be proved by adding \(\eqref{1steady field eq00}\) to \(\eqref{1steady field eq11}\).
	\end{proof}
	\begin{lemma}
		Let \(m(r)\) and \(p(r)\) be defined by \(\eqref{p y}\) and \(\eqref{m def}\) respectively. Then there exists a unique solution \(\ybar: [0,\infty)\to \mathbb{R}\) satisfying the equation \(\eqref{ode}\) and initial condition \(\eqref{central value}\).
	\end{lemma}
	\begin{proof}
		As discussed in Lemma \(\ref{2.1..}\), we take into account the corresponding solution \(\mu\). Let \(\mu :[0,\tilR)\to\mathbb{R}\) be the maximal solution to \(\eqref{mu ode}\). Inserting \(\eqref{lamda eliminate}\) into \(\eqref{1steady field eq11}\) yields
		\begin{align}\label{mu p w}
		\mu'(r)=4\pi re^{2\lambda(r)}(p(r)+w(r))
		\end{align}
		where
		\begin{align}\label{w def}
		w(r):=\frac{\int_{0}^{r}s^{2}\rho(s)ds}{r^{3}}.
		\end{align}
		Because the function \(\rho(r)\) is decreasing, we have
		\begin{align}
		w'(r)=\frac{\rho(r)}{r}-\frac{3\int_{0}^{r}s^{2}\rho(s)ds}{r^{4}}\le \frac{\rho(r)}{r}-\frac{\rho(r)r^{3}}{r^{4}}\le 0,
		\end{align}
		and
		\begin{align*}
		\left( e^{\lambda+\mu}(p+w)\right)'(r)=&e^{\lambda+\mu}  \left[ (\lambda'+\mu')(p+w)+p'+w' \right] \\
		=&e^{\lambda+\mu}\left[ 4\pi r e^{2\lambda}(\rho+p)(p+w)-4\pi r e^{2\lambda}(\rho+p)(p+w)+w'\right] \\
		=&e^{\lambda+\mu}w'\le 0,
		\end{align*}
		where we have used the \(\eqref{1steady momentum eq}\), \(\eqref{tov}\) and \(\eqref{mu p w}\) in the second line. This implies that
		\begin{align}\label{e lamda mu p w}
		e^{\lambda(r)+\mu(r)}(p(r)+w(r))\le e^{\lambda(0)+\mu(0)}(p(0)+w(0))=C>0.
		\end{align}
		Now assume that \(\tilR<\infty\). Then we have
		\begin{align*}
		w(r)\ge w(\tilR)=\frac{\int_{0}^{\tilR}s^{2}\rho(s)ds}{\tilR^{3}}=C>0.
		\end{align*}
		Combined with \(\eqref{e lamda mu p w}\), it leads to
		\begin{align*}
		e^{\lambda(r)+\mu(r)}\le C\ , \ 0\le r <\tilR.
		\end{align*} 
		Since \(\lambda(r)\ge0\) by \(\eqref{lamda eliminate}\) and \(\mu(r)\ge\mu_{0}\) by monotonicity, this implies that \(\mu\) is bounded on \([0,\tilR)\) which means that the solution $\mu(r)$ extends beyond $\tilR$ and therefore \(\tilR=\infty\).
	\end{proof}
	The following, a crucial step is to show that the steady state \(\rho\) and \(p\) have compact support, which is to prove the solution \(\ybar(r)\) has a unique zero at some finite radius \(R>0\).
	\begin{lemma}\label{2.4..}
		Let \(\ybar(r)\) satisfy the equation \(\eqref{ode}\) with \(\ybar(0)=\kappa>0\). Then the limit \(\ybar_{\infty}:=\lim_{r\to\infty}\ybar(r)\) exists and \(\ybar_{\infty}<0\), which implies that the function \(\ybar(r)\) has a unique zero.
	\end{lemma}
	\begin{proof}
		Since \(\ybar(r)\) is decreasing, the limit \(\ybar_{\infty}=C\) or \(\ybar_{\infty}=-\infty\). Assume that \(\ybar_{\infty}=-\infty\). Then there exists \(r_{0}<\infty\), when \(r>r_{0}\) we have \(\ybar(r)<0\) and \(g(\ybar(r))=h(\ybar(r))=0\). Now we consider the equation
		\begin{align}\label{r0 ode}
		\ybar^{\prime}(r)=-\frac{1}{1-2m(r)/r}\left( \frac{m(r)}{r^{2}}\right)\quad \textrm{with} \quad \ybar(r_{0})=0,
		\end{align}
		where \(m(r)\equiv C\) and \(r_{0}-2C>0\) due to \(\eqref{lamda eliminate}\) and \(\eqref{m def}\). Integrating \(\eqref{r0 ode}\) from \(r_{0}\) to \(r\), we get
		\begin{align*}
		\ybar(r)=-\int_{r_{0}}^{r}\frac{C}{s(s-2C)}ds=\frac{1}{2}\left(\ln\frac{r_{0}-2C}{r_{0}}-\ln\frac{r-2C}{r} \right) .
		\end{align*}
		Thus \(\ybar_{\infty}=\frac{1}{2}\ln\frac{r_{0}-2C}{r_{0}}=C\), which is a contradiction to \(\ybar_{\infty}=-\infty\). The following we need to show that \(\ybar_{\infty}<0\). Assume that \(\ybar_{\infty}\ge 0\). Then \(\ybar(r)\ge \ybar_{\infty}\) on \([0,\infty)\), and by the monotonicity of \(g\), 
		\begin{align*}
		m(r)=4\pi\int_{0}^{r}s^{2}g(y(s))ds\ge 4\pi g(\ybar_{\infty})\int_{0}^{r}s^{2}ds\ge\frac{4\pi}{3}r^{3}.
		\end{align*}
		According to the equation \(\eqref{ode}\),
		\begin{align}\label{decay estimate}
		\ybar'(r)\le-\frac{m(r)}{r^{2}}\le-\frac{4\pi}{3}r.
		\end{align}  
		Integrating this estimate we obtain a contradiction
		\begin{align*}
		\ybar(r)\le \ybar_{0}-\frac{2\pi}{3}r^{2}\to -\infty \ \textrm{as}\ r\to \infty,
		\end{align*}
		which completes the proof.
	\end{proof}
	Now we have proved that for every central value \(\ybar(0)=\kappa>0\) there exists a unique solution \(\ybar=\ybar_{\kappa}\) to \(\eqref{ode}\), which is defined on \([0,\infty)\), and the corresponding quantities \(\rho_{\kappa}\), \(p_{\kappa}\) are supported on the interval \([0,R_{\kappa}]\). In the literature \(\ybar(0)=\kappa\) is called the central redshift, which is used to parameterize the steady state solutions. It is worth noting that the central redshift \(\kappa\) and the central density\(\rho_{c}=\rho(0)\) are in a 1-1 relationship by \(\eqref{rho y}\) thus \(\rho_{c}\) is an equivalent parameterization. In order to show linear instability, we need to describe the asymptotic properties of the steady state solution. we consider the equation
	\begin{align}\label{pp}
	p'=-\frac{S(p)+p}{1-\frac{8\pi}{r}\int_{0}^{r}s^{2}S(p)ds}\left( \frac{4\pi}{r^{2}}\int_{0}^{r}s^{2}S(p)ds+4\pi rp\right) ,
	\end{align}
	where \(S(p):=\rho\), and its massless counterpart in terms of \(S^{\ast}(p)=p\):
	\begin{align}\label{p ast}
	p'=-\frac{2p}{1-\frac{8\pi}{r}\int_{0}^{r}s^{2}pds}\left( \frac{4\pi}{r^{2}}\int_{0}^{r}s^{2}pds+4\pi rp\right) .
	\end{align}
	Let \(p_{\kappa}\) and \(p^{\ast}_{\kappa}\) denote the solutions to \(\eqref{pp}\) and \(\eqref{p ast}\) respectively, and satisfy the boundary condition
	\begin{align*}
	p_{\kappa}(0)=e^{4\kappa}=p^{\ast}_{\kappa}(0).
	\end{align*}
	\begin{remark}
		We reparametrize our steady state family here and use the central pressure as the new parameter. However, the quantities \(\ybar\) and \(p\) are in a strictly increasing, one-to-one correspondence in such a way that \(\ybar\to\infty\) iff \(p\to\infty\), that is, \(p(0)\to\infty\) is equivalent to \(\ybar(0)\to\infty\).
	\end{remark}
	We now show that near the origin and for large \(\kappa\) the behavior of \(p_{\kappa}\) is captured by \(p_{\kappa}^{\ast}\).
	\begin{lemma}\label{2.6..}
		There exists a constant \(C>0\) such that for all \(\kappa>0\) and \(r\ge 0\),
		\begin{align*}
		|p_{\kappa}(r)-p_{\kappa}^{\ast}(r)|\le Ce^{6\kappa}\left( r^{2}+e^{4\kappa}r^{4}\right) \exp\left( C\left( e^{4\kappa}r^{2}+e^{8\kappa}r^{4}\right) \right) .
		\end{align*}
	\end{lemma}
	\begin{proof}
		The proof is inspired by \cite{HLR}. For simplicity we drop the subscript \(\kappa\) when there is no confusion and write \(p\) and \(p^{\ast}\) for the two solutions which we want to compare.  We first introduce re-scaled variables as follows:
		\begin{align*}
		p(r)=\alpha^{-2}\sigma(\tau),\quad p^{\ast}(r)=\alpha^{-2}\sigma^{\ast}(\tau), \quad r=\alpha\tau,
		\end{align*}
		where \(\alpha=e^{-2\kappa}\). A direct calculation gives 
		\begin{align}
		\sigma'(\tau)=&-\frac{\sigma(\tau)+\alpha^{2}S(\alpha^{-2}\sigma(\tau))}{1-\frac{8\pi}{\tau}\int_{0}^{\tau}s^{2}\alpha^{2}S(\alpha^{-2}\sigma(s))ds}\left( \frac{4\pi}{\tau^{2}}\int_{0}^{\tau}s^{2}\alpha^{2}S(\alpha^{-2}\sigma(s))ds+4\pi \tau \sigma(\tau)\right) ,\\	\label{2.22}
		(\sigma^{\ast})'(\tau)=&-\frac{2\sigma^{\ast}(\tau)}{1-\frac{8\pi}{\tau}\int_{0}^{\tau}s^{2}\sigma^{\ast}(s)ds}	
		\left( \frac{4\pi}{\tau^{2}}\int_{0}^{\tau}s^{2}\sigma^{\ast}(s)ds+4\pi \tau \sigma^{\ast}(\tau)\right) ,
		\end{align}
		and
		\begin{align*}
		\sigma(0)=1=\sigma^{\ast}(0).
		\end{align*}
		In order to estimate the difference \(\sigma'-(\sigma^{\ast})'\) and apply a Gronwall argument we
		need some preliminary estimates. First we
		note that \(\sigma\) and \(\sigma^{\ast}\) are decreasing, and hence, for \(\tau\ge 0\),
		\begin{align*}
		0\le \sigma(\tau), \sigma^{\ast}(\tau)\le 1,
		\end{align*}
		and
		\begin{align*}
		\alpha^{2}S(\alpha^{-2}\sigma(\tau))=&\sigma(\tau)+\alpha^{2}\le C,\\
		\left| \alpha^{2}S(\alpha^{-2}\sigma(\tau))-\sigma^{\ast}(\tau)\right| \le& C e^{-2\kappa}+\left| \sigma(\tau)-\sigma^{\ast}(\tau)\right| .
		\end{align*}
		By Buchdahl’s inequality (see \cite{Andreasson}), we have
		\begin{align*}
		\frac{1}{1-\frac{8\pi}{\tau}\int_{0}^{\tau}s^{2}\alpha^{2}S(\alpha^{-2}\sigma(s))ds},\quad \frac{1}{1-\frac{8\pi}{\tau}\int_{0}^{\tau}s^{2}\sigma^{\ast}(s)ds}< 9,\quad \tau\ge 0.
		\end{align*}
		Let us abbreviate now
		\begin{align*}
		x(\tau):=\max_{0\le s \le \tau}\left| \sigma(s)-\sigma^{\ast}(s)\right| .
		\end{align*}
		Combined with previous estimates, we get
		\begin{align*}
		&\left| \sigma'(\tau)-(\sigma^{\ast})'(\tau)\right|
		\le\left|\sigma(\tau)+\alpha^{2}S(\alpha^{-2}\sigma(\tau))-2\sigma^{\ast}(\tau) \right|\\ &\quad\quad\frac{1}{1-\frac{8\pi}{\tau}\int_{0}^{\tau}s^{2}\alpha^{2}S(\alpha^{-2}\sigma(s))ds}\left| \frac{4\pi}{\tau^{2}}\int_{0}^{\tau}s^{2}\alpha^{2}S(\alpha^{-2}\sigma(s))ds+4\pi \tau \sigma(\tau)\right| \\
		&\quad\quad+2\sigma^{\ast}(\tau)\left|  \frac{1}{1-\frac{8\pi}{\tau}\int_{0}^{\tau}s^{2}\alpha^{2}S(\alpha^{-2}\sigma(s))ds}- \frac{1}{1-\frac{8\pi}{\tau}\int_{0}^{\tau}s^{2}\sigma^{\ast}(s)ds}        \right| \\
		&\quad\quad\left| \frac{4\pi}{\tau^{2}}\int_{0}^{\tau}s^{2}\alpha^{2}S(\alpha^{-2}\sigma(s))ds+4\pi \tau \sigma(\tau)\right|+\frac{2\sigma^{\ast}(\tau)}{1-\frac{8\pi}{\tau}\int_{0}^{\tau}s^{2}\sigma^{\ast}(s)ds}\\
		&\quad\quad\left| \frac{4\pi}{\tau^{2}}\int_{0}^{\tau}s^{2}\alpha^{2}S(\alpha^{-2}\sigma(s))ds+4\pi \tau \sigma(\tau)-\frac{4\pi}{\tau^{2}}\int_{0}^{\tau}s^{2}\sigma^{\ast}(s)ds+4\pi \tau \sigma^{\ast}(\tau)\right| \\
		&\quad\quad\le C(e^{-2\kappa}+x(\tau))\tau+C(e^{-2\kappa}+x(\tau))\tau^{3}+C(e^{-2\kappa}+x(\tau))\tau\\
		&\quad\quad \le C(\tau+\tau^{3})(e^{-2\kappa}+x(\tau)).
		\end{align*}
		Gronwall’s lemma implies that
		\begin{align*}
		x(\tau)\le C(\tau^{2}+\tau^{4})\exp \left(C\left(\tau^{2}+\tau^{4} \right)  \right) e^{-2\kappa}.
		\end{align*}
		Recalling the original variable, we finally obtain
		\begin{align*}
		|p(r)-p^{\ast}(r)|\le Ce^{6\kappa}\left( r^{2}+e^{4\kappa}r^{4}\right) \exp\left( C\left( e^{4\kappa}r^{2}+e^{8\kappa}r^{4}\right) \right),
		\end{align*}
		which completes the proof.
	\end{proof}
	\begin{corollary}\label{2.7..}
		Let \(p^{\ast}_{\kappa}\) denote the solution of \(\eqref{p ast}\) with initial data \(p^{\ast}_{\kappa}(0)=e^{4\kappa}\). Then for all \(\kappa\ge 0\),
		\begin{align*}
		p^{\ast}_{\kappa}(r)=e^{4\kappa}p^{\ast}_{0}(e^{2\kappa}r),\quad r\ge 0.
		\end{align*}
	\end{corollary}
	\begin{proof}
		The proof is straightforward using the fact that the rescaled function \(\sigma^{\ast}(\tau)=e^{-4\kappa}p^{\ast}_{\kappa}(r)\) solves \(\eqref{2.22}\).
	\end{proof}
	The following result gives detailed estimate for large \(r\) for the behavior \(p_{0}^{\ast}(r)\).
	\begin{lemma}\label{2.8..}
		Let \(\rho_{0}^{\ast}\) and \(m_{0}^{\ast}\) denote the quantities induced by \(p_{0}^{\ast}\). Then for any \(\gamma\in (0,1)\),
		\begin{align*}
		\left|r^{2} \rho_{0}^{\ast}(r)-\frac{1}{16\pi}\right| +\left|\frac{m_{0}^{\ast}(r)}{r}-\frac{1}{4} \right| \le Cr^{-\gamma},\ r>0.
		\end{align*}
	\end{lemma}
	\begin{proof}
		The proof relies on transferring the steady state equations \(\eqref{p ast}\) into a planar, autonomous dynamical system analogous to that in \cite{HLR}. According to the equation of state \(p=\rho\), we have
		\begin{align*}
		\frac{d\rho}{dr}=&-\frac{2\rho}{1-\frac{2m}{r}}\left( \frac{m}{r^{2}}+4\pi r\rho\right), \\
		\frac{dm}{dr}=&4\pi r^{2}\rho.
		\end{align*}
		This can be written in terms of \(u_{1}(r)=r^{2}\rho(r)\), \(u_{2}(r)=m(r)/r\):
		\begin{align*}
		\frac{du_{1}}{dr}=&\frac{2}{r}u_{1}-\frac{2u_{1}}{1-2u_{2}}\left( \frac{1}{r}u_{2}+4\pi\frac{1}{r}u_{1}\right), \\
		\frac{du_{2}}{dr}=&4\pi \frac{1}{r}u_{1}-\frac{1}{r}u_{2}.
		\end{align*}
		Multiplying both equations with \(r\) and introducing \(w_{1}(\tau)=u_{1}(r)\), \(w_{2}(\tau)=u_{2}(r)\) with \(\tau=\ln r\), we obtain
		\begin{align}\label{f1}
		\frac{dw_{1}}{d\tau}=&\frac{w_{1}}{1-2w_{2}}\left( 2-6w_{2}-8\pi w_{1}\right) ,\\ \label{f2}
		\frac{dw_{2}}{d\tau}=&4\pi w_{1}-w_{2}.
		\end{align}
		We denote the right hand side of the system \(\eqref{f1},\eqref{f2}\) by \(F(w)\), which is defined
		and smooth for \(w_{1}\in \mathbb{R}\) and \(w_{2}\in (-\infty,1/2)\). The system has two steady states:
		\begin{align*}
		F(w)=0\quad \Leftrightarrow w=(0,0) \ \textrm{or}\ w=Z:=\left( \frac{1}{16\pi},\frac{1}{4}\right) .
		\end{align*}
		Firstly we observe that
		\begin{align*}
		DF(Z)=\left( \begin{matrix}
		-1&-\frac{3}{4\pi}\\
		4\pi&-1
		\end{matrix}
		\right) 
		\end{align*}
		has eigenvalues
		\begin{align*}
		\lambda_{1,2}=-1\pm \sqrt{3}i
		\end{align*}
		so that \(Z\) is an exponential sink. On the other hand,
		\begin{align*}
		DF(0,0)=\left( \begin{matrix}
		2&0\\
		4\pi&-1
		\end{matrix}
		\right)
		\end{align*}
		with eigenvalues \(-1\) and \(2\), stable direction \((0,1)\) and unstable direction \((3,4\pi)\). For the solution induced by \(p_{0}^{\ast}\) it holds that \(w(\tau)\to(0,0)\) for \(\tau\to-\infty\) which corresponds to \(r\to 0\). Since the corresponding trajectory lies in the first quadrant \([w_{1}>0,\ w_{2}>0]\), it must coincide with the corresponding branch \(T\) of the unstable manifold of \((0,0)\). Let \(D\) denote the triangular region bounded by the lines
		\begin{align*}
		[w_{1}=0],\quad[w_{2}=4/9],\quad [w_{1}=w_{2}].
		\end{align*}
		It is clear that \(Z\in D\). Now we show that the trajectory \(T\) cannot leave this domain. Firstly the unstable direction \((3,4\pi)\) points into the interior of \(D\), so the part of \(T\) close to the origin must lie in the domain \(D\). Then by the fact Buchdahl's inequality and \(w_{1}>0\), we can see that the trajectory \(T\) can not cross the boundary \([w_{1}=0]\) and \([w_{2}=4/9]\). Finally in the line \([w_{1}=w_{2}]\), we have
		\begin{align*}
		(-1,1)\cdot F(w_{1},w_{1})=\frac{1}{1-2w_{1}}\left((4\pi-3)w_{1}+8w_{1}^{2} \right) >0.
		\end{align*}
		This leads to the trajectory \(T\) can not cross the boundary \([w_{1}=w_{2}]\), hence \(T\) cannot leave the domain \(D\). According to Poincar\'e-Bendixson Theorem, the \(\omega-\)limit set of \(T\) must either coincide with \(Z\), or with a periodic orbit. However, according to Dulac’s negative criterion, the set does not contain a periodic orbit, since
		\begin{align*}
		\div \left( \frac{1}{w_{1}}F(w)\right) =-\frac{8\pi}{1-2w_{2}}-\frac{1}{w_{1}}<0.
		\end{align*}
		In view of the real part of the eigenvalues \(\lambda_{1,2}\) it follows that, for any \(0<\gamma<1\) and all \(\tau\) sufficiently large,
		\begin{align*}
		\left|w(\tau)-Z \right|\le C e^{-\gamma\tau} .
		\end{align*}
		The proof is completed by rewriting in terms of the original variables.
	\end{proof}
	We now combine the previous three lemmas to show that the crucial asymptotic properties of the steady state solution
	for \(\kappa\) large. 
	\begin{proposition}\label{2.9..}
		There exist parameters \(0<\alpha_{1}<\alpha_{2}<\frac{1}{4}\), \(\kappa_{0}>0\) sufficiently large, and constants $\bar{\delta}>0, C>0$ such that on the interval
		\begin{align*}
		[r_{\kappa}^{1},r_{\kappa}^{2}]=[\kappa^{\alpha_{1}}e^{-2\kappa},\kappa^{\alpha_{2}}e^{-2\kappa}]
		\end{align*}
		and for any \(\kappa\ge\kappa_{0}\) the following estimates hold:
		\begin{align}\label{asy 1}
		\left| r^{2}\rho_{\kappa}(r)-\frac{1}{16\pi}\right|,\left|r^{2}p_{\kappa}(r)-\frac{1}{16\pi} \right| ,\left|\frac{m_{\kappa}(r)}{r}-\frac{1}{4} \right|,\left| r\mu'_{\kappa}-1\right|,\left| r^{2}\mu''_{\kappa}+1\right|,\left| e^{2\lambda_{\kappa}}-2\right|,\left|r\lambda'_{\kappa},r^{2}\lambda_{\kappa}'' \right| \le C\kappa^{-\bar{\delta}},
		\end{align}
		and
		\begin{align}\label{asy 2}
		C_{\kappa}\exp(-C\kappa^{-\bar{\delta}}\ln\kappa)\le\frac{e^{\mu_{\kappa}(r)}}{r}\le C_{\kappa}\exp(C\kappa^{-\bar{\delta}}\ln\kappa),
		\end{align}
		where \(C_{\kappa}>0\) does depend on \(\kappa\), but \(C>0\) does not. Notice that in this estimate the exponential terms on both sides converge to 1 as \(\kappa\to\infty\).
	\end{proposition}
	\begin{proof}
		First we note that by Corollary \ref{2.7..} and Lemma \(\ref{2.8..}\),
		\begin{align*}
		\left| r^{2}p_{\kappa}^{\ast}(r)-\frac{1}{16\pi}\right| =\left| r^{2}\rho_{\kappa}^{\ast}(r)-\frac{1}{16\pi}\right|=\left|(e^{2\kappa}r)^{2}\rho_{0}^{\ast}(e^{2\kappa}r) -\frac{1}{16\pi}\right| \le C e^{-2\kappa\gamma}r^{-\gamma},
		\end{align*}
		so that together with Lemma \(\ref{2.6..}\),
		\begin{align*}
		\left| r^{2}p_{\kappa}(r)-\frac{1}{16\pi}\right| \le& Ce^{6\kappa}\left( r^{4}+e^{4\kappa}r^{6}\right) \exp\left( C\left( e^{4\kappa}r^{2}+e^{8\kappa}r^{4}\right) \right) +e^{-2\kappa\gamma}r^{-\gamma}\\
		\le& C(\kappa^{4\alpha_{2}}+\kappa^{6\alpha_{2}})\exp\left( C(\kappa^{2\alpha_{2}}+\kappa^{4\alpha_{2}}) -2\kappa\right)+C \kappa^{-\gamma\alpha_{1}}.
		\end{align*}
		Selecting \(0<\alpha_{1}<\alpha_{2}<\frac{1}{4}\), the first term on the right side of above inequality has a faster decay with respect to \(\kappa\) than the second term. Let \(\bar{\delta}=\gamma\alpha_{1}\). Without loss of generality, there exist \(\kappa_{0}>0\) such that on the interval \([r_{\kappa}^{1},r_{\kappa}^{2}]\) for any \(\kappa\ge\kappa_{0}\) we have
		\begin{align*}
		\left| r^{2}p_{\kappa}(r)-\frac{1}{16\pi}\right| \le C\kappa^{-\bar{\delta}}.
		\end{align*}
		By the equation of state \(\eqref{state eq}\), it follows that
		\begin{align*}
		\left| r^{2}\rho_{\kappa}(r)-\frac{1}{16\pi}\right|=\left| r^{2}p_{\kappa}(r)-\frac{1}{16\pi}\right|+\left|r^{2} \right| \le C\kappa^{-\bar{\delta}}.
		\end{align*}
		The scaling property in Lemma \ref{2.6..} implies that
		\begin{align*}
		\left| \frac{m_{\kappa}^{\ast}(r)}{r}-\frac{1}{4}\right| =	\left| \frac{m_{0}^{\ast}(e^{2\kappa}r)}{e^{2\kappa}r}-\frac{1}{4}\right|\le Ce^{-2\kappa\gamma}r^{-\gamma}\le C\kappa^{-\bar{\delta}}.
		\end{align*}
		Hence integrating the \(\rho-\rho^{\ast}\) estimate directly yields
		\begin{align*}
		\left| \frac{m_{\kappa}(r)}{r}-\frac{1}{4}\right| \le\left| \frac{m_{\kappa}(r)}{r}-\frac{m_{\kappa}^{\ast}(r)}{r}\right|+\left| \frac{m_{\kappa}^{\ast}(r)}{r}-\frac{1}{4}\right|\le C\kappa^{-\bar{\delta}}.
		\end{align*}
		Using the steady state equation \(\eqref{1steady field eq00},\eqref{1steady field eq11}\) and \(\eqref{lamda eliminate}\), we obtain 
		\begin{align*}
		\left| r\mu'_{\kappa}-1\right|,\left| r^{2}\mu''_{\kappa}+1\right|,\left| e^{2\lambda_{\kappa}}-2\right|,\left|r\lambda'_{\kappa},r^{2}\lambda_{\kappa}'' \right| \le C\kappa^{-\bar{\delta}},
		\end{align*}
		where we omit the lengthy but straightforward argument. In order to prove the asymptotical property \(\eqref{asy 2}\), we need some special preparation. Clearly,
		\begin{align*}
		\mu_{\kappa}(r)=&\mu_{\kappa}(r_{\kappa}^{1})+\int_{r_{\kappa}^{1}}^{r}\frac{1}{s}ds+\int_{r_{\kappa}^{1}}^{r}\left(\mu_{\kappa}'(s)-\frac{1}{s} \right) ds\\
		=&\mu_{\kappa}(r_{\kappa}^{1})+\ln\left(\frac{r}{r_{\kappa}^{1}} \right) +\int_{r_{\kappa}^{1}}^{r}\left(\mu_{\kappa}'(s)-\frac{1}{s} \right) ds.
		\end{align*}
		Therefore
		\begin{align*}
		e^{\mu_{\kappa}(r)}=e^{\mu_{\kappa}(r_{\kappa}^{1})}\frac{r}{r_{\kappa}^{1}}\exp\left(\int_{r_{\kappa}^{1}}^{r}\left(\mu_{\kappa}'(s)-\frac{1}{s} \right) ds \right) ,
		\end{align*}
		and using the previous results it follows that
		\begin{align*}
		\exp\left(\int_{r_{\kappa}^{1}}^{r}\left(\mu_{\kappa}'(s)-\frac{1}{s} \right) ds \right)\le\exp\left( \int_{r_{\kappa}^{1}}^{r_{\kappa}^{2}}\left(C\kappa^{-\bar{\delta}}\frac{1}{s} \right) ds\right) =\left(\frac{r_{\kappa}^{2}}{r_{\kappa}^{1}} \right) ^{C\kappa^{-\bar{\delta}}}=e^{C(\alpha_{2}-\alpha_{1})\kappa^{-\bar{\delta}}\ln\kappa}.
		\end{align*}
		The lower estimate is completely analogous and this finishes the proof of Proposition \(\ref{2.9..}\).
	\end{proof}

	%%%%%%%%%%%%%%%%%%%
	%%%%%%%%%%%%%%%%%%%
	\section{Linearized equation}\label{sec3}
	We start by introducing an additional quantity \(n=N(\rho)\) called number density:
	\begin{align*}
	N(\rho):=\exp\left( \int_{1}^{\rho}\frac{ds}{s+p(s)}\right).
	\end{align*}
	Clearly, we have
	\begin{align}\label{dndrho}
	\frac{dN}{d\rho}=\frac{N}{\rho+p(\rho)}.
	\end{align}
	Using \(\eqref{dndrho}\) and the equation \(\eqref{1steady momentum eq}\) it follows that
	\begin{align*}
	\left( \frac{dN}{d\rho}(\rho_{\kappa})\right)'=-\frac{N(\rho_{\kappa})}{(\rho_{\kappa}+p_{\kappa})^{2}}p_{\kappa}'= \frac{N(\rho_{\kappa})}{\rho_{\kappa}+p_{\kappa}}\mu_{\kappa}'=\frac{dN}{d\rho}(\rho_{\kappa})\mu_{\kappa}',
	\end{align*}
	on the interval \([0,R_{\kappa}]\), where \([0,R_{\kappa}]\) is  the support of the steady state. We integrate this differential equation to obtain
	\begin{align*}
	\frac{dN}{d\rho}(\rho_{\kappa}(r))=\frac{dN}{d\rho}(1)e^{\mu_{\kappa}(r)-\mu_{\kappa}(R_{\kappa})}.
	\end{align*}
	Let the function \(N_{\kappa}:=c_{\kappa}N\) and choose the normalization constant \(c_{\kappa}\) such that \(\frac{dN_{\kappa}}{d\rho}(1)=e^{\mu_{\kappa}(R_{\kappa})}\). We can get the simplified identity
	\begin{align}\label{n mu}
	\frac{n_{\kappa}}{\rho_{\kappa}+p_{\kappa}}=\frac{dN_{\kappa}}{d\rho}(\rho_{\kappa})=e^{\mu_{\kappa}}
	\end{align}
	on the interval \([0,R_{\kappa}]\). We introduce the following variable
	\begin{align}\label{n Psi}
	\Psi_{\kappa}(r):=\frac{1}{n_{\kappa}(r)}\frac{dp}{d\rho}(\rho_{\kappa}(r)),\ 0\le r\le R_{\kappa},
	\end{align}
	which is known as enthalpy of the steady state, and it is obvious that \(\Psi_{\kappa}n'_{\kappa}=-\mu_{\kappa}'\). In order to linearize the Einstein-Euler system \(\eqref{1field eq00}-\eqref{1momentum eq SS}\), we need to write them in Lagrangian coordinates, since there is difference with the gaseous case corresponding to a Cauchy problem. Let \(\eta(y,t)\) be the radial position of the fluid particle at time \(t\) so that
	\begin{align}\label{eta def1}
	\partial_{t}\eta=\frac{u}{u^{0}}\circ\eta \ \ \ \  \text{with} \ \ \ \ \eta(y,0)=\eta_{0}(y).
	\end{align}
	Here \(\eta_{0}\) is not necessarily the identity map but depend on the initial density profile. By a slight abuse of notation, we still write the Lagrangian quantities \(u=u(t,y),\rho=\rho(t,y),p=p(t,y),\lambda=\lambda(t,y)\) and \(\mu=\mu(t,y)\) instead of the corresponding Euler variables, and \('\) denotes the derivative with respect to \(y\). Then we get the Einstein-Euler equation in Lagrangian coordinates
	\begin{align}\label{field eq00 LC}
	&e^{-2\lambda}(\frac{2\eta\lambda'}{\eta'}-1)+1=8\pi \eta^{2}\left(\rho+e^{2\lambda}(\rho+p)u^{2} \right),\\ \label{field eq11 LC}
	&e^{-2\lambda}(\frac{2\eta\mu'}{\eta'}+1)-1=8\pi \eta^{2}\left(p+e^{2\lambda}(\rho+p)u^{2} \right),\\ \label{field eq01 LC}
	&\dot{\lambda}-e^{\mu}\frac{u\lambda'}{\left\langle u \right\rangle \eta'}=-4\pi \eta e^{\mu+2\lambda}\left\langle u\right\rangle u(\rho+p),\\ \label{mass eq LC}
	&\dot{\rho}+(\rho+p)\left[\dot{\lambda}+e^{\mu}\frac{u}{\left\langle u\right\rangle}\left( \frac{\mu'}{\eta'}+\frac{2}{\eta}\right)+e^{\mu}\frac{u'}{\left\langle u\right\rangle\eta'} +e^{2\lambda}\frac{u}{\left\langle u\right\rangle} \frac{\dot{\lambda}u+\dot{u}}{\left\langle u\right\rangle}  -e^{\mu+2\lambda}\frac{u^{2}}{\left\langle u\right\rangle^{2}}\frac{\lambda'u+u'}{\left\langle u\right\rangle\eta'} \right] =0,\\ \label{momentum eq LC}
	&(\rho+p)\left[ e^{2\lambda}(\dot{u}+2\dot{\lambda}u)+e^{\mu}\left\langle u\right\rangle\frac{\mu'}{\eta'}-e^{\mu+2\lambda}\frac{u}{\left\langle u\right\rangle}\frac{\lambda'u}{\eta'}\right] +e^{\mu}\left\langle u\right\rangle\frac{ p'}{\eta'}+e^{2\lambda}u \dot{p}-e^{\mu+2\lambda}\frac{u}{\left\langle u\right\rangle}\frac{p'u}{\eta'}=0.
	\end{align}
	 Let us now linearize the Einstein-Euler system \(\eqref{field eq00 LC}-\eqref{momentum eq LC}\) around a given steady state. By a slight abuse of notation again, we write \(n,u,\rho,p,\lambda,\mu\) for the Lagrangian perturbations, that is \(\rho_{\kappa}+\rho,p_{\kappa}+p,n_{\kappa}+n\) and so on correspond to the solution of the original nonlinear system. First we linearize the equation \(\rho_{\kappa}+\rho=N_{\kappa}^{-1}(n_{\kappa}+n)\) and \(p_{\kappa}+p=p(N_{\kappa}^{-1}(n_{\kappa}+n))\)
	\begin{align}\label{linear rho}
	&\rho=\frac{dN_{\kappa}^{-1}}{dn}(n_{\kappa})n=\frac{1}{\frac{dN_{\kappa}}{d\rho}(\rho_{\kappa})}n=e^{-\mu_{\kappa}}n,\\ \label{linear p}
	&p=\frac{dp}{d\rho}(\rho_{\kappa})\frac{dN_{\kappa}^{-1}}{dn}(n_{\kappa})n=\Psi_{\kappa}(\rho_{\kappa}+p_{\kappa})n,
	\end{align}
	where we have used \(\eqref{n mu}\) and \(\eqref{n Psi}\). Let \(\eta(t,y)=y(1+\zeta(t,y))\), so
	\begin{align}\label{d eta}
	\eta'=1+\zeta+y\zeta'.
	\end{align}
	Substituting \(\eqref{d eta}\) into \(\eqref{field eq00 LC}\), we have
	\begin{align}\label{field eq00 LC1}
	e^{-2(\lambda_{\kappa}+\lambda)}\left( \frac{2y(1+\zeta)(\lambda_{\kappa}'+\lambda')}{1+\zeta+y\zeta' }-1\right)+1=8\pi y^{2}(1+\zeta)^{2}\left( \rho_{\kappa}+\rho+e^{2(\lambda_{\kappa}+\lambda)}(\rho_{\kappa}+p_{\kappa}+\rho+p)u^{2}\right).
	\end{align}
	Since
	\begin{align*}
	(&1+\zeta+y\zeta')^{-1}=1-\zeta-y\zeta'+o(|\zeta|+|\zeta'|)\\
	&e^{-2\lambda}=1-2\lambda+o(\lambda)
	\end{align*}
	we simplify \(\eqref{field eq00 LC1}\) by discarding non-linear terms
	\begin{align*}
	e^{-2\lambda_{\kappa}}y(-2\lambda\lambda_{\kappa}'-y\lambda_{\kappa}'\zeta'+\lambda')+e^{-2\lambda_{\kappa}}\lambda=&4\pi y^{2}(\rho+2\rho_{\kappa}\zeta)\\
	e^{-2\lambda_{\kappa}}(-2y\lambda\lambda_{\kappa}'+y\lambda'+\lambda)=&4\pi y^{2}(\rho+2\rho_{\kappa}\zeta)+e^{-2\lambda_{\kappa}}y^{2}\lambda_{\kappa}'\zeta'\\
	(y\lambda e^{-2\lambda_{\kappa}})'=&y^{2}(4\pi\rho+8\pi\rho_{\kappa}\zeta+e^{-2\lambda_{\kappa}}\lambda_{\kappa}'\zeta').
	\end{align*}
	Intergating the above equation we obtain the linearization of \(\eqref{field eq00 LC}\)
	\begin{align}\label{linear 1}
	\lambda=\frac{e^{2\lambda_{\kappa}}}{y}\int_{0}^{y}s^{2}(4\pi\rho+8\pi\rho_{\kappa}\zeta+e^{-2\lambda_{\kappa}}\lambda_{\kappa}'\zeta')ds=:\frac{e^{2\lambda_{\kappa}}}{y}\int_{0}^{y}s^{2}f(s)ds.
	\end{align}
	Similarly linearizing \(\eqref{field eq11 LC}\) we arrive at
	\begin{align}\label{linear 2}
	\begin{split}
	y\mu'=&(2y\mu_{\kappa}'+1)\lambda+4\pi y^{2}e^{2\lambda_{\kappa}}p+y^{2}(\mu_{\kappa}'\zeta'+8\pi p_{\kappa}\zeta e^{2\lambda_{\kappa}})\\
	=&(2y\mu_{\kappa}'+1)\lambda+y\Psi_{\kappa}(\lambda_{\kappa}'+\mu_{\kappa}')e^{\mu_{\kappa}}\rho+y^{2}(\mu_{\kappa}'\zeta'+8\pi p_{\kappa}\zeta e^{2\lambda_{\kappa}}),
	\end{split}
	\end{align}
	where we have used \(\eqref{tov}\), \(\eqref{linear rho}\) and \(\eqref{linear p}\). Linearizing \(\eqref{momentum eq LC}\) we obtain the equation
	\begin{align*}
	0=&e^{2\lambda_{\kappa}}\dot{u}+\frac{\rho+p}{\rho_{\kappa}+p_{\kappa}}e^{\mu_{\kappa}}\mu_{\kappa}'+e^{\mu_{\kappa}}\mu'+\frac{e^{\mu_{\kappa}}}{\rho_{\kappa}+p_{\kappa}}p' \\
	=&e^{2\lambda_{\kappa}}\dot{u}+\frac{e^{\mu_{\kappa}}}{\rho_{\kappa}+p_{\kappa}}\mu_{\kappa}'\rho+\Psi_{\kappa}e^{\mu_{\kappa}}e^{\mu_{\kappa}}\mu_{\kappa}'\rho+e^{\mu_{\kappa}}\mu' +\frac{e^{\mu_{\kappa}}}{\rho_{\kappa}+p_{\kappa}}\Psi_{\kappa}'(\rho_{\kappa}+p_{\kappa})e^{\mu_{\kappa}}\rho  \\  &+\frac{e^{\mu_{\kappa}}}{\rho_{\kappa}+p_{\kappa}}\left[ \Psi_{\kappa}(\rho_{\kappa}'+p_{\kappa}')e^{\mu_{\kappa}}\rho+\Psi_{\kappa}(\rho_{\kappa}+p_{\kappa})e^{\mu_{\kappa}}\mu_{\kappa}'\rho+\Psi_{\kappa}(\rho_{\kappa}+p_{\kappa})e^{\mu_{\kappa}}\rho'     \right] \\
	=&e^{2\lambda_{\kappa}}\dot{u}+\Psi_{\kappa}e^{\mu_{\kappa}}e^{\mu_{\kappa}}\mu_{\kappa}'\rho+e^{\mu_{\kappa}}\mu'+\Psi_{\kappa}'e^{\mu_{\kappa}}e^{\mu_{\kappa}}\rho+\Psi_{\kappa}e^{\mu_{\kappa}}e^{\mu_{\kappa}}\rho'\\
	=&e^{2\lambda_{\kappa}}\dot{u}+e^{\mu_{\kappa}}\mu'+e^{\mu_{\kappa}}(e^{\mu_{\kappa}}\Psi_{\kappa}\rho)'.
	\end{align*}
	Multiplying by \(e^{\lambda_{\kappa}}\) we find that
	\begin{align}\label{linear 5}
	\begin{split}
	0=&e^{3\lambda_{\kappa}}\dot{u}+e^{\lambda_{\kappa}+\mu_{\kappa}}\mu'+e^{\lambda_{\kappa}+\mu_{\kappa}}(e^{\mu_{\kappa}}\Psi_{\kappa}\rho)'\\
	=&e^{3\lambda_{\kappa}}\dot{u}+\frac{1}{y}e^{\lambda_{\kappa}+\mu_{\kappa}}(2y\mu_{\kappa}'+1)\lambda+(e^{\lambda_{\kappa}+2\mu_{\kappa}}\Psi_{\kappa}\rho)'+ye^{3\lambda_{\kappa}+\mu_{\kappa}}(8\pi p_{\kappa}\zeta+e^{-2\lambda_{\kappa}}\mu_{\kappa}'\zeta'),
	\end{split}
	\end{align}
	where we have use the linearized equation \(\eqref{linear 2}\). Linearizing \(\eqref{field eq01 LC}\) and using \(\eqref{tov}\) we arrive at
	\begin{align}\label{linear 3}
	\begin{split}
	\dot{\lambda}=&e^{\mu_{\kappa}}\lambda_{\kappa}'u-4\pi y e^{\mu_{\kappa}+2\lambda_{\kappa}}u(\rho_{\kappa}+p_{\kappa})\\
	=&e^{\mu_{\kappa}}\lambda_{\kappa}'u-e^{\mu_{\kappa}}(\lambda_{\kappa}'+\mu_{\kappa}')u\\
	=&-e^{\mu_{\kappa}}\mu_{\kappa}'u.
	\end{split}
	\end{align}
	From \(\eqref{linear 3}\) it follows that linearizing \(\eqref{mass eq LC}\) we get
	\begin{align}\label{linear 4}
	\begin{split}
	0=&\dot{\rho}+(\rho_{\kappa}+p_{\kappa})\left[ \dot{\lambda}+e^{\mu_{\kappa}}u\left( \mu_{\kappa}'+\frac{2}{y}\right) +e^{\mu_{\kappa}}u'\right] \\
	=&\dot{\rho}+(\rho_{\kappa}+p_{\kappa})\left[e^{\mu_{\kappa}}\left( \frac{2u}{y}+u'\right)  \right] \\
	=&\dot{\rho}+\frac{n_{\kappa}}{y^{2}}(y^{2}u)'.
	\end{split}
	\end{align}
	With respect to the linearized equation of the free boundary problem, the usual approach is to transform the density perturbation into perturbation of the Lagrangian parameter. Linearizing \(\eqref{eta def1}\) we arrive at
	\begin{align}\label{linear eta}
	y\dot{\zeta}=e^{\mu_{\kappa}}u.
	\end{align}
	Together with the relation \(\eqref{linear 4}\) this yields
	\begin{align*}                                                     
	\dot{\rho}+\frac{n_{\kappa}}{y^{2}}(e^{-\mu_{\kappa}}y^{3}\dot{\zeta})'=0\quad \Longrightarrow \quad \rho+\frac{n_{\kappa}}{y^{2}}(e^{-\mu_{\kappa}}y^{3}\zeta)'=\rho_{0}+\frac{n_{\kappa}}{y^{2}}(e^{-\mu_{\kappa}}y^{3}\zeta)'| _{t=0}.
	\end{align*}
	For a small perturbation \(\rho_{0}\), we can pick \(\zeta_{0}\) such that
	\begin{align}\label{rho zeta}
	\rho_{0}+\frac{n_{\kappa}}{y^{2}}(e^{-\mu_{\kappa}}y^{3}\zeta)'| _{t=0}=0 \quad \Longrightarrow \quad \rho+\frac{n_{\kappa}}{y^{2}}(e^{-\mu_{\kappa}}y^{3}\zeta)'=0.
	\end{align}
	Since \(\rho(R_{\kappa})=0\), we have the following boundary condition by \(\eqref{rho zeta}\)
	\begin{align}
	(3-R_{\kappa}\mu_{\kappa}'(R_{\kappa}))\zeta(R_{\kappa})+R_{\kappa}\zeta'(R_{\kappa})=0.
	\end{align}
	Combined with the previous results, we give the linearized equation about the perturbation variable \(\zeta(t,y)\).
	\begin{proposition}\label{prop: linearized operator}
		For the given steady state solution \(n_{\kappa},\rho_{\kappa},p_{\kappa},\lambda_{\kappa},\mu_{\kappa},\Psi_{\kappa}\) induced by \(y_{\kappa}\), we arrive at the linearized equation of the Einstein-Euler system \(\eqref{1field eq00}-\eqref{1momentum eq SS}\) through the perturbation variable \(\zeta(t,y)\) with a Robin type boundary condition
		\begin{align}\label{linear eq}
		ye^{3\lambda_{\kappa}-\mu_{\kappa}}\ddot{\zeta}+e^{3\lambda_{\kappa}+\mu_{\kappa}}\frac{2y\mu_{\kappa}'+1}{y^{2}}\int_{0}^{y}s^{2}f(s)ds+\left( e^{2\mu_{\kappa}+\lambda_{\kappa}}\Psi_{\kappa}\rho\right)' +ye^{3\lambda_{\kappa}+\mu_{\kappa}}(8\pi p_{\kappa}\zeta+e^{-2\lambda_{\kappa}}\mu_{\kappa}'\zeta')=&0\\ \label{linear boundary condition}
		(3-R_{\kappa}\mu_{\kappa}'(R_{\kappa}))\zeta(R_{\kappa})+R_{\kappa}\zeta'(R_{\kappa})=&0,
		\end{align}
		where \(f\) and \(\rho\) given by \(\eqref{linear 1}\) and \(\eqref{rho zeta}\) respectively.
	\end{proposition}
	In general stability analysis, we usually hope that the linearized equation have great symmetry, however, the existence of the integral term in \(\eqref{linear eq}\) violates this property. We need a more precise treatment of the linearized equation. According to the definition of the function \(f\), we have
	\begin{align*}
	s^{2}f(s)=&4\pi s^{2}\rho+8\pi s^{2}\rho_{\kappa}\zeta+e^{-2\lambda_{\kappa}}s^{2}\lambda_{\kappa}'\zeta'\\
	=&-4\pi n_{\kappa}(e^{-\mu_{\kappa}}s^{3}\zeta)'+8\pi s^{2}\rho_{\kappa}\zeta+e^{-2\lambda_{\kappa}}s^{2}\lambda_{\kappa}'\zeta'\\
	=&-4\pi e^{\mu_{\kappa}}(\rho_{\kappa}+p_{\kappa})(-e^{-\mu_{\kappa}}\mu_{\kappa}'s^{3}\zeta+3e^{-\mu_{\kappa}}s^{2}\zeta+e^{-\mu_{\kappa}}s^{3}\zeta')+8\pi s^{2}\rho_{\kappa}\zeta+e^{-2\lambda_{\kappa}}s^{2}\lambda_{\kappa}'\zeta'\\
	=&-4\pi p_{\kappa}'s^{3}\zeta-12\pi p_{\kappa}s^{2}\zeta-4\pi\rho_{\kappa}s^{2}\zeta-4\pi (\rho_{\kappa}+p_{\kappa})s^{3}\zeta'+e^{-2\lambda_{\kappa}}s^{2}\lambda_{\kappa}'\zeta'\\
	=&-4\pi(p_{\kappa}s^{3})'\zeta-4\pi\rho_{\kappa}s^{2}\zeta-e^{-2\lambda_{\kappa}}\mu_{\kappa}'s^{2}\zeta'\\
	=&-4\pi(p_{\kappa}s^{3})'\zeta-4\pi\rho_{\kappa}s^{2}\zeta-4\pi s^{3}(p_{\kappa}+w_{\kappa})\zeta',
	\end{align*}
	where we have used \(\eqref{1steady momentum eq},\eqref{tov},\eqref{mu p w},\eqref{n mu}\) and \(\eqref{rho zeta}\). By integrating the above identity we get
	\begin{align}\label{key eq1}
	\begin{split}
	\int_{0}^{y}s^{2}f(s)ds=&\int_{0}^{y}[-4\pi(p_{\kappa}s^{3})'\zeta-4\pi\rho_{\kappa}s^{2}\zeta-4\pi s^{3}(p_{\kappa}+w_{\kappa})\zeta']ds\\
	=&\int_{0}^{y}[-4\pi\rho_{\kappa}s^{2}\zeta-4\pi s^{3}w_{\kappa}\zeta']ds-4\pi p_{\kappa}y^{3}\zeta\\
	=&-m_{\kappa}\zeta-4\pi p_{\kappa}y^{3}\zeta,
	\end{split}
	\end{align}
	where we used integration by parts twice and the integral term is eliminated. In order to get the desired symmetry, we deal with the terms about \(\zeta'\) and \(\zeta''\) in the linearized equation \(\eqref{linear eq}\)
	\begin{align}\label{key eq2}
	\begin{split}
	&\left( e^{2\mu_{\kappa}+\lambda_{\kappa}}\Psi_{\kappa}\rho\right)'+ye^{\lambda_{\kappa}+\mu_{\kappa}}\mu_{\kappa}'\zeta'\\
	=&\left(e^{2\mu_{\kappa}+\lambda_{\kappa}}y^{-2}(-e^{-\mu_{\kappa}}y^{3}\zeta)' \right)'+ye^{\lambda_{\kappa}+\mu_{\kappa}}\mu_{\kappa}'\zeta'\\
	=&\left( e^{\lambda_{\kappa}+\mu_{\kappa}}\left(\mu_{\kappa}'y\zeta-3\zeta-y\zeta' \right) \right) '+ye^{\lambda_{\kappa}+\mu_{\kappa}}\mu_{\kappa}'\zeta'\\
	=&e^{\lambda_{\kappa}+\mu_{\kappa}}[(\lambda_{\kappa}'+\mu_{\kappa}')(\mu_{\kappa}'y-3)+(\mu_{\kappa}''y+\mu_{\kappa}')]\zeta+e^{\lambda_{\kappa}+\mu_{\kappa}}[y\mu_{\kappa}'-y\lambda_{\kappa}'-4]\zeta'-ye^{\lambda_{\kappa}+\mu_{\kappa}}\zeta''\\
	=&-e^{\lambda_{\kappa}+\mu_{\kappa}}[e^{\mu_{\kappa}-\lambda_{\kappa}}y^{-2}(e^{\lambda_{\kappa}-\mu_{\kappa}}y^{3}\zeta)']'+e^{\lambda_{\kappa}+\mu_{\kappa}}[(\lambda_{\kappa}'+\mu_{\kappa}')(\mu_{\kappa}'y-3)+(\mu_{\kappa}''y+\mu_{\kappa}')]\zeta\\
	&+e^{\lambda_{\kappa}+\mu_{\kappa}}y(\lambda_{\kappa}''-\mu_{\kappa}'')\zeta+e^{\lambda_{\kappa}+\mu_{\kappa}}(\lambda_{\kappa}'-\mu_{\kappa}')\zeta.
	\end{split}
	\end{align}

	In this paper, we say the system is linearly unstable to mean that the linearised equation admits an growing mode solution of the form \(\zeta(y,t)=e^{at}\chi(y)\) with \(a>0\). Otherwise we call the system linearly stable. By \(\eqref{linear eq},\eqref{linear boundary condition},\eqref{key eq1}\) and \(\eqref{key eq2}\), we give the linearized equation about the perturbation variable \(\chi(y)\), which is described in the following proposition.
	\begin{proposition}	\label{3.2..}
		Given steady state solution \(n_{\kappa},\rho_{\kappa},p_{\kappa},\lambda_{\kappa},\mu_{\kappa},\Psi_{\kappa}\) induced by \(y_{\kappa}\). If the Lagrangian parameter perturbation has the form \(\zeta(y,t)=e^{at}\chi(y)\), then we have the linearized equation of the Einstein-Euler system \(\eqref{1field eq00}-\eqref{1momentum eq SS}\) through the perturbation variable \(\chi(y)\) with a Robin type boundary condition
		\begin{align}\label{linear eq chi}
		L\chi:=-\left[e^{\mu_{\kappa}-\lambda_{\kappa}}y^{-2}\left(e^{\lambda_{\kappa}-\mu_{\kappa}}y^{3}\chi \right) ' \right]'+ A_{1}\chi+A_{2}\chi+A_{3}\chi+A_{4}\chi+A_{5}\chi+A_{6}\chi=&-a^{2}ye^{2\lambda_{\kappa}-2\mu_{\kappa}}\chi\\ \label{linear boundary condition chi}
		(3-R_{\kappa}\mu_{\kappa}'(R_{\kappa}))\chi(R_{\kappa})+R_{\kappa}\chi'(R_{\kappa})=&0,
		\end{align}
		where
		\begin{align*}
		A_{1}\chi:=&y(\lambda_{\kappa}''-\mu_{\kappa}'')\chi\quad \quad\quad\quad\quad\quad\quad\quad\quad\quad
		A_{4}\chi:=-4\pi yp_{\kappa}e^{2\lambda_{\kappa}}(2y\mu_{\kappa}'+1)\chi\\
		A_{2}\chi:=&(\lambda_{\kappa}'-\mu_{\kappa}')\chi\quad\quad\quad\quad\quad\quad\quad\quad\quad\quad\quad A_{5}\chi:=-\frac{2y\mu_{\kappa}'+1}{y^{2}}e^{2\lambda_{\kappa}}m_{\kappa}\chi
		\\
		A_{3}\chi:=&[(\lambda_{\kappa}'+\mu_{\kappa}')(\mu_{\kappa}'y-3)+(\mu_{\kappa}''y+\mu_{\kappa}')]\chi\quad\quad
		A_{6}\chi:=8\pi yp_{\kappa}e^{2\lambda_{\kappa}}\chi.
		\end{align*}
	\end{proposition}
	According to the above result, the stability problem of the linearized system is reduced to the problem of the eigenvalues of the operator \(L\). If the operator \(L\) has a negative eigenvalue \(\nu_{\ast}<0\), then we can find a growing mode of the original linearized problem, so the system is unstable. Conversely, the system is stable. In the next section, we will give stability analysis by developing the necessary functional tools.
	%%%%%%%%%%%%%%%%%%
	%%%%%%%%%%%%%%%%%%
	\section{Linear stability analysis}\label{sec4}
	Given \(\chi_{1},\chi_{2}\in C^{2}([0,R_{\kappa}])\) satisfying the boundary condition \(\eqref{linear boundary condition chi}\) and \(\left\langle \cdot\right\rangle \) represent the \(L^{2}\) inner product with weight \(e^{\lambda_{\kappa}-\mu_{\kappa}}y^{3}\), we have using integration by parts
	\begin{align*}
	\left\langle L\chi_{1},\chi_{2}\right\rangle =&\int_{0}^{R_{\kappa}}e^{\mu_{\kappa}-\lambda_{\kappa}}y^{-2}\left(e^{\lambda_{\kappa}-\mu_{\kappa}}y^{3}\chi_{1} \right) '\left(e^{\lambda_{\kappa}-\mu_{\kappa}}y^{3}\chi_{2} \right) '+(A_{1}+\cdots+A_{6})\chi_{1}e^{\lambda_{\kappa}-\mu_{\kappa}}y^{3}\chi_{2}dy\\
	&-\left[ y\left(e^{\lambda_{\kappa}-\mu_{\kappa}}y^{3}\chi_{1} \right) '\chi_{2}\right] (R_{\kappa})\\
	=&-\int_{0}^{R_{\kappa}}\left[e^{\mu_{\kappa}-\lambda_{\kappa}}y^{-2}\left(e^{\lambda_{\kappa}-\mu_{\kappa}}y^{3}\chi_{2} \right) ' \right]'e^{\lambda_{\kappa}-\mu_{\kappa}}y^{3}\chi_{1}+(A_{1}+\cdots+A_{6})\chi_{2}e^{\lambda_{\kappa}-\mu_{\kappa}}y^{3}\chi_{1}dy\\&-\left[ y\left(e^{\lambda_{\kappa}-\mu_{\kappa}}y^{3}\chi_{1} \right) '\chi_{2}\right] (R_{\kappa})+\left[ y\left(e^{\lambda_{\kappa}-\mu_{\kappa}}y^{3}\chi_{2} \right) '\chi_{1}\right] (R_{\kappa})\\
	=&\left\langle L\chi_{2},\chi_{1}\right\rangle+R_{\kappa}^{4}e^{\lambda_{\kappa}-\mu_{\kappa}}\left( \chi_{2}'(R_{\kappa})\chi_{1}(R_{\kappa})-\chi_{1}'(R_{\kappa})\chi_{2}(R_{\kappa})\right) \\
	=&\left\langle L\chi_{2},\chi_{1}\right\rangle
	\end{align*}
	where have used the boundary condition \(\eqref{linear boundary condition chi}\). Therefore the operator \(L\) is symmetric under the inner product \(\left\langle \cdot\right\rangle\) and we note that in particular
	\begin{align}\label{Lchi,chi}
	\left\langle L\chi,\chi\right\rangle=\int_{0}^{R_{\kappa}}e^{\mu_{\kappa}-\lambda_{\kappa}}y^{-2}\left( \left(e^{\lambda_{\kappa}-\mu_{\kappa}}y^{3}\chi \right) '\right) ^{2}+(A_{1}+\cdots+A_{6})e^{\lambda_{\kappa}-\mu_{\kappa}}y^{3}\chi^{2}dy-\left[ y\left(e^{\lambda_{\kappa}-\mu_{\kappa}}y^{3}\chi \right) '\chi\right] (R_{\kappa}).
	\end{align}
	Making use of the symmetry of the operator \(L\), we give a crucial criterion for finding the smallest eigenvalue \(\nu\), which is described in the following lemma.
	\begin{lemma}\label{4.1..}
		Let \(H^{1}_{r}(B_{R}(\mathbb{R}^{5}))\) denote the subspace of spherically symmetric functions in \(H^{1}_{0}(B_{R}(\mathbb{R}^{5}))\). We consider functions in \(H^{1}_{r}(B_{R}(\mathbb{R}^{5}))\) to be functions defined by one radial variable \(y\in[0,R_{\kappa}]\) and supported in the interval \([0,R_{\kappa})\). In this space we have
		\begin{align*}
		\inf_{\Arrowvert\chi\Arrowvert_{y^{4}e^{3\lambda_{\kappa}-3\mu_{\kappa}}}=1}\left\langle L\chi,\chi\right\rangle=:\nu_{\ast}=\inf[\nu:\exists\chi\ne 0 s.t. L\chi=\nu e^{2\lambda_{\kappa}-2\mu_{\kappa}}y\chi]
		\end{align*}
		where \(\Arrowvert\chi\Arrowvert^{2}_{\omega}=\left\langle\chi,\omega\chi \right\rangle_{L^{2}([0,R_{\kappa}])} \). Moreover, the infimum is attained by some \(\chi_{\ast}\), which is an eigenfunction of \(L\) with eigenvalue \(\nu_{\ast}\).
	\end{lemma}
	\begin{proof}
		It is clear that
		\begin{align*}
		\nu_{\ast}=\inf_{\Arrowvert\chi\Arrowvert_{y^{4}e^{3\lambda_{\kappa}-3\mu_{\kappa}}}=1}\left\langle L\chi,\chi\right\rangle\le\inf[\nu:\exists\chi\ne 0 s.t. L\chi=\nu e^{2\lambda_{\kappa}-2\mu_{\kappa}}y\chi].
		\end{align*}
		In order to prove equality, it suffice then to prove that \(\nu_{\ast}\) is an eigenvalue of \(L\). Pick \(\chi_{n}\) with \(\Arrowvert\chi_{n}\Arrowvert_{y^{4}e^{3\lambda_{\kappa}-3\mu_{\kappa}}}=1\) such that \(\left\langle L\chi_{n},\chi_{n}\right\rangle\to\inf_{\Arrowvert\chi\Arrowvert_{y^{4}e^{3\lambda_{\kappa}-3\mu_{\kappa}}}=1}\left\langle L\chi,\chi\right\rangle\). It is straightforward to compute
		\begin{align}\label{lamda mu}
		\begin{split}
		&\lambda_{\kappa}'=e^{2\lambda_{\kappa}}\left( 4\pi y\rho_{\kappa}-\frac{m_{\kappa}}{y^{2}}\right)\quad  \lambda_{\kappa}''=2e^{4\lambda_{\kappa}}\left( 4\pi y\rho_{\kappa}-\frac{m_{\kappa}}{y^{2}}\right) ^{2}-12\pi e^{4\lambda_{\kappa}}\left( \frac{m_{\kappa}}{y}+4\pi y^{2}p_{\kappa}\right)\left( \rho_{\kappa}+p_{\kappa} \right) +2e^{2\lambda_{\kappa}}\frac{m_{\kappa}}{y^{3}}\\
		&\mu_{\kappa}'=e^{2\lambda_{\kappa}}\left( 4\pi yp_{\kappa}+\frac{m_{\kappa}}{y^{2}}\right)\quad  \mu_{\kappa}''=8\pi p_{\kappa}e^{2\lambda_{\kappa}}-e^{2\lambda_{\kappa}}\left(4\pi y(p_{\kappa}-\rho_{\kappa})+\frac{2m_{\kappa}}{y^{2}} \right) \left(e^{2\lambda_{\kappa}}\left( 4\pi yp_{\kappa}+\frac{m_{\kappa}}{y^{2}}\right) +\frac{1}{y} \right) .
		\end{split}
		\end{align}
		Therefore we have
		\begin{align}\label{bound}
		\left| \rho_{\kappa},p_{\kappa},\lambda_{\kappa},\mu_{\kappa},m_{\kappa}/y^{3}\right| \le C\quad \Longrightarrow \quad \left| \lambda_{\kappa}'',\mu_{\kappa}'',\lambda_{\kappa}'/y,\mu_{\kappa}'/y\right| \le C, \quad y\in[0,R_{\kappa}].
		\end{align} 
		This leads to
		\begin{align}\label{11111}
		\int_{0}^{R_{\kappa}}(A_{1}+\cdots+A_{6})e^{\lambda_{\kappa}-\mu_{\kappa}}y^{3}\chi_{n}^{2}dy\lesssim\int_{0}^{R_{\kappa}}y^{4}\chi_{n}^{2}dy\lesssim\Arrowvert\chi_{n}\Arrowvert_{y^{4}e^{3\lambda_{\kappa}-3\mu_{\kappa}}}=1.
		\end{align}
		It is obvious that the last term in \(\eqref{Lchi,chi}\) is free by the fact that \(\chi_{n}\) are supported in the interval \([0,R_{\kappa})\). Then since \(\eqref{11111}\) holds and the first term in \(\eqref{Lchi,chi}\) are nonnegative, \(\inf_{\Arrowvert\chi\Arrowvert_{y^{4}e^{3\lambda_{\kappa}-3\mu_{\kappa}}}=1}\left\langle L\chi,\chi\right\rangle\) is finite. Then we consider the main part of the functional \(\left\langle L\chi,\chi\right\rangle\), that is the first term
		\begin{align}\label{1122}
		\begin{split}
		&\int_{0}^{R_{\kappa}}e^{\mu_{\kappa}-\lambda_{\kappa}}y^{-2}\left( \left(e^{\lambda_{\kappa}-\mu_{\kappa}}y^{3}\chi \right) '\right) ^{2}dy\\
		=&\int_{0}^{R_{\kappa}}e^{\mu_{\kappa}-\lambda_{\kappa}}y^{-2}\left(\left( e^{\lambda_{\kappa}-\mu_{\kappa}}y^{3}\right)'\chi+ e^{\lambda_{\kappa}-\mu_{\kappa}}y^{3}\chi'  \right) ^{2}dy\\
		=&\int_{0}^{R_{\kappa}}e^{\mu_{\kappa}-\lambda_{\kappa}}y^{-2}\left( \left( \left( e^{\lambda_{\kappa}-\mu_{\kappa}}y^{3}\right)'\right) ^{2}  \chi^{2}+2\left( e^{\lambda_{\kappa}-\mu_{\kappa}}y^{3}\right)' e^{\lambda_{\kappa}-\mu_{\kappa}}y^{3}\chi\chi'+\left(e^{\lambda_{\kappa}-\mu_{\kappa}}y^{3} \right) ^{2}\left( \chi'\right) ^{2}           \right) dy\\
		=&\int_{0}^{R_{\kappa}}e^{\mu_{\kappa}-\lambda_{\kappa}}y^{-2}\left( \left(e^{\lambda_{\kappa}-\mu_{\kappa}}\left(\lambda_{\kappa}'-\mu_{\kappa}' \right) y^{3} \right) ^{2}   +\left(3e^{\lambda_{\kappa}-\mu_{\kappa}}y^{2} \right)^{2}   +6\left( e^{\lambda_{\kappa}-\mu_{\kappa}}\right)  ^{2}  \left(\lambda_{\kappa}'-\mu_{\kappa}' \right)y^{5}        \right) \chi^{2}dy\\
		&+\int_{0}^{R_{\kappa}}2\left(  e^{\lambda_{\kappa}-\mu_{\kappa}}\left(\lambda_{\kappa}'-\mu_{\kappa}' \right) y^{3}+3 e^{\lambda_{\kappa}-\mu_{\kappa}}y^{2}       \right) y\chi\chi'dy  +\int_{0}^{R_{\kappa}}e^{\lambda_{\kappa}-\mu_{\kappa}}y^{4}(\chi')^{2}dy.
		\end{split}
		\end{align}
		Using integration by parts, we have
		\begin{align*}
		\int_{0}^{R_{\kappa}}2e^{\lambda_{\kappa}-\mu_{\kappa}}y^{3}\chi\chi'dy=e^{\lambda_{\kappa}-\mu_{\kappa}}R_{\kappa}^{3}\chi^{2}(R_{\kappa})-\int_{0}^{R_{\kappa}}e^{\lambda_{\kappa}-\mu_{\kappa}}\left(\lambda_{\kappa}'-\mu_{\kappa}' \right) y^{3}\chi^{2}dy
		-\int_{0}^{R_{\kappa}}3e^{\lambda_{\kappa}-\mu_{\kappa}}y^{2}\chi^{2}dy.
		\end{align*}
		Substituting the above identity into \(\eqref{1122}\), we get
		\begin{align}\label{2323}
		\begin{split}
		&\int_{0}^{R_{\kappa}}e^{\mu_{\kappa}-\lambda_{\kappa}}y^{-2}\left( \left(e^{\lambda_{\kappa}-\mu_{\kappa}}y^{3}\chi \right) '\right) ^{2}dy\\
		=&\int_{0}^{R_{\kappa}}e^{\lambda_{\kappa}-\mu_{\kappa}}y^{4}(\chi')^{2}dy+\int_{0}^{R_{\kappa}}2e^{\lambda_{\kappa}-\mu_{\kappa}}\left( \lambda_{\kappa}'-\mu_{\kappa}'\right)y^{4}\chi\chi'dy+3e^{\lambda_{\kappa}-\mu_{\kappa}}R_{\kappa}^{3}\chi^{2}(R_{\kappa})\\
		&+\int_{0}^{R_{\kappa}}e^{\lambda_{\kappa}-\mu_{\kappa}}\left(     \left( \lambda_{\kappa}'-\mu_{\kappa}'\right)^{2}y^{4}+3 \left( \lambda_{\kappa}'-\mu_{\kappa}'\right)y^{3}                   \right) \chi^{2}dy.\\
		=&\int_{0}^{R_{\kappa}}e^{\lambda_{\kappa}-\mu_{\kappa}}y^{4}(\chi')^{2}dy+\int_{0}^{R_{\kappa}}2e^{\lambda_{\kappa}-\mu_{\kappa}}\left( \lambda_{\kappa}'-\mu_{\kappa}'\right)y^{4}\chi\chi'dy+\int_{0}^{R_{\kappa}}e^{\lambda_{\kappa}-\mu_{\kappa}}\left(     \left( \lambda_{\kappa}'-\mu_{\kappa}'\right)^{2}y^{4}+3 \left( \lambda_{\kappa}'-\mu_{\kappa}'\right)y^{3}                   \right) \chi^{2}dy
		\end{split}
		\end{align}
		Now
		\begin{align*}
		\Arrowvert\chi_{n}'\Arrowvert^{2}_{L^{2}(B_{R}(\mathbb{R}^{5}))}\lesssim&\int_{0}^{R_{\kappa}}e^{\lambda_{\kappa}-\mu_{\kappa}}y^{4}(\chi_{n}')^{2}dy\\
		=&\left| \left\langle L\chi_{n},\chi_{n}\right\rangle\right| -\int_{0}^{R_{\kappa}}(A_{1}+\cdots+A_{6})e^{\lambda_{\kappa}-\mu_{\kappa}}y^{3}\chi_{n}^{2}dy-\int_{0}^{R_{\kappa}}2e^{\lambda_{\kappa}-\mu_{\kappa}}\left( \lambda_{\kappa}'-\mu_{\kappa}'\right)y^{4}\chi_{n}\chi_{n}'dy\\
		&-\int_{0}^{R_{\kappa}}e^{\lambda_{\kappa}-\mu_{\kappa}}\left(     \left( \lambda_{\kappa}'-\mu_{\kappa}'\right)^{2}y^{4}+3 \left( \lambda_{\kappa}'-\mu_{\kappa}'\right)y^{3}                   \right) \chi_{n}^{2}dy\\
		\lesssim&\left| \left\langle L\chi_{n},\chi_{n}\right\rangle\right|+\Arrowvert\chi_{n}\Arrowvert_{y^{4}e^{3\lambda_{\kappa}-3\mu_{\kappa}}}+C(\delta)\int_{0}^{R_{\kappa}}y^{4}\chi_{n}^{2}dy+\delta\int_{0}^{R_{\kappa}}y^{4}(\chi_{n}')^{2}dy.
		\end{align*}
		The last term in the above inequality can be absorbed on the left-hand side if \(\delta\) is chosen sufficiently small (an absolute constant). Then we have
		\begin{align*}
		\Arrowvert\chi_{n}'\Arrowvert^{2}_{L^{2}(B_{R}(\mathbb{R}^{5}))}\lesssim\int_{0}^{R_{\kappa}}e^{\lambda_{\kappa}-\mu_{\kappa}}y^{4}(\chi_{n}')^{2}dy\lesssim\left| \left\langle L\chi_{n},\chi_{n}\right\rangle\right|+\Arrowvert\chi_{n}\Arrowvert_{y^{4}e^{3\lambda_{\kappa}-3\mu_{\kappa}}}\le C,
		\end{align*}
		and obviously
		\begin{align*}
		\Arrowvert\chi_{n}\Arrowvert_{L^{2}(B_{R}(\mathbb{R}^{5}))}\lesssim\Arrowvert\chi_{n}\Arrowvert_{y^{4}e^{3\lambda_{\kappa}-3\mu_{\kappa}}}\le C.
		\end{align*}
		Hence \(\chi_{n}\) is bounded in \(H^{1}(B_{R}(\mathbb{R}^{5}))\) and there exists an appropriate subsequence \(\chi_{n'}\) that converge weakly to some \(\chi_{\ast}\in H^{1}(B_{R}(\mathbb{R}^{5}))\). By the Rellich-Kondrachov theorem, \(\chi_{n'}\to\chi_{\ast}\) in \(L^{2}(B_{R}(\mathbb{R}^{5}))\). It follows that \(\Arrowvert\chi_{\ast}\Arrowvert_{y^{4}e^{3\lambda_{\kappa}-3\mu_{\kappa}}}=1\). By the lower semi-continuity of weak convergence, we have \(\lim\inf\Arrowvert\chi_{n'}\Arrowvert_{H^{1}(B_{R}(\mathbb{R}^{5}))}\ge\Arrowvert\chi_{\ast}\Arrowvert_{H^{1}(B_{R}(\mathbb{R}^{5}))}\). Since \(\Arrowvert\chi_{n'}\Arrowvert_{L^{2}(B_{R}(\mathbb{R}^{5}))}\to\Arrowvert\chi_{\ast}\Arrowvert_{L^{2}(B_{R}(\mathbb{R}^{5}))}\), we must have
		\begin{align*}
		\lim\inf\Arrowvert\chi_{n'}'\Arrowvert^{2}_{L^{2}(B_{R}(\mathbb{R}^{5}))}\ge \Arrowvert\chi_{\ast}'\Arrowvert^{2}_{L^{2}(B_{R}(\mathbb{R}^{5}))}.
		\end{align*}
		Since \(\Arrowvert\cdot\Arrowvert_{y^{4}e^{\lambda_{\kappa}-\mu_{\kappa}}}\) is an equivalent norm for \(L^{2}(B_{R}(\mathbb{R}^{5}))\), we have
		\begin{align*}
		\lim\inf\int_{0}^{R_{\kappa}}y^{4}e^{\lambda_{\kappa}-\mu_{\kappa}}(\chi_{n'}')^{2}dy\ge\int_{0}^{R_{\kappa}}y^{4}e^{\lambda_{\kappa}-\mu_{\kappa}}(\chi_{\ast}')^{2}dy
		\end{align*}
		It follows that \(\left\langle L\chi_{\ast},\chi_{\ast}\right\rangle\le \inf_{\Arrowvert\chi\Arrowvert_{y^{4}e^{3\lambda_{\kappa}-3\mu_{\kappa}}}=1}\left\langle L\chi,\chi\right\rangle\), and that means we must have equality and the infimum is attained. Finally we show \(\chi_{\ast}\) is in fact an eigenfunction of \(L\). For any function \(h\in H^{1}_{r}(B_{R}(\mathbb{R}^{5}))\), we have
		\begin{align*}
		0=& \frac{d}{d\epsilon}\left(\frac{\left\langle L(\chi_{\ast}+\epsilon h),\chi_{\ast}+\epsilon h\right\rangle}{\left\langle\chi_{\ast}+\epsilon h,\chi_{\ast}+\epsilon h \right\rangle_{y^{4}e^{3\lambda_{\kappa}-3\mu_{\kappa}}} } \right)_{\epsilon=0}=\frac{d}{d\epsilon}\left( \frac{\left\langle L\chi_{\ast},\chi
			\right\rangle +2\epsilon\left\langle L\chi_{\ast},h \right\rangle+\epsilon^{2} \left\langle Lh,h \right\rangle }{\left\langle\chi_{\ast},\chi_{\ast} \right\rangle_{y^{4}e^{3\lambda_{\kappa}-3\mu_{\kappa}}}+2\epsilon\left\langle \chi_{\ast},h\right\rangle _{y^{4}e^{3\lambda_{\kappa}-3\mu_{\kappa}}}+\epsilon^{2}\left\langle h,h\right\rangle _{y^{4}e^{3\lambda_{\kappa}-3\mu_{\kappa}}} }\right) _{\epsilon=0}\\
		=&\frac{2\left\langle L\chi_{\ast},h \right\rangle }{\left\langle \chi_{\ast},\chi_{\ast}\right\rangle_{y^{4}e^{3\lambda_{\kappa}-3\mu_{\kappa}}} }-\frac{2\left\langle L\chi_{\ast},\chi_{\ast}\right\rangle\left\langle\chi_{\ast},h \right\rangle_{y^{4}e^{3\lambda_{\kappa}-3\mu_{\kappa}}}  }{\left\langle\chi_{\ast},\chi_{\ast} \right\rangle _{y^{4}e^{3\lambda_{\kappa}-3\mu_{\kappa}}}^{2}},
		\end{align*}
		so \(\left\langle L\chi_{\ast},h \right\rangle=\nu_{\ast}\left\langle\chi_{\ast},h \right\rangle_{y^{4}e^{3\lambda_{\kappa}-3\mu_{\kappa}}}\). Hence \(\chi_{\ast}\) is a weak solution to \(L\chi=\nu_{\ast}e^{2\lambda_{\kappa}-2\mu_{\kappa}}y\chi\). By elliptic regularity, \(\chi_{\ast}\) is smooth on \((0,R_{\kappa}]\), and so the weak solution is in fact a classical solution. Therefore \(\chi_{\ast}\) is in fact an eigenfunction of \(L\) with eigenvalue \(\nu_{\ast}\), which completes the proof of Lemma \(\ref{4.1..}\).
	\end{proof}
	\begin{corollary}\label{4.2..}
		If \(\left\langle L\chi,\chi \right\rangle\ge 0 \) for any \(\chi\) satisfying the boundary condition \(\eqref{linear boundary condition chi}\), then the corresponding Einstein-Euler system is linearly stable under radial perturbations. Conversely, if there exist \(\chi\in H^{1}_{r}(B_{R}(\mathbb{R}^{5}))\) such that \(\left\langle L\chi,\chi \right\rangle< 0\), then it must be linearly unstable.
	\end{corollary}
	\begin{proof}
		If there exist \(\chi\in H^{1}_{r}(B_{R}(\mathbb{R}^{5}))\) such that \(\left\langle L\chi,\chi \right\rangle< 0\), then by the Lemma \(\ref{4.1..}\) there exist \(\nu_{\ast}<0\) and \(\chi_{\ast}\in H^{1}_{r}(B_{R}(\mathbb{R}^{5}))\) such that \(L{\chi_{\ast}}=\nu_{\ast}e^{2\lambda_{\kappa}-2\mu_{\kappa}}y\chi_{\ast}\). This, by the proposition \(\ref{3.2..}\), means the linearized system admits a solution of the form \(\zeta(y,t)=e^{\sqrt{-\nu_{\ast}}t}\chi_{\ast}(y)\). This
		grows exponentially in time, and hence the corresponding Einstein-Euler system is linearly unstable. Conversely, if \(\left\langle L\chi,\chi \right\rangle\ge 0\) for any \(\chi\) satisfying the boundary condition \(\eqref{linear boundary condition chi}\), then no such growing solutions exist and hence the corresponding Einstein-Euler system is linearly stable under radial perturbations.
	\end{proof}
	Now we prove our main conclusions on the linear (in)stability results for Einstein-Euler system. This will be split into two propositions below.
	\begin{proposition}\label{prop: linear stable}
		For \(\kappa>0\) sufficiently small, the operator $L$ defined in Proposition \ref{3.2..} does not admit any negative eigenvalues. Therefore the corresponding Einstein-Euler system is linearly stable under radial perturbations.
	\end{proposition}
	\begin{proof}
		Let \(y=R_{\kappa}z\) and \(\tilde{\chi}(z)=\chi(R_{\kappa}z)\). Note that \(\lambda_{\kappa}\to 0,\rho_{\kappa}\to 1,p_{\kappa}\to 0\) as \(\kappa\to 0^{+}\). Hence the constant in \(\eqref{bound}\) is independent of \(\kappa\), when \(\kappa>0\) is sufficiently small. Then we have the following estimate
		\begin{align}\label{3}
		\int_{0}^{R_{\kappa}}(A_{1}+\cdots+A_{6})e^{\lambda_{\kappa}-\mu_{\kappa}}y^{3}\chi^{2}dy\lesssim\int_{0}^{R_{\kappa}}y^{4}\chi^{2}dy,
		\end{align}
		where the implicit constant is independent of \(\kappa\).  Using Cauchy-Schwarz inequality, we get
		\begin{align}\label{4}
		\int_{0}^{R_{\kappa}}2e^{\lambda_{\kappa}-\mu_{\kappa}}\left( \lambda_{\kappa}'-\mu_{\kappa}'\right)y^{4}\chi\chi'dy\lesssim C(\delta)\int_{0}^{R_{\kappa}}y^{4}\chi^{2}dy+\delta\int_{0}^{R_{\kappa}}y^{4}(\chi')^{2}dy,
		\end{align}
		where \(\delta\) is independent of \(\kappa\) and chosen sufficiently small. For the boundary term, we have
		\begin{align*}
		-\left[ y\left(e^{\lambda_{\kappa}-\mu_{\kappa}}y^{3}\chi \right) '\chi\right] (R_{\kappa})=&-R_{\kappa}e^{\lambda_{\kappa}-\mu_{\kappa}}\left((\lambda_{\kappa}'-\mu_{\kappa}')R_{\kappa}^{3}\chi+3R_{\kappa}^{2}\chi+R_{\kappa}^{3}\chi'(R_{\kappa}) \right)\chi(R_{\kappa}) \\
		=&-R_{\kappa}^{4}e^{\lambda_{\kappa}-\mu_{\kappa}}\lambda_{\kappa}'\chi^{2}(R_{\kappa}).
		\end{align*}
		According to \(\eqref{2323},\eqref{3}\) and \(\eqref{4}\), the functional \(\eqref{Lchi,chi}\) has a lower bound
		\begin{align*}
		\left\langle L\chi,\chi\right\rangle=&\int_{0}^{R_{\kappa}}e^{\mu_{\kappa}-\lambda_{\kappa}}y^{-2}\left( \left(e^{\lambda_{\kappa}-\mu_{\kappa}}y^{3}\chi \right) '\right) ^{2}+(A_{1}+\cdots+A_{6})e^{\lambda_{\kappa}-\mu_{\kappa}}y^{3}\chi^{2}dy-\left[ y\left(e^{\lambda_{\kappa}-\mu_{\kappa}}y^{3}\chi \right) '\chi\right] (R_{\kappa})\\
		\gtrsim&\int_{0}^{R_{\kappa}}y^{4}(\chi')^{2}dy-C_{l}\int_{0}^{R_{\kappa}}y^{4}\chi^{2}dy+R_{\kappa}^{3}\chi^{2}(R_{\kappa})-C_{l}R_{\kappa}^{5}\chi^{2}(R_{\kappa}),
		\end{align*}
		where the constants \(C_{l}>0\) and are independent of \(\kappa\). From the decay estimates \(\eqref{decay estimate}\), it is obvious that \(R_{\kappa}\to 0\) as \(\kappa\to 0^{+}\). Now we consider the Poincare-Hardy-type inequality proved in \cite{Lam}.
		\begin{align*}
		\int_{0}^{1}z^{4}\left| v(z)\right|^{2} dz\lesssim\int_{0}^{1}z^{4}\left| v'(z)\right|^{2} dz+\left|v(1) \right| ^{2}\quad\quad \textrm{for all}\quad\quad v\in C^{1}([0,1]).
		\end{align*}
		The previous estimate tells us that
		\begin{align*}
		\left\langle L\chi,\chi\right\rangle\gtrsim&\int_{0}^{R_{\kappa}}y^{4}(\chi')^{2}dy-C_{l}\int_{0}^{R_{\kappa}}y^{4}\chi^{2}dy+R_{\kappa}^{3}\chi^{2}(R_{\kappa})-C_{l}R_{\kappa}^{5}\chi^{2}(R_{\kappa})\\
		\gtrsim&R_{\kappa}^{3}\int_{0}^{1}z^{4}(\tilde{\chi}')^{2}dz-C_{l}R_{\kappa}^{5}\int_{0}^{1}z^{4}(\tilde{\chi})^{2}dz+R_{\kappa}^{3}\tilde{\chi}^{2}(1)-C_{l}R_{\kappa}^{5}\tilde{\chi}^{2}(1)\\
		=&R_{\kappa}^{3}\left( \int_{0}^{1}z^{4}(\tilde{\chi}')^{2}dz+\tilde{\chi}^{2}(1)-C_{l}R_{\kappa}^{2}\int_{0}^{1}z^{4}(\tilde{\chi})^{2}dz-C_{l}R_{\kappa}^{2}\tilde{\chi}^{2}(1)\right),
		\end{align*}
		and hence, for small enough central redshift \(\kappa>0\), we have
		\begin{align*}
		\left\langle L\chi,\chi\right\rangle\ge 0\quad\quad \forall\chi\in H^{1}_{r}(B_{R}(\bbR^{5})).
		\end{align*}
		So the linear stability follows.
	\end{proof}
	Finally, it remains to prove linear instability for Einstein-Euler system of large central redshift.
	\begin{proposition}\label{prop: linear unstable}
		For \(\kappa\) sufficiently large, the operator $L$ defined in Proposition \ref{3.2..} admits a negative eigenvalue. Therefore the corresponding Einstein-Euler system is linearly unstable.
	\end{proposition}
	\begin{proof}
		Using Corollary \(\ref{4.2..}\), we just need to show that there exist \(\chi_{\kappa}\in H^{1}_{r}(B_{R}(\mathbb{R}^{5}))\) such that \(\left\langle L\chi_{\kappa},\chi_{\kappa} \right\rangle< 0\), when \(\kappa\) is sufficiently large. We localize the perturbation \(\chi_{\kappa}(y)\) on the interval \([r_{\kappa}^{1},r_{\kappa}^{2}]\) by setting \(\chi_{\kappa}(y)=y^{b}\xi_{\kappa}(y)\) for some \(b\in\mathbb{R}\) to be specified later, where the smooth cut-off function \(R(\xi_{\kappa})\subset[0,1]\) is supported in the interval \([r_{\kappa}^{1},r_{\kappa}^{2}]\) and identically equal to 1 on the interval \([2r_{\kappa}^{1},r^{2}_{\kappa}/2]\). Note that the latter interval is non-trivial for \(\kappa\) sufficiently large. In addition, we require that
		\begin{align*}
		|\xi_{\kappa}'(y)|\le\frac{4}{r_{\kappa}^{1}},\ y\in[r_{\kappa}^{1},2r_{\kappa}^{1}],\quad\quad |\xi_{\kappa}'(y)|\le\frac{4}{r_{\kappa}^{2}},\ y\in[r_{\kappa}^{2}/2,r_{\kappa}^{2}].
		\end{align*}
		Then we deal with \(\eqref{Lchi,chi}\) splited into two parts \(\left\langle L\chi_{\kappa},\chi_{\kappa} \right\rangle=B_{1}+B_{2}\), where
		\begin{align*}
		B_{1}=&\int_{r_{\kappa}^{1}}^{r_{\kappa}^{2}}e^{\mu_{\kappa}-\lambda_{\kappa}}y^{-2}\left[\left( e^{\lambda_{\kappa}-\mu_{\kappa}}y^{3+b}\right) ' \right] ^{2}\xi_{\kappa}^{2}dy+\int_{r_{\kappa}^{1}}^{r_{\kappa}^{2}}e^{\lambda_{\kappa}-\mu_{\kappa}}y^{3+2b}(A_{1}+\cdots+A_{6})\xi_{\kappa}^{2}dy\\
		B_{2}=&\left( \int_{r_{\kappa}^{1}}^{2r_{\kappa}^{1}}+\int_{r_{\kappa}^{2}/2}^{r_{\kappa}^{2}}\right) \left[e^{\lambda_{\kappa}-\mu_{\kappa}}y^{4+2b}\xi_{\kappa}'+2y^{1+b}\left(e^{\lambda_{\kappa}-\mu_{\kappa}}y^{3+b} \right)'\xi_{\kappa}  \right]\xi_{\kappa}'dy. 
		\end{align*}
		We first compute 
		\begin{align*}
		B_{1}=&\int_{r_{\kappa}^{1}}^{r_{\kappa}^{2}}e^{\lambda_{\kappa}-\mu_{\kappa}}y^{2+2b}\xi_{\kappa}^{2}\left[  y^{2}(\lambda_{\kappa}'-\mu_{\kappa}')^{2}+2(3+b)y(\lambda_{\kappa}'-\mu_{\kappa}')+(3+b)^{2}+y^{2}(\lambda_{\kappa}''-\mu_{\kappa}'') \right]dy \\
		&+\int_{r_{\kappa}^{1}}^{r_{\kappa}^{2}}e^{\lambda_{\kappa}-\mu_{\kappa}}y^{2+2b}\xi_{\kappa}^{2}\left[y(\lambda_{\kappa}'-\mu_{\kappa}')+ y\left((\lambda_{\kappa}'+\mu_{\kappa}')(\mu_{\kappa}'y-3)+(\mu_{\kappa}''y+\mu_{\kappa}') \right)\right]dy\\
		&+ \int_{r_{\kappa}^{1}}^{r_{\kappa}^{2}}e^{\lambda_{\kappa}-\mu_{\kappa}}y^{2+2b}\xi_{\kappa}^{2}\left[ -\frac{2y\mu_{\kappa}'+1}{y}e^{2\lambda_{\kappa}}m_{\kappa} +8\pi y^{2}p_{\kappa}e^{2\lambda_{\kappa}}-4\pi y^{2}p_{\kappa}e^{2\lambda_{\kappa}}(2y\mu_{\kappa}'+1) \right] dy.
		\end{align*}
		We expect that the integral becomes proportional to \(y^{-1}\) and hence choose \(b=-1\). By the asymptotical properties \(\eqref{asy 1}\) and \(\eqref{asy 2}\), when \(\kappa\) sufficiently large we have
		\begin{align*}
		B_{1}\le\int_{r_{\kappa}^{1}}^{r_{\kappa}^{2}}e^{\lambda_{\kappa}-\mu_{\kappa}}\xi_{\kappa}^{2}\left[- 3\right] dy\le-CC_{\kappa}^{-1}\int_{2r_{\kappa}^{1}}^{r_{\kappa}^{2}/2}y^{-1}dy\le-C_{1}C_{\kappa}^{-1}\ln\kappa,
		\end{align*}
		where \(C_{1}>0\) is independent of \(\kappa\) and \(C_{\kappa}>0\) is the constant introduced in Lemma \(\ref{2.9..}\). Then we turn to consider the
		magnitude of \(B_{2}\). We split this term into two parts
		\begin{align*}
		B_{2}=\left( \int_{r_{\kappa}^{1}}^{2r_{\kappa}^{1}}+\int_{r_{\kappa}^{2}/2}^{r_{\kappa}^{2}}\right)f_{\kappa,1}dy+\left( \int_{r_{\kappa}^{1}}^{2r_{\kappa}^{1}}+\int_{r_{\kappa}^{2}/2}^{r_{\kappa}^{2}}\right)f_{\kappa,2}dy,
		\end{align*}
		where
		\begin{align*}
		f_{\kappa,1}=e^{\lambda_{\kappa}-\mu_{\kappa}}y^{4+2b}(\xi_{\kappa}')^{2}\le CC_{\kappa}^{-1}y\left| \xi_{\kappa}'\right| ^{2},
		\end{align*}
		and
		\begin{align*}
		f_{\kappa,2}=e^{\lambda_{\kappa}-\mu_{\kappa}}y^{3+2b}\left[(\lambda_{\kappa}'-\mu_{\kappa}')y+(3+b) \right] \xi_{\kappa}\xi_{\kappa}'\le CC_{\kappa}^{-1}\left| \xi_{\kappa}'\right| .
		\end{align*}
		We recall that on the interval \([r_{\kappa}^{1},2r_{\kappa}^{1}]\) the estimate \(|\xi_{\kappa}'|\le4/r_{\kappa}^{1}\) holds, and \(|\xi_{\kappa}'|\le 4/r_{\kappa}^{2}\) holds on \([r_{\kappa}^{2}/2,r_{\kappa}^{2}]\). Thus
		\begin{align*}
		\left|B_{2} \right| \le C_{2}C_{\kappa}^{-1},
		\end{align*}
		where \(C_{2}>0\) is independent of \(\kappa\). Combined with the previous results, we show that there exist \(\chi_{\kappa}(y)\in H^{1}_{r}(B_{R}(\mathbb{R}^{5}))\) such that
		\begin{align*}
		\left\langle L\chi_{\kappa},\chi_{\kappa}\right\rangle\le  -C_{1}C_{\kappa}^{-1}\ln\kappa+C_{2}C_{\kappa}^{-1}.
		\end{align*}
		When \(\kappa\) is sufficiently large, we have \(\left\langle L\chi_{\kappa},\chi_{\kappa}\right\rangle< 0\), which completes the proof.
	\end{proof}

	\section{Quasilinearization}\label{sec5}
We start by introducing the renormalized fluid velocity field \(\vec{V}\) and the enthalpy \(\sigma\) as
\begin{align}\label{vecV}
\vec{V}:=\sigma u,\quad\quad\quad \sigma^{2}:=\rho+p. 
\end{align}
Since \(u\) is future directed unit timelike, we have
\begin{align*}
\vec{V}=(V^{0},V,0,0)\quad\quad V^{0}=e^{-\mu}\sqrt{\sigma^{2}+e^{2\lambda}V^{2}}=:e^{-\mu}\left\langle V\right\rangle.
\end{align*}
Then \eqref{mass conservation}-\eqref{momentum conservation} reduce to the following equations in the fluid domain \(\calB(t)\)
\begin{align}\label{re momentum eq}
\nabla_{\vec{V}}\vec{V}+\frac{1}{2}\nabla \sigma^{2}=0,\\ \label{re mass eq}
\nabla_{\alpha}V^{\alpha}=0.
\end{align}
Let \(D_{\vec{V}}:=V^{0}\partial_{t}+V\partial_{r}\). The spherically symmetric Einstein-Euler system \(\eqref{1field eq00}-\eqref{1momentum eq SS}\) can take the following form:
\begin{align}\label{field eq00}
&e^{-2\lambda}(2r\lambda'-1)+1=8\pi r^{2}\left(\rho+e^{2\lambda}V^{2} \right),\\ \label{field eq11}
&e^{-2\lambda}(2r\mu'+1)-1=8\pi r^{2}\left(p+e^{2\lambda}V^{2} \right),\\ \label{field eq01}
&\dot{\lambda}=-4\pi r e^{\mu+2\lambda}\left\langle V\right\rangle V,\\ 	\label{momentum eq SS}
&D_{\vec{V}}V+e^{-2\lambda}\mu'\left\langle V\right\rangle^{2}-8\pi r e^{2\lambda}\left\langle V\right\rangle^{2}V^{2}+\lambda'V^{2}+\frac{1}{2}e^{-2\lambda}\partial_{r}\sigma^{2}=0,\\ 	\label{mass eq SS}
&\frac{D_{\vec{V}}\left\langle V\right\rangle}{\left\langle V\right\rangle}-\frac{\partial_{r}\left\langle V\right\rangle}{\left\langle V\right\rangle}V+\partial_{r}V-4\pi r e^{2\lambda}\left\langle V\right\rangle^{2}V+\left( \mu'+\lambda'+\frac{2}{r}\right)V=0 .
\end{align}
	We introduce the Lagrangian parametrization \(\eta(\cdot,t):[0,R_{\kappa}]\to [0,R(t)]\) that represents the radial position of the fluid particle at time \(t\) so that
	\begin{align}\label{eta def}
	\partial_{t}\eta=\frac{V}{V^{0}}\circ\eta\ \ \ \  \text{with} \ \ \ \ \eta(y,0)=\eta_{0}(y).
	\end{align}
	Here \(\eta_{0}\) is not necessarily the identity map but depend on the initial density profile. We express the enthalpy as the sum of the steady state and the perturbation
	\begin{align*}
	\sigma^{2}(r,t)=\sigma_{\kappa}^{2}(\eta^{-1}(r,t))+\varepsilon(r,t).
	\end{align*}
	By a slight abuse of notation, we often write \(\sigma_{\kappa}^{2}(r,t)\) instead of \(\sigma_{\kappa}^{2}(\eta^{-1}(r,t))\), therefore we have
	\begin{align*}
	D_{\vec{V}}\sigma_{\kappa}^{2}=0.
	\end{align*}
	Then we derive the quasilinear equations for the renormalized fluid velocity \(V\) and the enthalpy perturbation variable \(\varepsilon\). Applying the covariant derivative \(\nabla_{\vec{V}}\) to the equation \(\eqref{re momentum eq}\) yields
	\begin{align}\label{2.2}
	\nabla_{\vec{V}}^{2}V^{\alpha}-\frac{1}{2}\nabla_{\beta}\sigma^{2}\nabla^{\beta}V^{\alpha}+\frac{1}{2}\nabla^{\alpha}D_{\vec{V}}\sigma^{2}=0,
	\end{align}
	where we have used the identity \(\nabla_{\alpha}V_{\beta}=\nabla_{\beta}V_{\alpha}\) in spherical symmetry. Let \(n\) be the unit outward pointing (spacetime) normal to \(\partial\calB\). Since \(\sigma^{2}\equiv 1\) on \(\partial\calB\), \(\nabla\sigma^{2}\) is normal (with respect to \(g\)) to \(\partial\calB\). Going back to \(\eqref{2.2}\) and restricting it to the boundary we get
	\begin{align}\label{2.3}
	\nabla_{\vec{V}}^{2}V^{\alpha}+\frac{a}{2}\nabla_{n}V^{\alpha}+\frac{1}{2}\nabla^{\alpha}D_{\vec{V}}\sigma^{2}=0,
	\end{align}
	where 
	\begin{align*}
	a:=\sqrt{\nabla_{\alpha}\sigma^{2}\nabla^{\alpha}\sigma^{2}}.
	\end{align*}
	Using the connection coefficient, we can express the \(V\) component equation of \(\eqref{2.3}\) as
	\begin{align}\label{2.4}
	(D_{\vec{V}}^{2}+\frac{1}{2}aD_{n})V=-\frac{1}{2}\nabla^{1}D_{\vec{V}}\sigma^{2}-D_{\vec{V}}(\Gamma^{1}_{\alpha\beta}V^{\alpha}V^{\beta})-\Gamma^{1}_{\nu\gamma}V^{\nu}(D_{\vec{V}}V^{\gamma}+\Gamma^{\gamma}_{\alpha\beta}V^{\alpha}V^{\beta})+\frac{1}{2}\Gamma^{1}_{\alpha\beta}(\nabla^{\alpha}\sigma^{2})V^{\beta}.
	\end{align}
	We next turn to the interior wave equation for \(V\). Applying \(\nabla_{\beta}\) to the vanishing divergence equation \(\nabla_{\alpha}V^{\alpha}=0\), commuting \(\nabla_{\beta}\) and \(\nabla_{\alpha}\), and using \(\nabla_{\alpha}V_{\beta}=\nabla_{\beta}V_{\alpha}\), gives
	\begin{align}\label{2.5}
	0=\nabla_{\alpha}\nabla^{\alpha}V_{\beta}-R_{\beta\lambda}V^{\lambda}.
	\end{align}
	Taking trace of \(\eqref{einstein eq}\) we get
	\begin{align*}
	\bar{R}=8\pi(V_{\alpha}V^{\alpha}+2),
	\end{align*}
	and therefore in view of \(\eqref{energy tensor}\) and \(\eqref{einstein eq}\)
	\begin{align}
	\bar{R}_{\alpha\beta}=8\pi(V_{\alpha}V_{\beta}+\frac{1}{2}g_{\alpha\beta}).
	\end{align}
	Plugging back into \(\eqref{2.5}\) gives
	\begin{align}\label{2.7}
	\nabla_{\alpha}\nabla^{\alpha}V^{\beta}=8\pi(\frac{1}{2}-\sigma^{2})V^{\beta}.
	\end{align}
	Let \(\square\) denote the wave operator of \(g\). we express the \(V\) component equation of \(\eqref{2.7}\) as
	\begin{align}\label{2.8}
	\square V=8\pi(\frac{1}{2}-\sigma^{2})V-D_{\nu}(g^{\alpha\nu}\Gamma^{1}_{\alpha\beta}V^{\beta})-g^{\alpha\gamma}\Gamma^{\nu}_{\nu\gamma}\Gamma^{1}_{\alpha\beta}V^{\beta}-\Gamma^{1}_{\nu\gamma}(D^{\nu}V^{\gamma}+g^{\alpha\nu}\Gamma^{\gamma}_{\alpha\beta}V^{\beta}).
	\end{align}
	To complete the set of fluid equations we need to derive wave equations for \(\sigma^{2}\) and \(D_{\vec{V}}\sigma^{2}\) with Dirichlet
	boundary conditions. Applying \(\nabla_{\beta}\) to \(\eqref{re momentum eq}\) and \(\eqref{2.2}\) and using \(\nabla_{\alpha}V^{\alpha}=0,\nabla_{\alpha}V_{\beta}=\nabla_{\beta}V_{\alpha}\) yield
\begin{align}\label{2.9}
\square\sigma^{2}=8\pi(\sigma^{2}-2\sigma^{4})-2(\nabla_{\alpha}V_{\beta})(\nabla_{\beta}V_{\alpha}),
\end{align}
and
\begin{align}\label{2.10}
	\square D_{\vec{V}}\sigma^{2}=16\pi(D_{\vec{V}}\sigma^{2}-3\sigma^{2}D_{\vec{V}}\sigma^{2})+6(\nabla^{\alpha}V^{\beta})(\nabla_{\alpha}\nabla_{\beta}\sigma^{2})+4(\nabla^{\alpha}V^{\beta})(\nabla_{\alpha}V^{\lambda})(\nabla_{\lambda}V_{\beta})-4R_{\lambda\alpha\beta\nu}(\nabla^{\alpha}V^{\beta})V^{\lambda}V^{\nu}.
\end{align}
Using the connection coefficients associated with the metric \(\eqref{metric}\), the equations \(\eqref{2.4},\eqref{2.8},\eqref{2.9},\eqref{2.10}\) become 
\begin{align}
\begin{split}
	(D_{\vec{V}}^{2}+\frac{1}{2}aD_{n})V=&-\frac{1}{2}e^{-2\lambda}\partial_{r}D_{\vec{V}}\sigma^{2}-D_{\vec{V}}\left( e^{-2\lambda}\left\langle V\right\rangle^{2}\mu'-8\pi r e^{2\lambda}\left\langle V\right\rangle^{2}V^{2}+\lambda'V^{2}\right) \\
	&+\left( 4\pi r e^{2\lambda}\left\langle V\right\rangle V^{2}-e^{-2\lambda}\left\langle V\right\rangle\mu'\right) \left( D_{\vec{V}}\left\langle V\right\rangle+\mu'\left\langle V\right\rangle V-4\pi re^{4\lambda}\left\langle V\right\rangle V^{3}\right)\\
	&+\left( 4\pi r e^{2\lambda}\left\langle V\right\rangle^{2}V-\lambda' V\right) \left( D_{\vec{V}}V+e^{-2\lambda}\mu'\left\langle V\right\rangle^{2}-8\pi r e^{2\lambda}\left\langle V\right\rangle^{2}V^{2}+\lambda'V^{2}\right) \\
	&+\frac{1}{2}e^{-2\lambda}\left( \mu'+\lambda'\right) V \partial_{r}\sigma^{2}-4\pi r\left\langle V\right\rangle^{2}V \partial_{r}\sigma^{2}\quad \text{on}\ \partial\calB,
\end{split}
\end{align}
\begin{align}
\begin{split}
\square V=&8\pi(\frac{1}{2}-\sigma^{2})V+D_{\vec{V}}(e^{-2\lambda}\mu')-V\partial_{r}(e^{-2\lambda}\mu')+e^{-2\lambda}\mu'\frac{D_{\vec{V}}\left\langle V\right\rangle}{\left\langle V\right\rangle}-e^{-2\lambda}\mu'\frac{V\partial_{r}\left\langle V\right\rangle}{\left\langle V\right\rangle}\\
&-\frac{1}{\left\langle V\right\rangle}D_{\vec{V}}\left( 4\pi r e^{2\lambda}\left\langle V\right\rangle V^{2}\right) +\partial_{r}\left( 4\pi r \left\langle V\right\rangle^{2}V-e^{-2\lambda}\lambda' V\right) -4\pi r\left\langle V\right\rangle^{2}V+16\pi r^{2}e^{4\lambda}\left\langle V\right\rangle^{2}V^{3}\\
&+e^{-2\lambda}\mu'\left( \frac{D_{\vec{V}}\left\langle V\right\rangle}{\left\langle V\right\rangle}-\frac{V\partial_{r}\left\langle V\right\rangle}{\left\langle V\right\rangle}+\mu'V\right)-e^{-2\lambda}\lambda'\left( \partial_{r}V-4\pi r e^{2\lambda}\left\langle V\right\rangle^{2}V+\lambda'V\right) \\
& -8\pi re^{2\lambda}\left\langle V\right\rangle V\left( \frac{D_{\vec{V}}V}{\left\langle V\right\rangle}-\frac{V\partial_{r}V}{\left\langle V\right\rangle}+e^{-2\lambda}\mu'\left\langle V\right\rangle-4\pi re^{2\lambda}\left\langle V\right\rangle V^{2}\right) \\
&-e^{-2\lambda}\left( -4\pi r e^{2\lambda}\mu'\left\langle V\right\rangle^{2}V-4\pi re^{2\lambda}\lambda'\left\langle V\right\rangle^{2}V-8\pi e^{2\lambda}\left\langle V\right\rangle^{2}V+\mu'\lambda'V+(\lambda')^{2}V+\frac{2}{r}\lambda'V\right),
\end{split}
\end{align}
\begin{align}
\begin{split}
\square\varepsilon=&8\pi\left( \varepsilon-2\left( \varepsilon+\sigma_{\kappa}^{2}\right) ^{2}+2\sigma_{\kappa}^{4}\right) +4e^{-2\lambda}\mu'\left\langle V\right\rangle\partial_{r}\left\langle V\right\rangle+2\left( \frac{V}{r}\right) ^{2}-\frac{V^{2}}{\left\langle V\right\rangle^{2}}\partial_{r}\partial_{r}\sigma_{\kappa}^{2}-\frac{2V}{\left\langle V\right\rangle^{2}}(\partial_{r}V)(\partial_{r}\sigma_{\kappa}^{2})\\
&+\frac{1}{\left\langle V\right\rangle^{2}}(D_{\vec{V}}V)(\partial_{r}\sigma_{\kappa}^{2})-\mu'\frac{V^{2}}{\left\langle V\right\rangle^{2}}\partial_{r}\sigma_{\kappa}^{2}-\frac{VD_{\vec{V}}\left\langle V\right\rangle}{\left\langle V\right\rangle^{3}}\partial_{r}\sigma_{\kappa}^{2}-2\frac{V^{2}\partial_{r}\left\langle V\right\rangle}{\left\langle V\right\rangle^{3}}\partial_{r}\sigma_{\kappa}^{2}-4\pi r e^{2\lambda}V^{2}\partial_{r}\sigma_{\kappa}^{2}\\
&-4\left( \frac{D_{\vec{V}}V}{\left\langle V\right\rangle}-\frac{V\partial_{r}\left\langle V\right\rangle}{\left\langle V\right\rangle}+e^{-2\lambda}\mu'\left\langle V\right\rangle-4\pi re^{2\lambda}\left\langle V\right\rangle V^{2}\right) \left( \partial_{r}\left\langle V\right\rangle-4\pi re^{4\lambda}\left\langle V\right\rangle V^{2}\right) \\
&+e^{-2\lambda}(\mu'-\lambda')\partial_{r}\sigma_{\kappa}^{2}+e^{-2\lambda}\partial_{r}\partial_{r}\sigma_{\kappa}^{2}+e^{-2\lambda}\frac{2}{r}\partial_{r}\sigma_{\kappa}^{2}-2\sigma_{\kappa}^{4}-4e^{-2\lambda}\mu'\left\langle V\right\rangle\partial_{r}\left\langle V\right\rangle,
\end{split}
\end{align}
\begin{align}
\begin{split}
\square D_{\vec{V}}\sigma^{2}=&16\pi\left(D_{\vec{V}}\sigma^{2}-3\sigma^{2}D_{\vec{V}}\sigma^{2}\right)+4\left( \frac{D_{\vec{V}}\left\langle V\right\rangle-V\partial_{r}\left\langle V\right\rangle}{\left\langle V\right\rangle}+\mu'V\right) ^{3}+8\left( \frac{V}{r}\right) ^{3} +64\pi\frac{1}{r}e^{2\lambda}\left\langle V\right\rangle^{2}V^{3}\\
&+6e^{-2\lambda}\left( \partial_{r}V-4\pi r e^{2\lambda}\left\langle V\right\rangle^{2}V+\lambda'V\right) \left( \partial_{r}\partial_{r}\sigma^{2}+4\pi r e^{4\lambda}VD_{\vec{V}}\sigma^{2}-4\pi re^{4\lambda}V^{2}\partial_{r}\sigma^{2}\right) \\
&+12\left( \frac{D_{\vec{V}}\left\langle V\right\rangle-V\partial_{r}\left\langle V\right\rangle}{\left\langle V\right\rangle}+\mu'V\right)\left( e^{-2\lambda}\mu'\left\langle V\right\rangle-4\pi re^{2\lambda}\left\langle V\right\rangle V^{2}\right) \left( \partial_{r}\left\langle V\right\rangle-4\pi re^{4\lambda}\left\langle V\right\rangle V^{2}\right) \\
&+12\frac{D_{\vec{V}}V-V\partial_{r}\left\langle V\right\rangle}{\left\langle V\right\rangle}\left( \frac{D_{\vec{V}}\left\langle V\right\rangle-V\partial_{r}\left\langle V\right\rangle}{\left\langle V\right\rangle}+\mu'V\right) \left( \partial_{r}\left\langle V\right\rangle-4\pi re^{4\lambda}\left\langle V\right\rangle V^{2}\right)\\
&-12\left( \frac{D_{\vec{V}}V}{\left\langle V\right\rangle}-\frac{V\partial_{r}\left\langle V\right\rangle}{\left\langle V\right\rangle}+e^{-2\lambda}\mu'\left\langle V\right\rangle-4\pi re^{2\lambda}\left\langle V\right\rangle V^{2}\right) \left( \frac{[D_{\vec{V}},\partial_{r}]\sigma^{2}-V\partial_{r}\partial_{r}\sigma^{2}}{\left\langle V\right\rangle}\right) \\
&-12\left( \frac{D_{\vec{V}}V}{\left\langle V\right\rangle}-\frac{V\partial_{r}\left\langle V\right\rangle}{\left\langle V\right\rangle}+e^{-2\lambda}\mu'\left\langle V\right\rangle-4\pi re^{2\lambda}\left\langle V\right\rangle V^{2}\right)\left( -\mu'\frac{1}{\left\langle V\right\rangle}D_{\vec{V}}\sigma^{2}+\mu'\frac{V}{\left\langle V\right\rangle}\partial_{r}\sigma^{2}\right) \\
&-12\left( \frac{D_{\vec{V}}V}{\left\langle V\right\rangle}-\frac{V\partial_{r}\left\langle V\right\rangle}{\left\langle V\right\rangle}+e^{-2\lambda}\mu'\left\langle V\right\rangle-4\pi re^{2\lambda}\left\langle V\right\rangle V^{2}\right)\left( \frac{\partial_{r}D_{\vec{V}}\sigma^{2}}{\left\langle V\right\rangle}+4\pi re^{2\lambda}\left\langle V\right\rangle V\partial_{r}\sigma^{2}\right) \\
&-6\left( \frac{D_{\vec{V}}\left\langle V\right\rangle-V\partial_{r}\left\langle V\right\rangle}{\left\langle V\right\rangle}+\mu'V\right)\left( \frac{D_{\vec{V}}D_{\vec{V}}\sigma^{2}}{\left\langle V\right\rangle^{2}}-\frac{(D_{\vec{V}}\left\langle V\right\rangle-V\partial_{r}\left\langle V\right\rangle)(D_{\vec{V}}\sigma^{2}-V\partial_{r}\sigma^{2})}{\left\langle V\right\rangle^{3}}\right) \\
&-6\left( \frac{D_{\vec{V}}\left\langle V\right\rangle-V\partial_{r}\left\langle V\right\rangle}{\left\langle V\right\rangle}+\mu'V\right)\left( -\frac{2V}{\left\langle V\right\rangle^{2}}\partial_{r}D_{\vec{V}}\sigma^{2}-\frac{(D_{\vec{V}}V)(\partial_{r}\sigma^{2})}{\left\langle V\right\rangle^{2}}+\frac{V(\partial_{r}V)(\partial_{r}\sigma^{2})}{\left\langle V\right\rangle^{2}}\right) \\
&-6\left( \frac{D_{\vec{V}}\left\langle V\right\rangle-V\partial_{r}\left\langle V\right\rangle}{\left\langle V\right\rangle}+\mu'V\right)\left( -\frac{V[D_{\vec{V}},\partial_{r}]\sigma^{2}}{\left\langle V\right\rangle^{2}}+\frac{V^{2}\partial_{r}\partial_{r}\sigma^{2}}{\left\langle V\right\rangle^{2}}-e^{-2\lambda}\mu'\partial_{r}\sigma^{2}\right) \\
&+12\left( \partial_{r}V-4\pi r e^{2\lambda}\left\langle V\right\rangle^{2}V+\lambda'V\right) \left( \frac{D_{\vec{V}}V-V\partial_{r}\left\langle V\right\rangle}{\left\langle V\right\rangle}e^{-2\lambda}\mu'\left\langle V\right\rangle-4\pi re^{2\lambda}\left\langle V\right\rangle V^{2}\right)\\
&\quad \left( \partial_{r}\left\langle V\right\rangle-4\pi re^{4\lambda}\left\langle V\right\rangle V^{2}\right)+4\left( \partial_{r}V-4\pi r e^{2\lambda}\left\langle V\right\rangle^{2}V+\lambda'V\right) ^{3}+\frac{8}{r^{2}}e^{-2\lambda}\mu'V\left\langle V\right\rangle^{2}+8\frac{\lambda'}{r^{2}}V^{3}\\
&-4e^{2\lambda}\left( 48\pi r^{2}e^{4\lambda}\left\langle V\right\rangle^{2}V^{2}-4\pi re^{2\lambda}\frac{D_{\vec{V}}(\left\langle V\right\rangle V)}{\left\langle V\right\rangle}+4\pi re^{2\lambda}\frac{V\partial_{r}(\left\langle V\right\rangle V)}{\left\langle V\right\rangle}+\lambda'\mu'e^{-2\lambda}-(\mu')^{2}e^{-2\lambda}-\mu''e^{-2\lambda}\right)\\
&\quad [ -\left( \frac{D_{\vec{V}}\left\langle V\right\rangle-V\partial_{r}\left\langle V\right\rangle}{\left\langle V\right\rangle}+\mu'V\right)V^{2}+\left( D_{\vec{V}}V-V\partial_{r}V+e^{-2\lambda}\mu'\left\langle V\right\rangle^{2}-4\pi re^{2\lambda}\left\langle V\right\rangle^{2}V\right) V \\
&\quad+e^{-2\lambda}\left\langle V\right\rangle^{2}\left( \partial_{r}V-4\pi r e^{2\lambda}\left\langle V\right\rangle^{2}V+\lambda'V\right)] .
\end{split}
\end{align}
Then we compute commutator identities with the main linear operators.
\begin{lemma}\label{2.1.}
	For any scalar function \(\phi\),
	\begin{align*}
	[D_{\vec{V}},\partial_{r}]\phi=&\mu'D_{\vec{V}}\phi-\mu'V\partial_{r}\phi-\frac{\partial_{r}\left\langle V\right\rangle}{\left\langle V\right\rangle}D_{\vec{V}}\phi+\frac{\partial_{r}\left\langle V\right\rangle}{\left\langle V\right\rangle}V\partial_{r}\phi-(\partial_{r}V)(\partial_{r}\phi)\\
	[D_{\vec{V}},\partial_{r}\partial_{r}]\phi=&2\mu'[D_{\vec{V}},\partial_{r}]\phi-\frac{2\partial_{r}\left\langle V\right\rangle}{\left\langle V\right\rangle}[D_{\vec{V}},\partial_{r}]\phi-2\mu'\partial_{r}D_{\vec{V}}\phi-\frac{2\partial_{r}\left\langle V\right\rangle}{\left\langle V\right\rangle}\partial_{r}D_{\vec{V}}\phi-2\mu'V\partial_{r}\partial_{r}\phi\\
	&+\frac{2\partial_{r}\left\langle V\right\rangle}{\left\langle V\right\rangle}V\partial_{r}\partial_{r}\phi+\mu''D_{\vec{V}}\phi-\mu''V\partial_{r}\phi-(\mu')^{2}D_{\vec{V}}\phi+(\mu')^{2}V\partial_{r}\phi+2\mu'\frac{\partial_{r}\left\langle V\right\rangle}{\left\langle V\right\rangle}D_{\vec{V}}\phi\\
	&-2\mu'\frac{\partial_{r}\left\langle V\right\rangle}{\left\langle V\right\rangle}V\partial_{r}\phi-\frac{\partial_{r}\partial_{r}\left\langle V\right\rangle}{\left\langle V\right\rangle}D_{\vec{V}}\phi+V\frac{\partial_{r}\partial_{r}\left\langle V\right\rangle}{\left\langle V\right\rangle}\partial_{r}\phi-2(\partial_{r}V)(\partial_{r}\partial_{r}\phi)-(\partial_{r}\partial_{r}V)(\partial_{r}\phi)\\
	[D_{\vec{V}},\square]\phi=&-D_{\vec{V}}\left( \frac{1}{\left\langle V\right\rangle^{2}}\right) D_{\vec{V}}D_{\vec{V}}\phi+D_{\vec{V}}\left( \frac{2V}{\left\langle V\right\rangle^{2}}\right) \partial_{r}D_{\vec{V}}\phi+\frac{2V}{\left\langle V\right\rangle^{2}}[D_{\vec{V}},\partial_{r}]D_{\vec{V}}\phi-D_{\vec{V}}\left( \frac{2V}{\left\langle V\right\rangle^{2}}\right) (\partial_{r}V)(\partial_{r}\phi)\\
	&-\frac{2V}{\left\langle V\right\rangle^{2}}(\partial_{r}\phi)[D_{\vec{V}},\partial_{r}]V-\frac{2V}{\left\langle V\right\rangle^{2}}(\partial_{r}V)[D_{\vec{V}},\partial_{r}]\phi+D_{\vec{V}}\left( \frac{1}{\left\langle V\right\rangle^{2}}\right) (D_{\vec{V}}V)(\partial_{r}\phi)+\frac{1}{\left\langle V\right\rangle^{2}}(D_{\vec{V}}V)[D_{\vec{V}},\partial_{r}]\phi\\
	&D_{\vec{V}}\left( \mu'\frac{V}{\left\langle V\right\rangle^{2}}\right) D_{\vec{V}}\phi-D_{\vec{V}}\left( \mu'\frac{V^{2}}{\left\langle V\right\rangle^{2}}\right) \partial_{r}\phi-\mu'\frac{V^{2}}{\left\langle V\right\rangle^{2}}[D_{\vec{V}},\partial_{r}]\phi-\frac{VD_{\vec{V}}\left\langle V\right\rangle}{\left\langle V\right\rangle^{3}}[D_{\vec{V}},\partial_{r}]\phi\\
	&+\frac{2V^{2}\partial_{r}\left\langle V\right\rangle}{\left\langle V\right\rangle^{3}}[D_{\vec{V}},\partial_{r}]\phi+D_{\vec{V}}(4\pi re^{2\lambda}V)D_{\vec{V}}\phi-D_{\vec{V}}(4\pi re^{2\lambda}V^{2})\partial_{r}\phi-4\pi re^{2\lambda}V^{2}[D_{\vec{V}},\partial_{r}]\phi\\
	&+D_{\vec{V}}[e^{-2\lambda}(\mu'-\lambda')]\partial_{r}\phi+e^{-2\lambda}(\mu'-\lambda')[D_{\vec{V}},\partial_{r}]\phi+D_{\vec{V}}\left( e^{-2\lambda}-\frac{V^{2}}{\left\langle V\right\rangle^{2}}\right) \partial_{r}\partial_{r}\phi\\
	&\left( e^{-2\lambda}-\frac{V^{2}}{\left\langle V\right\rangle^{2}}\right)[D_{\vec{V}},\partial_{r}\partial_{r}]\phi+D_{\vec{V}}(e^{-2\lambda})\frac{2}{r}\partial_{r}\phi-e^{-2\lambda}V\frac{1}{r^{2}}\partial_{r}\phi+e^{-2\lambda}\frac{2}{r}[D_{\vec{V}},\partial_{r}]\phi\\
	&-\frac{2V}{\left\langle V\right\rangle^{2}}(\partial_{r}D_{\vec{V}}V)(\partial_{r}\phi)+\frac{1}{\left\langle V\right\rangle^{2}}(D_{\vec{V}}D_{\vec{V}}V)(\partial_{r}\phi)+D_{\vec{V}}\left( \frac{D_{\vec{V}}\left\langle V\right\rangle}{\left\langle V\right\rangle^{3}}\right) D_{\vec{V}}\phi\\
	&-D_{\vec{V}}\left( \frac{VD_{\vec{V}}\left\langle V\right\rangle}{\left\langle V\right\rangle^{3}}\right) \partial_{r}\phi-D_{\vec{V}}\left( \frac{2V\partial_{r}\left\langle V\right\rangle}{\left\langle V\right\rangle^{3}}\right) D_{\vec{V}}\phi+D_{\vec{V}}\left( \frac{2V^{2}\partial_{r}\left\langle V\right\rangle}{\left\langle V\right\rangle^{3}}\right) \partial_{r}\phi\\
	[D_{\vec{V}},D_{\vec{V}}^{2}-&\frac{1}{2}\nabla_{\alpha}\sigma^{2}\nabla^{\alpha}]\phi=\frac{1}{2}D_{\vec{V}}\left( \frac{D_{\vec{V}}\sigma^{2}-V\partial_{r}\sigma^{2}}{\left\langle V\right\rangle^{2}}\right) \left( D_{\vec{V}}\phi-V\partial_{r}\phi\right) -\frac{1}{2}D_{\vec{V}}(e^{-2\lambda}\partial_{r}\sigma^{2})\partial_{r}\phi\\
	&-\frac{1}{2}\frac{D_{\vec{V}}\sigma^{2}-V\partial_{r}\sigma^{2}}{\left\langle V\right\rangle^{2}}\left( (D_{\vec{V}}V)(\partial_{r}\phi)+V[D_{\vec{V}},\partial_{r}]\phi\right) -\frac{1}{2}e^{-2\lambda}\partial_{r}\sigma^{2}[D_{\vec{V}},\partial_{r}]\phi.
	\end{align*}
\end{lemma}
The above identities can be obtained by direct calculation. Then we give an important relation between the background metric and the fluid variable, which is similar to the Tolman-Oppenheimer-Volkov equation in steady state.

\begin{lemma}\label{2.2.}
	The solution of Einstein-Euler system satisfies the following identities:
	\begin{align}\label{2.15}
	e^{-2\lambda}=&1-\frac{8\pi\int_{0}^{r}s^{2}[\rho(s)+e^{2\lambda}V^{2}]ds}{r}\\ \label{2.16}
	\lambda'+\mu'=&4\pi re^{2\lambda}\left( \rho+p+2e^{2\lambda}V^{2}\right) \\ \label{2.17}
	\mu'=&4\pi re^{2\lambda}\left( p+e^{2\lambda}V^{2}+\frac{\int_{0}^{r}s^{2}[\rho(s)+e^{2\lambda}V^{2}]ds}{r^{3}}\right) .
	\end{align}
\end{lemma}
\begin{proof}
	Using the boundary condition at \(r=0\), \eqref{2.15} can be obtained by integrating equation \eqref{field eq00}. Equation \(\eqref{2.16}\) can be proved by adding \(\eqref{field eq00}\) to \eqref{field eq11}. Finally using \(\eqref{2.15}\) to substitute for the term \((e^{-2\lambda}-1)\), we  arrive at the desired identity \(\eqref{2.17}\).
\end{proof}

Let the notation \(\varpi\) represent the perturbed terms including \(V,\varepsilon\) and their derivatives. Applying Lemma \ref{2.1.} and \ref{2.2.} yields higher order equations for the fluid variables which we record in the form of a few lemmas for future reference.
\begin{lemma}\label{2.3.}
	\(D_{\vec{V}}^{k}V\) satisfies
	\begin{align}\label{fk}
	(D_{\vec{V}}^{2}+\frac{1}{2}aD_{n})D_{\vec{V}}^{k}V=f_{k},
	\end{align}
	on \(\partial\calB\), where \(f_{k}\) is a linear combination (coefficients are related to \(\mu',\lambda',e^{2\lambda},\partial_{r}\sigma^{2}\)) of the perturbed terms including the linear terms of \(\partial_{r}D_{\vec{V}}^{j_{1}+1}\sigma^{2},D_{\vec{V}}^{j_{2}}V\)
	and nonlinear terms \(\varpi(\partial_{r}D_{\vec{V}}^{j_{3}+1}\sigma^{2},D_{\vec{V}}^{j_{4}}V,\partial_{r}D_{\vec{V}}^{j_{5}}V)\), with \(j_{1},j_{2}\le k\), \(j_{3},j_{5}\le k-1\), \(j_{4}\le k+1\).
\end{lemma}
\begin{proof}
  Here we just outline the crucial steps. Applying Lemma \(\ref{2.2.}\), we have
	\begin{align*}
	D_{\vec{V}}\mu'=&D_{\vec{V}}\left[ 4\pi r e^{2\lambda}(p+e^{2\lambda}V^{2})\right] +D_{\vec{V}}\left( \frac{e^{2\lambda}-1}{2r}\right) \\
	=&D_{\vec{V}}\left[ 4\pi r e^{2\lambda}(p+e^{2\lambda}V^{2})\right]+\frac{e^{2\lambda}}{r}D_{\vec{V}}\lambda-V\frac{e^{2\lambda}-1}{2r^{2}}\\
	=&\frac{1}{2}\left[ \mu'-4\pi r e^{2\lambda}(p+e^{2\lambda}V^{2})\right]\left[  \frac{D_{\vec{V}}\left\langle V\right\rangle}{\left\langle V\right\rangle}-\frac{\partial_{r}\left\langle V\right\rangle}{\left\langle V\right\rangle}V+\partial_{r}V-4\pi r e^{2\lambda}\left\langle V\right\rangle^{2}V+\left( \mu'+\lambda'\right)V\right]\\
	&+D_{\vec{V}}\left[ 4\pi r e^{2\lambda}(p+e^{2\lambda}V^{2})\right]+\frac{e^{2\lambda}}{r}D_{\vec{V}}\lambda, 
	\end{align*}
	where we have used \(\eqref{mass eq SS}\). Using equation \(\eqref{field eq01}\), we have
	\begin{align*}
	D_{\vec{V}}\lambda=V^{0}\dot{\lambda}+V\lambda'=-4\pi re^{2\lambda}\left\langle V\right\rangle^{2}V+V\lambda'.
	\end{align*}
	Therefore \(D_{\vec{V}}\) applied to \(\mu'\) have the desired forms, as well as \(D_{\vec{V}}\lambda'\). This means that the derivative acting on the background metric \(\mu,\lambda\) can be converted to the fluid variable \(V,\varepsilon\). By the induction and commutator identities, a lengthy but straightforward argument completes the proof.
\end{proof}

Next we record the wave equation for \(D_{\vec{V}}^{k}V\).
\begin{lemma}\label{2.4.}
	\(D_{\vec{V}}^{k}V\) satisfies
	\begin{align}\label{Fk}
	\square D_{\vec{V}}^{k}V=F_{k}+\frac{1}{r}F_{k-1},
	\end{align}
	in \(\calB\), where \(F_{k}\) is a linear combination (coefficients are related to \(\mu',\mu'',\lambda',\lambda'',e^{2\lambda}\), \(\sigma^{2},\partial_{r}\sigma^{2},\partial_{r}\partial_{r}\sigma^{2}\)) of the perturbed terms including the linear terms of \(\partial_{r}D_{\vec{V}}^{j_{1}}V, D_{\vec{V}}^{j_{2}}V,\partial_{r}D_{\vec{V}}^{j_{3}+1}\sigma^{2},D_{\vec{V}}^{j_{4}+1}\sigma^{2}\), with \(j_{1},j_{4}\le k\), \(j_{2}\le k+1\), \(j_{3}\le k-1\)
	and nonlinear terms \(\varpi(\partial_{r}D_{\vec{V}}^{j_{5}}V,\ D_{\vec{V}}^{j_{6}}V,\partial_{r}\partial_{r}D_{\vec{V}}^{j_{7}}V,\partial_{r}D_{\vec{V}}^{j_{8}+1}\sigma^{2},D_{\vec{V}}^{j_{9}+1}\sigma^{2},\partial_{r}\partial_{r}D_{\vec{V}}^{j_{10}+1}\sigma^{2})\), with \(j_{5},j_{9}\le k\), \(j_{7},j_{8}\le k-1\), \(j_{6}\le k+1\),\(j_{10}\le k-2\).
\end{lemma}
The next lemma contains the wave equation for \(D_{\vec{V}}^{k+1}\sigma^{2}\).
\begin{lemma}\label{2.5.}
	\(D_{\vec{V}}^{k+1}\sigma^{2}\) satisfies
	\begin{align}\label{Hk}
	\square D_{\vec{V}}^{k+1}\sigma^{2}=-12e^{-2\lambda}\mu'\partial_{r}D_{\vec{V}}^{k+1}\sigma^{2}+H_{k}+\frac{1}{r}H_{k-1},
	\end{align}
	in \(\calB\), where \(H_{k}\) is a linear combination (coefficients are related to \(\mu',\mu'',\lambda',\lambda'',e^{2\lambda}\), \(\sigma^{2},\partial_{r}\sigma^{2},\partial_{r}\partial_{r}\sigma^{2}\)) of the perturbed terms including the linear terms of \(\partial_{r}D_{\vec{V}}^{j_{1}}V, D_{\vec{V}}^{j_{2}}V,\partial_{r}D_{\vec{V}}^{j_{3}+1}\sigma^{2},D_{\vec{V}}^{j_{4}+1}\sigma^{2}\), with \(j_{1},j_{4}\le k\), \(j_{2}\le k+1\), \(j_{3}\le k-1\)
	and nonlinear terms \(\varpi(\partial_{r}D_{\vec{V}}^{j_{5}}V,\ D_{\vec{V}}^{j_{6}}V,\partial_{r}\partial_{r}D_{\vec{V}}^{j_{7}}V,\partial_{r}D_{\vec{V}}^{j_{8}+1}\sigma^{2},D_{\vec{V}}^{j_{9}+1}\sigma^{2},\partial_{r}\partial_{r}D_{\vec{V}}^{j_{10}+1}\sigma^{2})\), with \(j_{5},j_{8}\le k\), \(j_{7},j_{10}\le k-1\), \(j_{6},j_{9}\le k+1\).
\end{lemma}
  The proof of the above two lemmas is similar to the Lemma \(\ref{2.3.}\) and we omit the details. In the next section, we will develop the nonlinear theory for solutions to the Einstein-Euler equation in the spherically symmetric motion.

	\section{Nonlinear a priori estimate}\label{sec6}
For any function \(\phi\) we introduce the following energies
\begin{align}
\begin{split}
E[\phi,t]:=&\int_{0}^{R}r^{2}|\partial_{t,r}\phi|^{2}dr+R^{2}|D_{\vec{V}}\phi|^{2}(R),\\
\bar{E}[\phi,t]:=&\int_{0}^{R}r^{2}|\partial_{t,r}\phi|^{2}dr,
\end{split}
\end{align}
where \(R\) is the radius of the domain \(\calB(t)\), and the higher-order energies are defined as
\begin{align*}
E_{j}[\phi,t]=E[D_{\vec{V}}^{j}\phi,t],\quad E_{\le k}[\phi,t]=\sum_{j=0}^{k}E_{j}[\phi,t],\quad \bar{E}_{j}[\phi,t]=\bar{E}[D_{\vec{V}}^{j}\phi,t],\quad \bar{E}_{\le k}[\phi,t]=\sum_{j=0}^{k}\bar{E}_{j}[\phi,t].
\end{align*}
For the perturbed solutions ($\varepsilon,V$) to Einstein-Euler equation, we introduce the following unified energy
\begin{align*}
\calE_{l}(t):=&E_{\le l}[V,t]+\bar{E}_{\le l+1}[\varepsilon,t]\\
\scE_{l}(t):=&\sum_{j+k\le l+1} \int_{0}^{R}r^{2}|\partial_{r}^{j}D_{\vec{V}}^{k}V|^{2}dr+\sum_{j+k\le l+1}\int_{0}^{R}r^{2}|\partial_{r}^{j}D_{\vec{V}}^{k+1}\varepsilon|^{2}dr +\calE_{l}(t).
\end{align*}
Our main goal in this section is to prove the following proposition under the smallness assumption.
\begin{proposition}\label{3.1.}
	Suppose (\(V,\varepsilon\)) is a solution to Einstein-Euler system \(\eqref{field eq00}-\eqref{mass eq SS}\) with
	\begin{align}\label{bootstrap}
	\scE_{l}(t)\le C_{l},\quad\quad |\lambda|\le C_{\lambda},\quad\quad \Arrowvert\partial_{y}^{m'}\eta(y,t)\Arrowvert_{L^{2}(\Omega)},\ |\partial_{y}^{m}\eta(y,t)|,\ |\partial_{y}\eta(y,t)-1|,\ |\eta(y,t)-y|\le C_{\eta},
	\end{align}
	for some constants \(0<C_{l},C_{\eta}\ll 1,C_{\lambda}\) and \(l\) sufficiently large satisfying \(1<m\le l-1, l\le m'\le l+1\).
	Then we have
	\begin{align}\label{3.3}
	\scE_{l}(t)\le C_{0}\scE_{l}(0)+\int_{0}^{t}\varrho\scE_{l}(s)+C_{1}\scE_{l-1}(s)+C_{2}\scE_{l}^{\frac{3}{2}}(s)ds,\quad\quad t\in[0,T],
	\end{align}
	for some positive constants \(C_{0},C_{1},C_{2}\) and sufficiently small absolute constant \(\varrho\) to be chosen later.
\end{proposition}
To prove Proposition \ref{3.1.} we need to show that higher order energy \(\calE_{l}(t)\) controls the total energy \(\scE_{l}(t)\). The
result is stated in the following proposition. 
\begin{proposition}\label{3.2.}
	Under the assumptions of Proposition \(\ref{3.1.}\), for any \(t\in[0,T]\), we have
	\begin{align}
	\sum_{p+k\le l+1} \int_{0}^{R}r^{2}|\partial_{r}^{p}D_{\vec{V}}^{k}V|^{2}dr+\sum_{p+k\le l+1}\int_{0}^{R}r^{2}|\partial_{r}^{p}D_{\vec{V}}^{k+1}\varepsilon|^{2}dr\lesssim \calE_{l}(t)+\scE_{l}(0)+\int_{0}^{t}\scE_{l}^{\frac{3}{2}}(s)ds.
	\end{align}
	The implicit coefficient in this estimate depends on the constants of the bootstrap assumption \(\eqref{bootstrap}\).
\end{proposition}
In order to prove Proposition \(\ref{3.2.}\), we first recall standard Hardy inequalities.
\begin{lemma}\label{hardy}
	For a given function \(\phi\), we have
	\begin{align*}
	\int_{0}^{R}\phi^{2}dr\lesssim\int_{0}^{R}r^{2}\left( \phi^{2}+|\partial_{r}\phi|^{2}\right) dr,
	\end{align*}
	and if \(k<1\), the following estimate hold
	\begin{align*}
	\int_{0}^{R}r^{k-2}\left( \phi-\phi(0)\right) ^{2}dr\lesssim \int_{0}^{R}r^{k}|\partial_{r}\phi|^{2}dr.
	\end{align*}
\end{lemma}
The proof of Lemma \(\ref{hardy}\) refer to \cite{KMP}. The next lemma allows us to bound lower order terms in \(L^{\infty}\).
\begin{lemma}
	Under the bootstrap assumption \eqref{bootstrap}, we have
	\begin{align*}
	\Arrowvert\partial_{r}^{a}D_{\vec{V}}^{k}V\Arrowvert_{\infty}+\Arrowvert\partial_{r}^{a}D_{\vec{V}}^{k+1}\varepsilon\Arrowvert_{\infty}\lesssim\scE_{l}^{\frac{1}{2}}(t),\quad\quad a+k\le l-1,\ t\in[0,T].
	\end{align*}
\end{lemma}
\begin{proof}
	This follows from the Sobolev embedding \(H^{1}(0,R)\hookrightarrow L^{\infty}(0,R)\) and standard Hardy inequality.
\end{proof}

The following lemma gives a estimate for the \(L^{2}\) norms.
\begin{lemma}\label{3.5.}
	Under the bootstrap assumption \eqref{bootstrap}, if \(a+k\le l, t\in[0,T]\), then we have
	\begin{align*}
	\int_{0}^{R}r^{2}|\partial_{r}^{a}D_{\vec{V}}^{k}V|^{2}dr+\int_{0}^{R}r^{2}|\partial_{r}^{a}D_{\vec{V}}^{k+1}\varepsilon|^{2}dr\lesssim \scE_{l}(0)+\int_{0}^{t}\scE_{l}(s)ds.
	\end{align*}
\end{lemma}
\begin{proof}
	We rcall the Lagrangian parametrization \(\eta(y,t)\), that is
	\begin{align*}
	\partial_{t}\eta=\frac{V}{V^{0}}\circ\eta\ \ \ \  \text{with} \ \ \ \ \eta(y,0)=\eta_{0}(y).
	\end{align*}
	For any function \(\phi\), we have
	\begin{align}
	\frac{d\phi\left( \eta(y,t),t\right) }{dt}=\frac{D_{\vec{V}}\phi}{V^{0}}(\eta(y,t),t).
	\end{align}
	Multiplying both sides of the above equation by \(y^{2}\phi\) and integrating over \([0,R_{\kappa}]\times [0,t]\), we get
	\begin{align*}
	\frac{1}{2}\int_{0}^{R_{\kappa}}y^{2}\left| \phi(\eta(y,t),t)\right| ^{2}dy-	\frac{1}{2}\int_{0}^{R_{\kappa}}y^{2}\left| \phi(\eta(y,0),0)\right| ^{2}dy=\int_{0}^{t}\int_{0}^{R_{\kappa}}y^{2}\phi(\eta(y,s),s)\frac{D_{\vec{V}}\phi}{V^{0}}(\eta(y,s),s)dyds.
	\end{align*}
	Under the bootstrap assumption \eqref{bootstrap}, we can bound the Jacobian of the Lagrangian coordinate transformation from \([0,R_{\kappa}]\) to \([0,R(t)]\) using the fundamental theorem of calculus. Therefore we give
	\begin{align}\label{22}
	\int_{0}^{R}r^{2}\left| \phi(r,t)\right| ^{2}dr\lesssim\int_{0}^{R}r^{2}\left| \phi(r,0)\right| ^{2}dr+\int_{0}^{t}\int_{0}^{R}r^{2}\left( |D_{\vec{V}}\phi(r,s)|^{2}+|\phi(r,s)|^{2}\right) drds.
	\end{align}
	We apply this estimate to \(\phi=\partial_{r}^{a}D_{\vec{V}}^{k}V\) as well \(\phi=\partial_{r}^{a}D_{\vec{V}}^{k+1}\varepsilon\). Then as long as \(a+k\le l\), the right-hand side of \(\eqref{22}\) is bounded by
	\begin{align*}
     \scE_{l}(0)+\int_{0}^{t}\scE_{l}(s)ds,
	\end{align*}
	which completes the proof of Lemma \(\ref{3.5.}\).
\end{proof}

The next lemma will be used to control \(\Arrowvert\partial_{r}\partial_{r}\phi\Arrowvert_{L^{2}(\calB(t))}\) in terms of \(\Arrowvert\square\phi\Arrowvert_{L^{2}(\calB(t))}\) and \(\Arrowvert\partial_{r}D_{\vec{V}}\phi\Arrowvert_{L^{2}(\calB(t))}\).
\begin{lemma}\label{3.6.}
	Let \(A\) be defined as \(A:=\frac{a^{rr}}{r^{2}}\partial_{r}(r^{2}\partial_{r}\phi)\), where \(a^{rr}=e^{-2\lambda}-\frac{V^{2}}{\left\langle V\right\rangle^{2}}\). Then for any \(\phi\),
	\begin{align}\label{11}
	\begin{split}
	A\phi=&\square\phi+\frac{1}{\left\langle V\right\rangle^{2}}D_{\vec{V}}D_{\vec{V}}\phi-\frac{2V}{\left\langle V\right\rangle^{2}}\partial_{r}D_{\vec{V}}\phi+\frac{2V}{\left\langle V\right\rangle^{2}}(\partial_{r}V)(\partial_{r}\phi)-\frac{1}{\left\langle V\right\rangle^{2}}(D_{\vec{V}}V)(\partial_{r}\phi)-\mu'\frac{V}{\left\langle V\right\rangle^{2}}D_{\vec{V}}\phi\\
	&+\mu'\frac{V^{2}}{\left\langle V\right\rangle^{2}}\partial_{r}\phi-\frac{D_{\vec{V}}\left\langle V\right\rangle}{\left\langle V\right\rangle^{3}}D_{\vec{V}}\phi+\frac{VD_{\vec{V}}\left\langle V\right\rangle}{\left\langle V\right\rangle^{3}}\partial_{r}\phi+\frac{2V\partial_{r}\left\langle V\right\rangle}{\left\langle V\right\rangle^{3}}D_{\vec{V}}\phi-\frac{2V^{2}\partial_{r}\left\langle V\right\rangle}{\left\langle V\right\rangle^{3}}\partial_{r}\phi-4\pi re^{2\lambda}VD_{\vec{V}}\phi\\
	&+4\pi re^{2\lambda}V^{2}\partial_{r}\phi+e^{-2\lambda}(\lambda'-\mu')\partial_{r}\phi-\frac{V^{2}}{\left\langle V\right\rangle^{2}}\frac{2}{r}\partial_{r}\phi\\
	=:&\square\phi+R_{\phi},
	\end{split}
	\end{align}
	and there is constant \(c_{0}>0\) such that
	\begin{align*}
	\inf_{0\le t\le T} a^{rr}>c_{0}.
	\end{align*}
\end{lemma}
\begin{proof}
	The proof of \(\eqref{11}\) is a direct calculation using the definition of the wave operator
	\begin{align*}
	\square\phi=-e^{-2\mu}\partial_{t}\partial_{t}\phi+e^{-2\lambda}\partial_{r}\partial_{r}\phi+e^{-2\mu}(\dot{\mu}-\dot{\lambda})\partial_{t}\phi+e^{-2\lambda}\left( \mu'-\lambda'+\frac{2}{r}\right) \partial_{r}\phi,
	\end{align*}
	and we have
	\begin{align*}
	a^{rr}=e^{-2\lambda}-\frac{V^{2}}{\left\langle V\right\rangle^{2}}=e^{-2\lambda}-\frac{V^{2}}{\sigma^{2}+e^{2\lambda}V^{2}}=\frac{e^{-2\lambda}\sigma^{2}}{\sigma^{2}+e^{2\lambda}V^{2}}>0.
	\end{align*}
	By continuty, the proof of Lemma \ref{3.6.} is completed.
\end{proof}

To use the above Lemma, we will use standard Hardy inequalities, and derive the following elliptic estimate which is adapted to our energy spaces with  weight \(r^{2}\).
\begin{lemma}\label{3.7.}
	Under the bootstrap assumption \eqref{bootstrap}, we have the elliptic estimate
	\begin{align*}
		\int_{0}^{R}r^{2}|\partial_{r}\partial_{r}\phi|^{2}dr\lesssim \int_{0}^{R}r^{2}|\partial_{r}\phi|^{2}dr+\int_{0}^{R}r^{2}\left| A\phi\right| ^{2}dr.
	\end{align*}
\end{lemma}
\begin{proof}
	We start by exploiting the elliptic structure of the integral
	\begin{align}\label{111}
	\begin{split}
	\int_{0}^{R}(A\phi)\left[ \partial_{r}(r^{2}\partial_{r}\phi)\right] dr=&\int_{0}^{R}\frac{a^{rr}}{r^{2}}\left| \partial_{r}(r^{2}\partial_{r}\phi)\right| ^{2}dr\\
	=&\int_{0}^{R}\left( a^{rr}r^{2}\left| \partial_{r}\partial_{r}\phi\right| ^{2}+4a^{rr}r(\partial_{r}\partial_{r}\phi)(\partial_{r}\phi)+4a^{rr}\left| \partial_{r}\phi\right| ^{2}\right) dr.
	\end{split}
	\end{align}
	Using integration by parts, we have
	\begin{align}\label{1111}
	\int_{0}^{R}4a^{rr}r(\partial_{r}\partial_{r}\phi)(\partial_{r}\phi)dr=2a^{rr}(R)R\left| \partial_{r}\phi\right| ^{2}(R)-\int_{0}^{R}2a^{rr}\left| \partial_{r}\phi\right| ^{2}dr-\int_{0}^{R}2(\partial_{r}a^{rr})r\left| \partial_{r}\phi\right| ^{2}dr.
	\end{align}
	Combined with \(\eqref{111}\) and \eqref{1111}, we get
	\begin{align}\label{222}
	\begin{split}
	&\int_{0}^{R}\left( a^{rr}r^{2}\left| \partial_{r}\partial_{r}\phi\right| ^{2}+2a^{rr}\left| \partial_{r}\phi\right| ^{2}\right) dr+2a^{rr}(R)R\left| \partial_{r}\phi\right| ^{2}(R)\\
	=&\int_{0}^{R}(A\phi)\left[ \partial_{r}(r^{2}\partial_{r}\phi)\right] dr+\int_{0}^{R}2(\partial_{r}a^{rr})r\left| \partial_{r}\phi\right| ^{2}dr.
	\end{split}
	\end{align}
	Since \(a^{rr}>c_{0}>0\) by Lemma \ref{3.6.}, the left-hand side of \(\eqref{222}\) controls
	\begin{align*}
	\int_{0}^{R}r^{2}\left| \partial_{r}\partial_{r}\phi\right| ^{2}dr+\int_{0}^{R}\left| \partial_{r}\phi\right| ^{2}dr.
	\end{align*}
	Notice that \(r^{-1}(\partial_{r}a^{rr})\) is bounded under the bootstrap assumption \eqref{bootstrap}. Therefore using Cauchy-Schwarz and Hardy inequalities, the right-hand side of \eqref{222} can be bounded by
	\begin{align*}
	\int_{0}^{R}r^{2}|\partial_{r}\phi|^{2}dr+\int_{0}^{R}r^{2}\left| A\phi\right| ^{2}dr,
	\end{align*}
	which completes the proof of Lemma \ref{3.7.}.
\end{proof}

Notice that Einstein-Euler system \(\eqref{field eq00}-\eqref{mass eq SS}\) does not contain the time derivative of \(\mu\), so we need some additional estimates. 
\begin{lemma}\label{mu lemma}
	Under the bootstrap assumption \eqref{bootstrap}, if \(k\le l\), then we have
	\begin{align*}
	|\mu(t,r)|\lesssim 1\quad\quad
	\text{and}\quad\quad |\partial_{t}^{k}\mu(r,t)|\lesssim \scE_{l}^{\frac{1}{2}}(t).
	\end{align*}
\end{lemma}
\begin{proof}
	For the first estimate, we consider the equation \(\eqref{field eq11}\) by recasting the following form
	\begin{align}\label{mu}
	\mu'(t,r)=\frac{e^{2\lambda}-1}{2r}+4\pi re^{2\lambda}(p+e^{2\lambda}V^{2}).
	\end{align}
	By integrating \eqref{mu} and using the boundary condition \(\eqref{boundary condition field}\), we get
	\begin{align*}
	\mu(t,r)=-\int_{r}^{\infty}\left( \frac{e^{2\lambda}-1}{2s}+4\pi se^{2\lambda}(p+e^{2\lambda}V^{2})\right) ds.
	\end{align*}
	When \(r\ge R\), \(p(t,r)=V(t,r)=0\) and therefore we have
	\begin{align}\label{mu1}
	\mu(t,r)=-\int_{R}^{\infty}\left( \frac{e^{2\lambda}-1}{2s}\right) ds-\int_{r}^{R}\left(\frac{e^{2\lambda}-1}{2s}+ 4\pi se^{2\lambda}(p+e^{2\lambda}V^{2})\right) ds.
	\end{align}
	As mentioned in Lemma \ref{2.4..}, we can show that the first term of the right-hand side of \(\eqref{mu1}\) is bounded. We omit the details. Combined with \(\eqref{bootstrap}\) and \(\eqref{mu1}\), we prove the first estimate in Lemma \(\ref{mu lemma}\). Then we differentiate the time derivative with respect to equation \(\eqref{mu}\) to obtain
	\begin{align*}
	(\dot{\mu})'(t,r)=-4\pi e^{\mu+4\lambda}\left\langle V\right\rangle V+4\pi r e^{\mu}\frac{D_{\vec{V}}(e^{2\lambda}p+e^{4\lambda}V^{2})}{\left\langle V\right\rangle}-4\pi r e^{\mu}\frac{V\partial_{r}(e^{2\lambda}p+e^{4\lambda}V^{2})}{\left\langle V\right\rangle}.
	\end{align*}
	According to \eqref{boundary condition field} and \(\eqref{bootstrap}\), we integrate the above equation to get
	\begin{align*}
	|\dot{\mu}(r,t)|\lesssim \scE_{l}^{\frac{1}{2}}(t)\lesssim 1,
	\end{align*}
	where we have used the Sobolev embedding inequality and the Hardy inequality. For \(1< k\le l\), we only need to differentiate the \(k\) time derivative about equation \(\eqref{mu}\) and use a similar method as before, which completes the proof of Lemma \ref{mu lemma}.
\end{proof}

\begin{proof}[Proof of Proposition \(\eqref{3.2.}\)]
	we will use the induction argument on the order \(p\) of \(\partial_{r}^{p}\), starting with the fluid quantities \(V\) and \(\varepsilon\). When \(p=0,1\) the estimates for \(\partial_{r}^{p}D_{\vec{V}}^{k}V\) and \(\partial_{r}^{p}D_{\vec{V}}^{k+1}\varepsilon\) follow from the definition of the energies. Now we assume that the estimate holds for orders less or
	equal to \(1\le p\le l\), that is
	\begin{align}\label{p}
		\sum_{q\le p}\sum_{q+k\le l+1} \int_{0}^{R}r^{2}|\partial_{r}^{q}D_{\vec{V}}^{k}V|^{2}dr+\sum_{q\le p}\sum_{q+k\le l+1}\int_{0}^{R}r^{2}|\partial_{r}^{q}D_{\vec{V}}^{k+1}\varepsilon|^{2}dr\lesssim \calE_{l}(t)+\scE_{l}(0)+\int_{0}^{t}\scE_{l}^{\frac{3}{2}}(s)ds,
	\end{align}
	and prove the estimates for \(p+1\), that is,
	\begin{align}\label{p+1}
		\sum_{k\le l-p} \int_{0}^{R}r^{2}|\partial_{r}^{p+1}D_{\vec{V}}^{k}V|^{2}dr+\sum_{k\le l-p}\int_{0}^{R}r^{2}|\partial_{r}^{p+1}D_{\vec{V}}^{k+1}\varepsilon|^{2}dr\lesssim \calE_{l}(t)+\scE_{l}(0)+\int_{0}^{t}\scE_{l}^{\frac{3}{2}}(s)ds.
	\end{align}
	We start with the estimate for \(\Arrowvert\partial_{r}^{p+1}D_{\vec{V}}^{k+1}\varepsilon\Arrowvert^{2}_{L^{2}(\calB(t))}\). Applying Lemma \(\ref{3.7.}\) to \(\phi:=\partial_{r}^{p-1}D_{\vec{V}}^{k+1}\varepsilon\) we need to estimate \(\Arrowvert A\partial_{r}^{p-1}D_{\vec{V}}^{k+1}\varepsilon\Arrowvert^{2}_{L^{2}(\calB(t))}\). Using the notation of Lemma \ref{2.5.}, we have
	\begin{align}\label{ss}
	\begin{split}
	\int_{0}^{R}r^{2}|\partial_{r}^{p+1}D_{\vec{V}}^{k+1}\varepsilon|^{2}dr\lesssim&\int_{0}^{R}r^{2}|\partial_{r}^{p}D_{\vec{V}}^{k+1}\varepsilon|^{2}dr+\int_{0}^{R}r^{2}\left| [\square,\partial_{r}^{p-1}]D_{\vec{V}}^{k+1}\varepsilon\right| ^{2}dr+
	\int_{0}^{R}r^{2}|\partial_{r}(\mu'\partial_{r}D_{\vec{V}}^{k+1}\varepsilon)|^{2}dr\\
	&+\int_{0}^{R}r^{2}\left| R_{\partial_{r}^{p-1}D_{\vec{V}}\varepsilon}\right| ^{2}dr+\int_{0}^{R}r^{2}\left| \partial_{r}^{p-1}\left( H_{k}+\frac{1}{r}H_{k-1}\right) \right| ^{2}dr.
		\end{split}
	\end{align}
	Except for the last term on the right-hand side of \(\eqref{ss}\), all the terms are bounded by the right-hand side of \(\eqref{p+1}\) using the induction hypothesis \(\eqref{p}\). In view of Lemma \ref{2.5.} and the Hardy inequalities, the remainder terms in the last term on the right-hand side of \eqref{ss} are
	\begin{align*}
	\int_{0}^{R}r^{2}|\varpi\cdot\partial_{r}^{p+1}D_{\vec{V}}^{k}\varepsilon|^{2}dr,\quad \int_{0}^{R}r^{2}|\varpi\cdot\partial_{r}^{p+1}D_{\vec{V}}^{k-1}V|^{2}dr,
	\end{align*}
	where \(\Arrowvert\varpi\Arrowvert_{\infty}\lesssim\scE_{l}^{\frac{1}{2}}(t)\), and the all other term appearing in \eqref{ss} can be bounded by the right-hand side of \(\eqref{p+1}\) using the induction hypothesis \(\eqref{p}\). For the two top order terms above, since \(p+1+k-1\le l\) we can use Lemma \ref{3.5.} to bound these terms by the right-hand side of \(\eqref{p+1}\) as well. The estimate for \(\Arrowvert\partial_{r}^{p+1}D_{\vec{V}}^{k}V\Arrowvert^{2}_{L^{2}(\calB(t))}\) is similar, so we omit the details. Based on the argument inductively, we finally obtain
	\begin{align*}
	\sum_{p+k\le l+1} \int_{0}^{R}r^{2}|\partial_{r}^{p}D_{\vec{V}}^{k}V|^{2}dr+\sum_{p+k\le l+1}\int_{0}^{R}r^{2}|\partial_{r}^{p}D_{\vec{V}}^{k+1}\varepsilon|^{2}dr\lesssim \calE_{l}(t)+\scE_{l}(0)+\int_{0}^{t}\scE_{l}^{\frac{3}{2}}(s)ds,
	\end{align*}
	which completes the proof of Proposition \ref{3.2.}.
\end{proof}
The following we develop the necessary nonlinear energy estimates for sufficiently small perturbed solutions \((V,\varepsilon)\) to the Einstein-Euler system \eqref{field eq00}-\eqref{mass eq SS}. We start by recording a general multiplier identity for the wave equation. Let \(Q=Q^{\alpha}\nabla_{\alpha}\) be an arbitrary first order multiplier. Then a direct calculation shows that
\begin{align}\label{key id}
(\square\phi)(Q\phi)=\nabla_{\alpha}((Q\phi)(\nabla^{\alpha}\phi)-\frac{1}{2}Q^{\alpha}(\nabla_{\beta}\phi)(\nabla^{\beta}\phi))+\frac{1}{2}(\nabla_{\alpha}Q^{\alpha})(\nabla_{\beta}\phi)(\nabla^{\beta}\phi)-(\nabla^{\alpha}Q^{\beta})(\nabla_{\alpha}\phi)(\nabla_{\beta}\phi).
\end{align}
The next energy estimate is used for \(D_{\vec{V}}^{k}V\).
\begin{lemma}\label{3.8.}
	Under the bootstrap assumption \eqref{bootstrap}, for any \(k\le l\) we have
	\begin{align}\label{estimate1}
	\begin{split}
	E[D_{\vec{V}}^{k}V,t]
	\lesssim&\scE_{l}(0)+\int_{0}^{t}\int_{0}^{R}r^{2}\left( F_{k}+\frac{1}{r}F_{k-1}\right) \left( D_{\vec{V}}^{k+1}V\right) drds+\int_{0}^{t}R^{2}\left( f_{k}\right) \left( D_{\vec{V}}^{k+1}V\right) (R)ds\\
	&+\int_{0}^{t}\int_{0}^{R}r^{2}|\partial_{r}D_{\vec{V}}^{k}V|^{2}drds+\int_{0}^{t}\scE_{l}^{\frac{3}{2}}(s)ds.
	\end{split}
	\end{align}
\end{lemma}
\begin{proof}
	Multiplying the equation \eqref{fk} by \(a^{-1}D_{\vec{V}}^{k+1}V\) we get
	\begin{align*}
	\frac{1}{2}D_{\vec{V}}\left( \frac{1}{a}(D_{\vec{V}}^{k+1}V)^{2}\right) +\frac{1}{2}(D_{n}D_{\vec{V}}^{k}V)(D_{\vec{V}}^{k+1}V)=\frac{1}{a}f_{k}D_{\vec{V}}^{k+1}V-\frac{1}{2a^{2}}(D_{\vec{V}}a)(D_{\vec{V}}^{k+1}V)^{2},
	\end{align*}
	which upon integration over \(\partial\calB=\cup_{t\in[0,T]}\partial\calB(t)\) gives
	\begin{align}\label{key 1}
	\begin{split}
	&\int_{\partial\calB(t)}\frac{V^{0}}{a}(D_{\vec{V}}^{k+1}V)^{2}dS_{g}+\int_{0}^{t}\int_{\partial\calB(\tau)}(D_{n}D_{\vec{V}}^{k}V)(D_{\vec{V}}^{k+1}V)dS_{g}d\tau\\
	=&\int_{\partial\calB(0)}\frac{V^{0}}{a}(D_{\vec{V}}^{k+1}V)^{2}dS_{g}+\int_{0}^{t}\int_{\partial\calB(\tau)}\frac{2}{a}f_{k}D_{\vec{V}}^{k+1}VdS_{g}d\tau\\
	&-\int_{0}^{t}\int_{\partial\calB(\tau)}\frac{1}{a^{2}}(D_{\vec{V}}a)(D_{\vec{V}}^{k+1}V)^{2}dS_{g}d\tau+\int_{0}^{t}\int_{\partial\calB(\tau)}\frac{\div\mkern-17.5mu\slash\ \ V}{a}(D_{\vec{V}}^{k+1}V)^{2}dS_{g}d\tau,
	\end{split}
	\end{align}
	where \(\div\mkern-17.5mu\slash\ \) denotes the divergence operator on \(\partial\calB\). To treat the second term on the left, we integrate \(\eqref{key id}\) with \(Q=V,\phi=D_{\vec{V}}^{k}V\) over \(\cup_{t\in[0,T]}\calB(t)\). Using the fact
	that \(V\) is tangent to \(\partial\calB\), we get
	\begin{align}\label{key 2}
	\begin{split}
	&\int_{\calB(t)}\left( (D_{\vec{V}}^{k+1}V)(\nabla_{0}D_{\vec{V}}^{k}V)+\frac{V^{0}}{2}(\nabla^{\alpha}D_{\vec{V}}^{k}V)(\nabla_{\alpha}D_{\vec{V}}^{k}V)\right) dx_{g}-\int_{0}^{t}\int_{\partial\calB(\tau)}(D_{\vec{V}}^{k+1}V)(D_{n}D_{\vec{V}}^{k}V)dS_{g}d\tau\\
	=&\int_{\calB(0)}\left( (D_{\vec{V}}^{k+1}V)(\nabla_{0}D_{\vec{V}}^{k}V)+\frac{V^{0}}{2}(\nabla^{\alpha}D_{\vec{V}}^{k}V)(\nabla_{\alpha}D_{\vec{V}}^{k}V)\right) dx_{g}-\int_{0}^{t}\int_{\calB(\tau)}(\square D_{\vec{V}}^{k}V)(D_{\vec{V}}^{k+1}V)dx_{g}d\tau\\
	&\int_{0}^{t}\int_{\calB(\tau)}\frac{1}{2}(\nabla_{\alpha}V^{\alpha})(\nabla^{\beta}D_{\vec{V}}^{k}V)(\nabla_{\beta}D_{\vec{V}}^{k}V)dx_{g}d\tau-\int_{0}^{t}\int_{\calB(\tau)}(\nabla^{\alpha}V^{\beta})(\nabla_{\alpha}D_{\vec{V}}^{k}V)(\nabla_{\beta}D_{\vec{V}}^{k}V)dx_{g}d\tau.
	\end{split}
	\end{align}
	By adding \(\eqref{key 2}\) to \(\eqref{key 1}\), we have
	\begin{align}\label{key 3}
	\begin{split}
	&\int_{\calB(t)}\left( (D_{\vec{V}}^{k+1}V)(\nabla_{0}D_{\vec{V}}^{k}V)+\frac{V^{0}}{2}(\nabla^{\alpha}D_{\vec{V}}^{k}V)(\nabla_{\alpha}D_{\vec{V}}^{k}V)\right) dx_{g}+\int_{\partial\calB(t)}\frac{V^{0}}{a}(D_{\vec{V}}^{k+1}V)^{2}dS_{g}\\
	=&\int_{\calB(0)}\left( (D_{\vec{V}}^{k+1}V)(\nabla_{0}D_{\vec{V}}^{k}V)+\frac{V^{0}}{2}(\nabla^{\alpha}D_{\vec{V}}^{k}V)(\nabla_{\alpha}D_{\vec{V}}^{k}V)\right) dx_{g}+\int_{\partial\calB(0)}\frac{V^{0}}{a}(D_{\vec{V}}^{k+1}V)^{2}dS_{g}\\
	&-\int_{0}^{t}\int_{\calB(\tau)}\left( F_{k}+\frac{1}{r}F_{k-1}\right) (D_{\vec{V}}^{k+1}V)dx_{g}d\tau-\int_{0}^{t}\int_{\calB(\tau)}(\nabla^{\alpha}V^{\beta})(\nabla_{\alpha}D_{\vec{V}}^{k}V)(\nabla_{\beta}D_{\vec{V}}^{k}V)dx_{g}d\tau\\
	&+\int_{0}^{t}\int_{\partial\calB(\tau)}\frac{2}{a}f_{k}D_{\vec{V}}^{k+1}VdS_{g}d\tau-\int_{0}^{t}\int_{\partial\calB(\tau)}\frac{1}{a^{2}}(D_{\vec{V}}a)(D_{\vec{V}}^{k+1}V)^{2}dS_{g}d\tau+\int_{0}^{t}\int_{\partial\calB(\tau)}\frac{\div\mkern-17.5mu\slash\ \ V}{a}(D_{\vec{V}}^{k+1}V)^{2}dS_{g}d\tau,
	\end{split}
	\end{align}	
	where we have used \(\eqref{re mass eq}\) and \(\eqref{Fk}\). Because the vectorfield \(V\) is timelike and future-directed, according to the positivity of energy-momentum tensor, the first term on the left-hand side of \eqref{key 3} satisfies
	\begin{align*}
	(D_{\vec{V}}^{k+1}V)(\nabla_{0}D_{\vec{V}}^{k}V)+\frac{V^{0}}{2}(\nabla^{\alpha}D_{\vec{V}}^{k}V)(\nabla_{\alpha}D_{\vec{V}}^{k}V)\gtrsim|\partial_{t,r}D_{\vec{V}}^{k}V|^{2}.
	\end{align*}
	For the steady-state solution to \eqref{1steady field eq00}-\eqref{1steady momentum eq}, we have the following \textit{Taylor sign condition}
	\begin{align}\label{taylor sign condition}
	\nabla_{\calN}\sigma_{\kappa}^{2}\le -a_{0}< 0, \quad \text{on}\quad \partial\calB,\quad\quad \text{where}\quad \nabla_{\calN}=\calN^{\alpha}\partial_{\alpha}.
	\end{align}
	Since the solution we shall construct in this paper is a small perturbation of the steady-state solution, the condition \eqref{taylor sign condition} holds also for the perturbed solution. Therefore \(a\) must be positive and the left-hand side of \eqref{key 3} controls
	\begin{align*}
	\int_{\calB(t)}|\partial_{t,r}D_{\vec{V}}^{k}V|^{2}dx_{g}+\int_{\partial\calB(t)}|D_{\vec{V}}^{k+1}V|^{2}dS_{g}\gtrsim	E[D_{\vec{V}}^{k}V,t],
	\end{align*}
	which completes the proof of Lemma \ref{3.8.} through some direct calculations.
\end{proof}

The following lemma is the main energy identity for \(D_{\vec{V}}^{k+1}\varepsilon\).
\begin{lemma}\label{3.9.}
 Suppose \(Q=V-\nu n\) with \(\nu>0\) sufficiently small, then the following energy identity holds
	\begin{align}\label{estimate2}
	\begin{split}
	&\int_{\calB(t)}\left[ -(QD_{\vec{V}}^{k+1}\varepsilon)(\nabla^{0}D_{\vec{V}}^{k+1}\varepsilon)+\frac{Q^{0}}{2}(\nabla_{\alpha}D_{\vec{V}}^{k+1}\varepsilon)(\nabla^{\alpha}D_{\vec{V}}^{k+1}\varepsilon)\right] dx_{g}+\int_{0}^{t}\int_{\partial\calB(\tau)}\frac{1}{2}\nu(nD_{\vec{V}}^{k+1}\varepsilon)^{2}dS_{g}d\tau\\
	=&\int_{\calB(0)}\left[ -(QD_{\vec{V}}^{k+1}\varepsilon)(\nabla^{0}D_{\vec{V}}^{k+1}\varepsilon)+\frac{Q^{0}}{2}(\nabla_{\alpha}D_{\vec{V}}^{k+1}\varepsilon)(\nabla^{\alpha}D_{\vec{V}}^{k+1}\varepsilon)\right] dx_{g}-\int_{0}^{t}\int_{\calB(\tau)}\left( H_{k}+\frac{1}{r}H_{k-1}\right) (QD_{\vec{V}}^{k+1}\varepsilon)dx_{g}d\tau\\
	&+\int_{0}^{t}\int_{\calB(\tau)}\left[ -(\nabla^{\alpha}Q^{\beta})(\nabla_{\alpha}D_{\vec{V}}^{k+1}\varepsilon)(\nabla_{\beta}D_{\vec{V}}^{k+1}\varepsilon)-2e^{-2\lambda}\mu'(\partial_{r}D_{\vec{V}}^{k+1}\varepsilon)(QD_{\vec{V}}^{k+1}\varepsilon)\right] dx_{g}d\tau\\
	&+\int_{0}^{t}\int_{\calB(\tau)}\frac{1}{2}(\nabla_{\alpha}Q^{\alpha})(\nabla_{\beta}D_{\vec{V}}^{k+1}\varepsilon)(\nabla^{\beta}D_{\vec{V}}^{k+1}\varepsilon)dx_{g}d\tau+\int_{0}^{t}\int_{\calB(\tau)}14e^{-2\lambda}\mu'(\partial_{r}D_{\vec{V}}^{k+1}\varepsilon)(QD_{\vec{V}}^{k+1}\varepsilon)dx_{g}d\tau.
	\end{split}
	\end{align}
\end{lemma}
\begin{proof}
	Integrating \(\eqref{key id}\) with \(\phi=D_{\vec{V}}^{k+1}\varepsilon\) over \(\cup_{t\in[0,T]}\calB(t)\), we get
	\begin{align}\label{key 4}
	\begin{split}
	&\int_{\calB(t)}\left[ -(QD_{\vec{V}}^{k+1}\varepsilon)(\nabla^{0}D_{\vec{V}}^{k+1}\varepsilon)+\frac{Q^{0}}{2}(\nabla_{\alpha}D_{\vec{V}}^{k+1}\varepsilon)(\nabla^{\alpha}D_{\vec{V}}^{k+1}\varepsilon)\right] dx_{g}\\
	&-\int_{0}^{t}\int_{\partial\calB(\tau)}n_{\alpha}\left( (QD_{\vec{V}}^{k+1}\varepsilon)(\nabla^{\alpha}D_{\vec{V}}^{k+1}\varepsilon)-\frac{1}{2}Q^{\alpha}(\nabla_{\alpha}D_{\vec{V}}^{k+1}\varepsilon)(\nabla^{\alpha}D_{\vec{V}}^{k+1}\varepsilon)\right) dS_{g}d\tau\\
	=&\int_{\calB(0)}\left[ -(QD_{\vec{V}}^{k+1}\varepsilon)(\nabla^{0}D_{\vec{V}}^{k+1}\varepsilon)+\frac{Q^{0}}{2}(\nabla_{\alpha}D_{\vec{V}}^{k+1}\varepsilon)(\nabla^{\alpha}D_{\vec{V}}^{k+1}\varepsilon)\right] dx_{g}-\int_{0}^{t}\int_{\calB(\tau)}(\square D_{\vec{V}}^{k+1}\varepsilon)(QD_{\vec{V}}^{k+1}\varepsilon)dx_{g}d\tau\\
	&+\int_{0}^{t}\int_{\calB(\tau)}\frac{1}{2}(\nabla_{\alpha}Q^{\alpha})(\nabla_{\beta}D_{\vec{V}}^{k+1}\varepsilon)(\nabla^{\beta}D_{\vec{V}}^{k+1}\varepsilon)dx_{g}d\tau-\int_{0}^{t}\int_{\calB(\tau)}(\nabla^{\alpha}Q^{\beta})(\nabla_{\alpha}D_{\vec{V}}^{k+1}\varepsilon)(\nabla_{\beta}D_{\vec{V}}^{k+1}\varepsilon)dx_{g}d\tau.
	\end{split}
	\end{align}
	Since \(D_{\vec{V}}^{k+1}\varepsilon\) is constant on \(\partial\calB\), we have
	\begin{align*}
	(\nabla_{\alpha}D_{\vec{V}}^{k+1}\varepsilon)(\nabla^{\alpha}D_{\vec{V}}^{k+1}\varepsilon)=(nD_{\vec{V}}^{k+1}\varepsilon)^{2},
	\end{align*}
	and
	\begin{align*}
	n_{\alpha}\left( (QD_{\vec{V}}^{k+1}\varepsilon)(\nabla^{\alpha}D_{\vec{V}}^{k+1}\varepsilon)-\frac{1}{2}Q^{\alpha}(\nabla_{\alpha}D_{\vec{V}}^{k+1}\varepsilon)(\nabla^{\alpha}D_{\vec{V}}^{k+1}\varepsilon)\right)=-\frac{1}{2}\nu(nD_{\vec{V}}^{k+1}\varepsilon)^{2}
	\end{align*}
	on \(\partial\calB\). Using the notation of Lemma \ref{2.5.}, we arrive at \eqref{estimate2}.
\end{proof}

Since \(Q\) is future-directed timelike, the first term on the left-hand side of \(\eqref{estimate2}\) satisfies
\begin{align*}
\int_{0}^{R}r^{2}|\partial_{t,r}D_{\vec{V}}^{k+1}\varepsilon|^{2}dr\lesssim\int_{\calB(t)}\left[ -(QD_{\vec{V}}^{k+1}\varepsilon)(\nabla^{0}D_{\vec{V}}^{k+1}\varepsilon)+\frac{Q^{0}}{2}(\nabla_{\alpha}D_{\vec{V}}^{k+1}\varepsilon)(\nabla^{\alpha}D_{\vec{V}}^{k+1}\varepsilon)\right] dx_{g}.
\end{align*}
Therefore the left-hand side of the energy identity \(\eqref{estimate2}\) controls
\begin{align*}
\bar{E}[D_{\vec{V}}^{k+1}\varepsilon,t]+\int_{0}^{t}R^{2}|\partial_{t,r}D_{\vec{V}}^{k+1}\varepsilon|^{2}(R)ds.
\end{align*}
When \(k=l\) the last term on the right-hand side of \(\eqref{estimate2}\) is equivalent to
\begin{align*}
\int_{0}^{t}\bar{E}[D_{\vec{V}}^{l+1}\varepsilon,s]ds,
\end{align*}
and cannot be controlled by the right-hand side of \(\eqref{3.3}\). In order to close the energy estimates, we introduce conformal metric \(\tilde{g}\) defined as
\begin{align*}
\tilde{g}=-e^{2\mu(t,r)}dt^{2}+e^{2\lambda(t,r)}dr^{2}+e^{14\mu}r^{2}(d\theta^{2}+\sin^{2}\theta d\varphi^{2}).
\end{align*}
Using the notation \(\tilde{\square}\) to represent the wave operator with respect to metric \(\tilde{g}\), a straightforward calculation shows that 
\begin{align*}
\tilde{\square}D_{\vec{V}}^{l+1}\varepsilon=\square D_{\vec{V}}^{l+1}\varepsilon+14e^{-2\lambda}\mu'\partial_{r}D_{\vec{V}}^{l+1}\varepsilon+ \textrm{non-linear terms},
\end{align*}
and this makes the last term on the right-hand side of \(\eqref{estimate2}\) can be eliminated, which is recorded in the following lemma.
\begin{lemma}\label{3.10.}
	Suppose \(Q=V-\nu n\) with \(\nu>0\) sufficiently small, then the following energy identity holds
		\begin{align}\label{estimate3}
	\begin{split}
	&\int_{\calB(t)}\left[ -(QD_{\vec{V}}^{l+1}\varepsilon)(\tilde{\nabla}^{0}D_{\vec{V}}^{l+1}\varepsilon)+\frac{Q^{0}}{2}(\tilde{\nabla}_{\alpha}D_{\vec{V}}^{l+1}\varepsilon)(\tilde{\nabla}^{\alpha}D_{\vec{V}}^{l+1}\varepsilon)\right] dx_{\tilde{g}}+\int_{0}^{t}\int_{\partial\calB(\tau)}\frac{1}{2}\nu(nD_{\vec{V}}^{l+1}\varepsilon)^{2}dS_{\tilde{g}}d\tau\\
	=&\int_{\calB(0)}\left[ -(QD_{\vec{V}}^{l+1}\varepsilon)(\tilde{\nabla}^{0}D_{\vec{V}}^{l+1}\varepsilon)+\frac{Q^{0}}{2}(\tilde{\nabla}_{\alpha}D_{\vec{V}}^{l+1}\varepsilon)(\tilde{\nabla}^{\alpha}D_{\vec{V}}^{l+1}\varepsilon)\right] dx_{\tilde{g}}-\int_{0}^{t}\int_{\calB(\tau)}\left( \tilde{\square}D_{\vec{V}}^{l+1}\varepsilon\right) (QD_{\vec{V}}^{l+1}\varepsilon)dx_{\tilde{g}}d\tau\\
	&+\int_{0}^{t}\int_{\calB(\tau)}\left[ -(\tilde{\nabla}^{\alpha}Q^{\beta})(\tilde{\nabla}_{\alpha}D_{\vec{V}}^{l+1}\varepsilon)(\tilde{\nabla}_{\beta}D_{\vec{V}}^{l+1}\varepsilon)-2e^{-2\lambda}\mu'(\partial_{r}D_{\vec{V}}^{l+1}\varepsilon)(QD_{\vec{V}}^{l+1}\varepsilon)\right] dx_{\tilde{g}}d\tau\\
	&+\int_{0}^{t}\int_{\calB(\tau)}\frac{1}{2}(\tilde{\nabla}_{\alpha}Q^{\alpha})(\tilde{\nabla}_{\beta}D_{\vec{V}}^{l+1}\varepsilon)(\tilde{\nabla}^{\beta}D_{\vec{V}}^{l+1}\varepsilon)dx_{\tilde{g}}d\tau+\int_{0}^{t}\int_{\calB(\tau)} 2e^{-2\lambda}\mu'\left( \partial_{r}D_{\vec{V}}^{l+1}\varepsilon\right) (QD_{\vec{V}}^{l+1}\varepsilon)dx_{\tilde{g}}d\tau,
	\end{split}
	\end{align}
	where the notation \(\tilde{\nabla}\) denotes covariant derivative with respect to \(\tilde{g}\).
\end{lemma}
The proof of the above lemma is similar to the Lemma \(\ref{3.9.}\) and we omit the details. Then we turn to the proof of Proposition \(\ref{3.1.}\). 
\begin{proof}[Proof of Proposition \(\eqref{3.1.}\)]
	According to Proposition \(\ref{3.2.}\), we only need to show that \(\calE_{l}(t)\) is bounded by the right-hand side of \(\eqref{3.3}\), that is, 
	\begin{align}\label{3.12}
	E_{\le l}[V,t]+\bar{E}_{\le l+1}[\varepsilon,t]\le C_{0}\scE_{l}(0)+\int_{0}^{t}\varrho\scE_{l}(s)+C_{1}\scE_{l-1}(s)+C_{2}\scE_{l}^{\frac{3}{2}}(s)ds.
	\end{align}
	{\bf Step\ 1:} First we show that
	\begin{align}
	\bar{E}_{\le l+1}[\varepsilon,t]\le C_{0}\scE_{l}(0)+\int_{0}^{t}\varrho\scE_{l}(s)+C_{1}\scE_{l-1}(s)+C_{2}\scE_{l}^{\frac{3}{2}}(s)ds.
	\end{align}
	As we have discussed before, the method is to apply energy estimate \(\eqref{estimate2}\) for \(k\le l-1\) and energy estimate \(\eqref{estimate3}\) for \(k=l\). Using the Lemma \(\ref{3.9.}\), \(\bar{E}_{\le l}[\varepsilon,t]\) is directly bounded by the right-hand of \(\eqref{3.3}\). Therefore we only consider the most difficult case, that is \(k=l\). In view of Lemma \(\ref{3.10.}\), we have
	\begin{align}\label{3.13}
	\begin{split}
	\bar{E}[D_{\vec{V}}^{l+1}\varepsilon,t]\lesssim&\scE_{l}(0)+\left| \int_{0}^{t}\int_{\calB(\tau)}\left[ \tilde{\square}D_{\vec{V}}^{l+1}\varepsilon-2e^{-2\lambda}\mu'\left( \partial_{r}D_{\vec{V}}^{l+1}\varepsilon\right) \right] (QD_{\vec{V}}^{l+1}\varepsilon)dx_{\tilde{g}}d\tau\right| \\
		&+\left| \int_{0}^{t}\int_{\calB(\tau)}\left[ -(\tilde{\nabla}^{\alpha}Q^{\beta})(\tilde{\nabla}_{\alpha}D_{\vec{V}}^{l+1}\varepsilon)(\tilde{\nabla}_{\beta}D_{\vec{V}}^{l+1}\varepsilon)-2e^{-2\lambda}\mu'(\partial_{r}D_{\vec{V}}^{l+1}\varepsilon)(QD_{\vec{V}}^{l+1}\varepsilon)\right] dx_{\tilde{g}}d\tau\right| \\
	&+\left| \int_{0}^{t}\int_{\calB(\tau)}\frac{1}{2}(\tilde{\nabla}_{\alpha}Q^{\alpha})(\tilde{\nabla}_{\beta}D_{\vec{V}}^{l+1}\varepsilon)(\tilde{\nabla}^{\beta}D_{\vec{V}}^{l+1}\varepsilon)dx_{\tilde{g}}d\tau\right|. 
	\end{split}
	\end{align}
	By using Cauchy-Schwarz inequality and taking \(\nu\) small enough, a straight forward calculation shows that the last two terms on the right-hand of \(\eqref{3.13}\) can be bounded by the right-hand of \(\eqref{3.3}\). For the main term
	\begin{align*}
	\left| \int_{0}^{t}\int_{\calB(\tau)}\left[ \tilde{\square}D_{\vec{V}}^{l+1}\varepsilon-2e^{-2\lambda}\mu'\left( \partial_{r}D_{\vec{V}}^{l+1}\varepsilon\right) \right] (QD_{\vec{V}}^{l+1}\varepsilon)dx_{\tilde{g}}d\tau\right|,
	\end{align*}
	in fact, we need to consider the contribution 
	\begin{align*}
	\int_{0}^{t}\int_{0}^{R}r^{2}\left|H_{l}+\frac{1}{r}H_{l-1} \right|^{2}drd\tau.
	\end{align*}
	In view of Lemma \ref{2.5.} and the Hardy inequalities, all the terms can be bounded by the right-hand of \(\eqref{3.3}\) except for the top order linear terms
	\begin{align}\label{3.15}
	\int_{0}^{t}\int_{0}^{R}r^{2}\left| \partial_{r}D_{\vec{V}}^{l}V\right| ^{2}drd\tau+\int_{0}^{t}\int_{0}^{R}r^{2}\left| D_{\vec{V}}^{l+1}V\right| ^{2}drd\tau.
	\end{align}
	Notice that \(l\) is large enough, therefore using  \(\eqref{momentum eq SS}\) and computing commutators the contribution can be bounded by
	\begin{align*}
	\sum_{j+k\le l}\int_{0}^{t}\int_{0}^{R}r^{2}\left| \partial_{r}^{j}D_{\vec{V}}^{k}V\right| ^{2}drd\tau+\sum_{j+k\le l+1}\int_{0}^{t}\int_{0}^{R}r^{2}\left| \partial_{r}^{j}D_{\vec{V}}^{k}\varepsilon\right| ^{2}drd\tau+\int_{0}^{t}\scE_{l}^{\frac{3}{2}}(\tau)d\tau,
	\end{align*}
	which completes the proof of \(\eqref{3.12}\).\\
	{\bf Step\ 2:} Here we prove that
	\begin{align}
	E_{\le l}[V,t]\le C_{0}\scE_{l}(0)+\int_{0}^{t}\varrho\scE_{l}(s)+C_{1}\scE_{l-1}(s)+C_{2}\scE_{l}^{\frac{3}{2}}(s)ds.
	\end{align}
	Similar to the previous step, we consider the most difficult part \(E[D_{\vec{V}}^{l}V,t]\). In view of Lemma \(\ref{3.8.}\), we need to control
	\begin{align}\label{3.17}
	\int_{0}^{t}\int_{0}^{R}r^{2}\left( F_{l}+\frac{1}{r}F_{l-1}\right) \left( D_{\vec{V}}^{l+1}V\right) drd\tau+\int_{0}^{t}R^{2}\left( f_{l}\right) \left( D_{\vec{V}}^{l+1}V\right) (R)d\tau.
	\end{align}
	Using Cauchy-Schwarz and Hardy inequalities, the contribution of the first term of \(\eqref{3.17}\) can be bounded by
	\begin{align*}
		\sum_{j+k\le l+1}\int_{0}^{t}\int_{0}^{R}r^{2}\left| \partial_{r}^{j}D_{\vec{V}}^{k}V\right| ^{2}drd\tau+\sum_{j+k\le l+1}\int_{0}^{t}\int_{0}^{R}r^{2}\left| \partial_{r}^{j}D_{\vec{V}}^{k}\varepsilon\right| ^{2}drd\tau+\int_{0}^{t}\scE_{l}^{\frac{3}{2}}(\tau)d\tau,
	\end{align*}
	therefore it can be bounded by the right-hand of \(\eqref{3.3}\) in the same way as in the treatment of \(\eqref{3.15}\). For the second term of \(\eqref{3.17}\), by Lemma \(\ref{2.3.}\) and Cauchy-Schwarz we need to control
	\begin{align}\label{3.18}
	\sum_{k=0}^{l+1}\int_{0}^{t}R^{2}\left| \partial_{r}D_{\vec{V}}^{k}\varepsilon\right|^{2}(R)d\tau+\sum_{k=0}^{l}\int_{0}^{t}R^{2}\left| D_{\vec{V}}^{k}V\right|^{2}(R)d\tau+\int_{0}^{t}\scE_{l}^{\frac{3}{2}}(\tau)d\tau.
	\end{align}
	In view of Lemma \(\ref{3.9.}\) and \(\eqref{3.10.}\), the first term of \(\eqref{3.18}\) can be bounded by the right-hand of \(\eqref{3.3}\) in the same way as before. Combined with the previous discussion, we have completed the proof of Proposition \(\ref{3.1.}\).
\end{proof}

\section{Nonlinear instability}\label{sec7}
In this section, we describe how to prove the nonlinear instability. Based on the sharp nonlinear estimates in the previous section, we are now ready to show a bootstrap argument that allows us to control the growth of \(\scE_{l}(t)\) in terms of the linear growth rate \(\sqrt{-\nu_{\ast}}\). The idea comes from a bootstrap framework \cite{GHS}, which is key passage from linear instability to nonlinear instability. The following we make some necessary preparations. Define the Lagrangian variables
\begin{align*}
&\bar{V}=V\circ\eta \ \ (Lagrangian\ velocity),\\
&\bar{\varepsilon}=\varepsilon\circ\eta\ \ (Lagrangian\ enthalpy),
\end{align*}
and the corresponding energy
\begin{align*}
\bar{\scE_{l}}(t):=\sum_{j+k\le l+1}\left(  \int_{0}^{R_{\kappa}}y^{2}|\partial_{y}^{j}\partial_{t}^{k}\bar{V}|^{2}dy\right) ^{\frac{1}{2}}+\sum_{j+k\le l+1}\left( \int_{0}^{R_{\kappa}}y^{2}|\partial_{y}^{j}\partial_{t}^{k+1}\bar{\varepsilon}|^{2}dy\right) ^{\frac{1}{2}}+\sum_{k\le l+1}\left( R_{\kappa}|\partial_{t}^{k}\bar{V}|(R_{\kappa})\right) .
\end{align*}
The bootstrap assumptions \(\eqref{bootstrap}\) implies that the norm in the Lagrangian coordinate system is
equivalent to the norm in the Cartesian coordinate system, so Proposition \ref{3.1.} still holds in the Lagrangian coordinate system, that is
\begin{align}\label{estimate lagrangian}
\bar{\scE_{l}}(t)\le C_{0}\bar{\scE_{l}}(0)+\int_{0}^{t}\varrho\bar{\scE_{l}}(s)+C_{1}\bar{\scE}_{l-1}(s)+C_{2}\bar{\scE_{l}}^{\frac{3}{2}}(s)ds,
\end{align}
for some positive constants \(C_{0},C_{1},C_{2}\) and sufficiently small absolute constant \(\varrho\) to be chosen later. In order to control the growth rate of the linear term \(\bar{\scE}_{l-1}\), we apply Sobolev interpolation inequality and roughly illustrate the idea.
\begin{align}\label{6.2}
\begin{split}
\left( \int_{0}^{t}\bar{\scE}_{l-1}(s)ds\right) ^{2}
\simeq\sum_{j+k\le l}\int_{0}^{t}\int_{0}^{R_{\kappa}}y^{2}|\partial_{y}^{j}\partial_{t}^{k}\bar{V}|^{2}dyds+\sum_{j+k\le l}\int_{0}^{t}\int_{0}^{R_{\kappa}}y^{2}|\partial_{y}^{j}\partial_{t}^{k+1}\bar{\varepsilon}|^{2}dyds+\sum_{k\le l}\int_{0}^{t}R_{\kappa}^{2}|\partial_{t}^{k}\bar{V}|^{2}(R_{\kappa})ds.
\end{split}
\end{align}
The first term on the right-hand side of \eqref{6.2} can be written as
\begin{align}\label{interpolation}
\begin{split}
&\sum_{j+k\le l}\int_{0}^{t}\int_{0}^{R_{\kappa}}y^{2}|\partial_{y}^{j}\partial_{t}^{k}\bar{V}|^{2}dyds\\
\le&\theta_{1}\sum_{j+k\le l+1}\int_{0}^{t}\int_{0}^{R_{\kappa}}y^{2}|\partial_{y}^{j}\partial_{t}^{k}\bar{V}|^{2}dyds+C_{\theta_{1}}\sum_{j\le l}\int_{0}^{t}\int_{0}^{R_{\kappa}}y^{2}|\partial_{y}^{j}\bar{V}|^{2}dyds\\
\le&\theta_{1}\sum_{j+k\le l+1}\int_{0}^{t}\int_{0}^{R_{\kappa}}y^{2}|\partial_{y}^{j}\partial_{t}^{k}\bar{V}|^{2}dyds+\theta_{2}C_{\theta_{1}}\sum_{j\le l+1}\int_{0}^{t}\int_{0}^{R_{\kappa}}y^{2}|\partial_{y}^{j}\bar{V}|^{2}dyds+C_{\theta_{1}}C_{\theta_{2}}\int_{0}^{t}\int_{0}^{R_{\kappa}}y^{2}|\bar{V}|^{2}dyds,
\end{split}
\end{align}
where we apply Sobolev interpolation inequality to time norm and space norm with weight \(y^{2}\) respectively. Since \(\theta_{1},\theta_{2}\) can be small enough, we have
\begin{align*}
\left( \sum_{j+k\le l}\int_{0}^{t}\int_{0}^{R_{\kappa}}y^{2}|\partial_{y}^{j}\partial_{t}^{k}\bar{V}|^{2}dyds\right) ^{\frac{1}{2}}\le \int_{0}^{t}\varrho\bar{\scE}_{l}(s)+C_{\varrho}\Arrowvert\bar{V}\Arrowvert_{L^{2}(\Omega)}ds,
\end{align*}
for sufficiently small absolute constant \(\varrho\) to be chosen later. The remainder terms on the the right-hand side of \eqref{6.2} are bounded by the trace theorem and the same way. Combined with the previous discussion, we get
\begin{align}\label{6.4}
\bar{\scE}_{l}(t)\le C_{0}\bar{\scE_{l}}(0)+\int_{0}^{t}\varrho\bar{\scE_{l}}(s)+C_{2}\bar{\scE_{l}}^{\frac{3}{2}}(s)+C_{\varrho}\Arrowvert\bar{V},\bar{\varepsilon}\Arrowvert_{L^{2}(\Omega)}ds,
\end{align}
for some positive constants \(C_{0},C_{2},C_{\varrho}\) and sufficiently small absolute constant \(\varrho\) to be chosen later. In the statement of the following proposition, for any given \(\delta>0\) and \(\theta_{0}>\delta\), we define
\begin{align}\label{T delta}
T^{\delta}\equiv\frac{1}{\sqrt{-\nu_{\ast}}}\ln\frac{\theta_{0}}{\delta},
\end{align}
where \(\sqrt{-\nu_{\ast}}\) is the fastest linear growth mode.
\begin{proposition}\label{7.1.}
	Assume the stars have large enough central densities. For any sufficiently small \(\delta>0\), there exists a family of initial data \((\bar{V}^{\delta}(0),\bar{\varepsilon}^{\delta}(0))=\delta(\bar{V}_{0},\bar{\varepsilon}_{0})\) such that the perturbed solutions \((\bar{V}^{\delta}(t),\bar{\varepsilon}^{\delta}(t))\) to the Einstein-Euler system with equation of state \eqref{state eq} for \(t\in[0,T^{\delta}]\) satisfy
	\begin{align*}
	\Arrowvert\bar{V}^{\delta}(T^{\delta}),\bar{\varepsilon}^{\delta}(T^{\delta})\Arrowvert_{L^{2}(\Omega)}\ge \tau_{0}>0,
	\end{align*}
	where \(\tau_{0}\) is independent of \(\delta\).
\end{proposition}
\begin{remark}
	The above result shows that no matter how small the amplitude of initial perturbed data is taken to be, we can find a solution such that the corresponding energy escapes at a time \(T^{\delta}\): the system is nonlinear instability.
\end{remark}

\begin{proof}
Formally the perturbation solution \(\Upsilon=(\bar{V},\bar{\varepsilon})\) of Einstein-Euler system satisfies
\begin{align}\label{duhamel}
\dot{\Upsilon}=L\Upsilon+N\Upsilon,
\end{align}
where \(L\) is the linearized operator and \(N\) is the nonlinear operator. According to the Sobolev embedding inequalities, we have
\begin{align*}
\Arrowvert N\Upsilon\Arrowvert_{L^{2}(\Omega)}\le C_{N}\bar{\scE_{l}}^{2},
\end{align*}
for some positive constant \(C_{N}\). In Section \ref{sec4} we prove that the existence of the fastest linear growth mode \(\sqrt{-\nu_{\ast}}\) of the linearized operator around steady states with large central density. This means that \(e^{tL}\) generates a strongly continuous semigroup on \(L^{2}(\Omega)\) such that
\begin{align*}
\Arrowvert e^{tL}\Arrowvert_{(L^{2}(\Omega),L^{2}(\Omega))}\le C_{L}e^{\sqrt{-\nu_{\ast}}t},
\end{align*}
for some positive constant \(C_{L}\). Consider a family of initial data \(\Upsilon^{\delta}(0)=\delta \Upsilon_{0}\) with \(\Arrowvert\Upsilon_{0}\Arrowvert_{L^{2}(\Omega)}=1\), where \(\Upsilon_{0}\) is the eigenfunction corresponding to the fastest growing modes of the linearized system satisfying
\begin{align}\label{5555}
\Arrowvert e^{tL}\Upsilon_{0}\Arrowvert_{L^{2}(\Omega)}=e^{\sqrt{-\nu_{\ast}}t}.
\end{align}
In order to finish instability argument, we need the following two additional bootstrap assumptions
\begin{align}\label{bootstrap2}
\bar{\scE_{l}}^{\frac{1}{2}}(t)\le\frac{\sqrt{-\nu_{\ast}}}{4C_{2}},\quad\quad \Arrowvert\Upsilon(t)\Arrowvert_{L^{2}(\Omega)}\le 2\delta e^{\sqrt{-\nu_{\ast}}t}.
\end{align}
Then we apply \(\eqref{6.4}\) with \(\varrho=\frac{\sqrt{-\nu_{\ast}}}{4}\), \(C_{\varrho}=C_{\nu_{\ast}}\) and obtain
\begin{align}
\begin{split}
\bar{\scE_{l}}(t)\le& C_{0}\bar{\scE_{l}}(0)+\int_{0}^{t}\frac{\sqrt{-\nu_{\ast}}}{4}\bar{\scE_{l}}(s)+C_{2}\bar{\scE_{l}}^{\frac{3}{2}}(s)+C_{\nu_{\ast}}\Arrowvert\Upsilon(s)\Arrowvert_{L^{2}(\Omega)}ds\\
\le&C_{0}\bar{\scE_{l}}(0)+\int_{0}^{t}\frac{\sqrt{-\nu_{\ast}}}{2}\bar{\scE_{l}}(s)+2C_{\nu_{\ast}}\delta e^{\sqrt{-\nu_{\ast}}s}ds\\
\le&C_{0}\bar{\scE_{l}}(0)+\frac{2C_{\nu_{\ast}}}{\sqrt{-\nu_{\ast}}}\delta e^{\sqrt{-\nu_{\ast}}t}+\int_{0}^{t}\frac{\sqrt{-\nu_{\ast}}}{2}\bar{\scE_{l}}(s)ds.
\end{split}
\end{align}
It follows from the Gronwall lemma that
\begin{align}\label{6.9}
\begin{split}
\bar{\scE_{l}}(t)\le& C_{0}\bar{\scE_{l}}(0)+\frac{2C_{\nu_{\ast}}}{\sqrt{-\nu_{\ast}}}\delta e^{\sqrt{-\nu_{\ast}}t}+\int_{0}^{t}\frac{\sqrt{-\nu_{\ast}}}{2}\left( C_{0}\bar{\scE_{l}}(0)+\frac{2C_{\nu_{\ast}}}{\sqrt{-\nu_{\ast}}}\delta e^{\sqrt{-\nu_{\ast}}s}\right) e^{\frac{\sqrt{-\nu_{\ast}}}{2}(t-s)}ds\\
\le& C_{0}\bar{\scE_{l}}(0)e^{\frac{\sqrt{-\nu_{\ast}}}{2}t}+\frac{2C_{\nu_{\ast}}}{\sqrt{-\nu_{\ast}}}\delta e^{\sqrt{-\nu_{\ast}}t}\\
\le&\tilde{C}(C_{\Upsilon_{0}}+1)\delta e^{\sqrt{-\nu_{\ast}}t},
\end{split}
\end{align}
where \(\tilde{C}=\max(\frac{2C_{\nu_{\ast}}}{\sqrt{-\nu_{\ast}}},C_{0})\) and \(C_{\Upsilon_{0}}>0\). Applying the Duhamel principle to \eqref{duhamel} yields
\begin{align}\label{dul1}
\begin{split}
\Arrowvert\Upsilon(t)-\delta e^{Lt}\Upsilon_{0}\Arrowvert_{L^{2}(\Omega)}=&\left| \left| \int_{0}^{t}e^{L(t-s)}N(\Upsilon(s))ds\right| \right| _{L^{2}(\Omega)}\\
\le& C_{L}\int_{0}^{t}e^{\sqrt{-\nu_{\ast}}(t-s)}\Arrowvert N(\Upsilon(s))\Arrowvert_{L^{2}(\Omega)}ds\\
\le&C_{L}C_{N}\int_{0}^{t}e^{\sqrt{-\nu_{\ast}}(t-s)}\bar{\scE_{l}}^{2}(s)ds\\
\le& 2\tilde{C}^{2}C_{L}C_{N}[C_{\Upsilon_{0}}^{2}+1]\delta^{2}\int_{0}^{t}e^{\sqrt{-\nu_{\ast}}(t-s)}e^{2\sqrt{-\nu_{\ast}}s}ds\\
\le& \bar{C}[C_{\Upsilon_{0}}^{2}+1]\delta^{2}e^{2\sqrt{-\nu_{\ast}}t},
\end{split}
\end{align}
where \(\bar{C}=\frac{2\tilde{C}^{2}C_{L}C_{N}}{\sqrt{-\nu_{\ast}}}\). Using \eqref{5555} and \eqref{dul1}, at the escape time \(t=T^{\delta}\) we have
\begin{align}\label{6.11}
\begin{split}
\Arrowvert\Upsilon(T^{\delta})\Arrowvert_{L^{2}(\Omega)}\ge&\left| \delta e^{\sqrt{-\nu_{\ast}}T^{\delta}}-\bar{C}[C_{\Upsilon_{0}}^{2}+1]\delta^{2}e^{2\sqrt{-\nu_{\ast}}T^{\delta}}\right| \\
\ge&\left| \theta_{0}-\bar{C}[C_{\Upsilon_{0}}^{2}+1]\theta_{0}^{2}\right| .
\end{split}
\end{align}
Then we close the bootstrap assumptions \eqref{bootstrap} and \eqref{bootstrap2} by taking an appropriate constant \(\theta_{0}\). Set
\begin{align*}
\theta_{0}=\min\left( \frac{1}{2\bar{C}[C_{\Upsilon_{0}}^{2}+1]},\frac{\left(\frac{\sqrt{-\nu_{\ast}}}{4C_{2}} \right) ^{2}}{4\tilde{C}[C_{\Upsilon_{0}}+1]},\frac{C_{l}^{\frac{1}{2}}}{2\tilde{C}[C_{\Upsilon_{0}}+1]},\frac{C_{\eta}}{2\tilde{C}[C_{\Upsilon_{0}}+1]}   \right).
\end{align*}
We have from \eqref{6.9} and \eqref{dul1},
\begin{align*}
\bar{\scE_{l}}^{\frac{1}{2}}(t)\le&\tilde{C}(C_{\Upsilon_{0}}+1)\delta e^{\sqrt{-\nu_{\ast}}T^{\delta}}\le \frac{\sqrt{-\nu_{\ast}}}{8C_{2}}\\
\Arrowvert\Upsilon(t)\Arrowvert_{L^{2}(\Omega)}\le&\left( 1+\bar{C}[C_{\Upsilon_{0}}^{2}+1]\delta e^{\sqrt{-\nu_{\ast}}T^{\delta}}\right) \delta e^{\sqrt{-\nu_{\ast}}t}\le\frac{3}{2}\delta e^{\sqrt{-\nu_{\ast}}t},
\end{align*}
which close the bootstrap assumptions \eqref{bootstrap2}. Applying the fundamental theorem of calculus, the bootstrap assumptions \eqref{bootstrap} can be closed in the same way as in the treatment of \eqref{bootstrap2}. We omit the details. Combined with the previous estimate \eqref{6.11}, we shall see at the escape time \(t=T^{\delta}\),
\begin{align*}
	\Arrowvert\Upsilon(T^{\delta})\Arrowvert_{L^{2}(\Omega)}\ge \left| \theta_{0}-\bar{C}[C_{\Upsilon_{0}}^{2}+1]\theta_{0}^{2}\right|\ge \frac{1}{2}\theta_{0},
\end{align*}
where \(\theta_{0}\) is independent of \(\delta\), which completes the proof of Proposition \ref{7.1.}.
\end{proof}

\bibliographystyle{plain}
\bibliography{ee}

\begin{thebibliography}{10}

\bibitem{Andreasson}
H.~Andr\'easson.
\newblock Sharp bounds on $2m/r$ of general spherically symmetric static
  objects.
\newblock {\em J. Differ. Equ.}, 245:2243--2266, 2008.

\bibitem{Ch-video}
D.~Christodoulou.
\newblock https://www.youtube.com/watch?v={F}{F}{M}{S}teogq40.

\bibitem{Ch-hp1}
D.~Christodoulou.
\newblock Self-gravitating relativistic fluids: a two-phase model.
\newblock {\em Arch. Rational Mech. Anal.}, 130(4):343--400, 1995.

\bibitem{FoSc1}
G.~Fournodavlos and V.~Schlue.
\newblock On ``hard stars'' in general relativity.
\newblock {\em Ann. Henri Poincar\'{e}}, 20(7):2135--2172, 2019.

\bibitem{F-P}
B.~Friedman and V.~R. Pandharipande.
\newblock Hot and cold, nuclear and neutron matter.
\newblock {\em Nuclear Physics}, A361:502--520, 1981.

\bibitem{GinLin1}
D.~{Ginsberg} and H.~{Lindblad}.
\newblock On the local well-posedness for the relativistic {E}uler equations
  for a liquid body.
\newblock {\em Ann. PDE}, 9(2):Paper No. 23, 120, 2023.

\bibitem{GHS}
Y.~Guo, C.~Hallstrom, and D.~Spirn.
\newblock Dynamics near unstable, interfacial fluids.
\newblock {\em Comm. Math. Phys.}, 270(3):635--689, 2007.

\bibitem{HL}
M.~Had\v{z}i\'{c} and Z.~Lin.
\newblock Turning point principle for relativistic stars.
\newblock {\em Comm. Math. Phys.}, 387(2):729--759, 2021.

\bibitem{HLR}
M.~Had\v{z}i\'{c}, Z.~Lin, and G.~Rein.
\newblock Stability and instability of self-gravitating relativistic matter
  distributions.
\newblock {\em Arch. Ration. Mech. Anal.}, 241(1):1--89, 2021.

\bibitem{HM}
Z.~Hao and S.~Miao.
\newblock On nonlinear instability of liquid lane-emden stars.
\newblock {\em arXiv:2304.06217}, 2023.

\bibitem{Jang}
J.~Jang.
\newblock Nonlinear instability theory of {L}ane-{E}mden stars.
\newblock {\em Comm. Pure Appl. Math.}, 67(9):1418--1465, 2014.

\bibitem{KMP}
A.~Kufner, L.~Maligranda, and L.-E. Persson.
\newblock The hardy inequality.
\newblock {\em Vydavatelsky Servis, Plzen}, 2007.

\bibitem{Lam}
K.~M. Lam.
\newblock Linear stability of liquid lane-emden stars.
\newblock {\em arXiv:2208.06736}, 2022.

\bibitem{Lich-book}
A.~Lichnerowicz.
\newblock {\em Relativistic Hydrodynamics and Magnetohydrodynamics: Lectures on
  the Existence of Solutions}.
\newblock W. A. Benjamin, Inc., 1967.

\bibitem{MS}
S.~Miao and S.~Shahshahani.
\newblock Well-posedness for the free boundary hard phase model in general
  relativity.
\newblock {\em Adv. Math.}, 443:Paper No. 109614, 66, 2024.

\bibitem{MSW1}
S.~Miao, S.~Shahshahani, and S.~Wu.
\newblock Well-posedness of free boundary hard phase fluids in {M}inkowski
  background and their {N}ewtonian limit.
\newblock {\em Camb. J. Math.}, 9(2):269--350, 2021.

\bibitem{Oliynyk3}
T.~A. Oliynyk.
\newblock {Dynamical relativistic liquid bodies}.
\newblock {\em arXiv e-prints}, page arXiv:1907.08192, Jul 2019.

\bibitem{Rez-book}
L.~Rezzolla and O.~Zanotti.
\newblock {\em Relativistic Hydrodynamics}.
\newblock Oxford University Press, 2013.

\bibitem{Walecka}
J.~D. Walecka.
\newblock A theory of highly condensed matter.
\newblock {\em Ann. Phys.}, 83:491--569, 1974.

\bibitem{Zeldovich}
Ya.~B. Zel'dovich.
\newblock The equation of state at ultrahigh densities and its relativistic
  limitations.
\newblock {\em J. Exp. Theor. Phys. (U.S.S.R.)}, 41:1609--1615, 1961 (English
  translation in Soy. Phys.- JETP 14, 1143-1147 (1962)).

\end{thebibliography}

\bigskip

\centerline{\scshape Zeming Hao}
\smallskip
{\footnotesize
	\centerline{School of Mathematics and Statistics, Wuhan University}
	\centerline{Wuhan, Hubei 430072, China}
	\centerline{\email{2021202010062@whu.edu.cn}}
}

\medskip

\centerline{\scshape Shuang Miao}
\smallskip
{\footnotesize
	\centerline{School of Mathematics and Statistics, Wuhan University}
	\centerline{Wuhan, Hubei 430072, China}
	\centerline{\email{shuang.m@whu.edu.cn}}
}

\end{document}